\documentclass[10pt]{amsart}

\usepackage{amsthm, amssymb}
\usepackage[leqno]{amsmath}
\usepackage{caption}
\usepackage{subcaption}
\usepackage{relsize}
\usepackage{booktabs}
\usepackage{hyperref}
\usepackage{float}

\newcommand\MyLBrace[2]{%
	\left.\rule{0pt}{#1}\right\}\text{#2}}

\newcommand{\tikzmark}[1]{\tikz[overlay,remember picture,baseline=(#1.base)] \node (#1) {\strut};}

\usepackage{tikz-cd}
\usepackage{tikz}
\usetikzlibrary{3d, matrix,decorations.pathreplacing,calc,patterns}

\usepackage[margin=1in]{geometry}
\theoremstyle{plain}
\newtheorem{theorem}{Theorem}[section]
\newtheorem{lemma}[theorem]{Lemma}
\newtheorem{proposition}[theorem]{Proposition}
\newtheorem{corollary}[theorem]{Corollary}

\theoremstyle{definition}
\newtheorem{definition}[theorem]{Definition}
\newtheorem{example}[theorem]{Example}
\newtheorem{remark}[theorem]{Remark}

\numberwithin{equation}{section}
\allowdisplaybreaks[1]

\newcommand{\C}{\mathbb{C}}
\newcommand{\Z}{\mathbb{Z}}
\newcommand{\ep}{\varepsilon}
\newcommand{\RR}{\mathbb{R}}

\newcommand{\Lie}{{\textup{Lie}}}

\makeatletter
\newcommand{\leqnomode}{\tagsleft@true\let\veqno\@@leqno}
\newcommand{\reqnomode}{\tagsleft@false\let\veqno\@@eqno}
\makeatother

\newcommand{\flag}{{\mathcal{F} \ell}}

\newcommand{\w}{{\widetilde{w}}}
\newcommand{\dw}{{\dot{w}}}
\newcommand{\Lam}[2]{{\Lambda^{{(#1)}}_{{#2}}}}
\newcommand{\A}[2]{{A^{{(#1)}}_{{#2}}}}
\newcommand{\X}[2]{{X^{{(#1)}}_{{#2}}}}
\newcommand{\ro}[2]{{\rho^{{(#1)}}_{{#2}}}}

\newcommand{\xii}[2]{{\xi^{{(#1)}}_{{#2}}}}
\newcommand{\bolda}[3]{{\mathbf{a}^{{(#1)}}_{{{#2}}, {{#3}}}}}
\newcommand{\Cstar}{{\C^{\ast}}}
\newcommand{\diag}{{\text{diag}}}
\newcommand{\Ad}{{\text{Ad}}}
\newcommand{\quo}{{\textup{quo}}}
\newcommand{\Pic}{{\textup{Pic}}}

\newcommand{\cA}{{\mathcal{A}}}

\newcommand{\defi}[1]{{\textit{#1}}}

\newcommand{\PjA}[1]{{\Phi_{{#1}}^{\cA}}}
\newcommand{\Cone}{{\textup{Cone}}}

\DeclareMathOperator{\GL}{GL}
\DeclareMathOperator{\U}{U}

\usepackage{graphicx}

\makeatletter
\newcommand*\bigcdot{\mathpalette\bigcdot@{.5}}
\newcommand*\bigcdot@[2]{\mathbin{\vcenter{\hbox{\scalebox{#2}{$\m@th#1\bullet$}}}}}
\makeatother

\begin{document}

\title[Flag Bott manifolds]{Flag Bott manifolds and \\
	the toric closure of a generic orbit \\
	associated to a generalized Bott manifold}

\author{Shintar\^o Kuroki}
\address{Faculty of Science, Department of Applied Mathematics, Okayama University of Science \\
	1-1 Ridai-cho Kita-ku Okayama-shi Okayama 700-0005, JAPAN}
\email{kuroki@xmath.ous.ac.jp}

\author{Eunjeong Lee}
\address{Center for Geometry and Physics, Institute for Basic Science (IBS), Pohang 37673, Republic of Korea}
\email{eunjeong.lee@ibs.re.kr}

\author{Jongbaek Song}
\address{School of Mathematics, KIAS, 85 Hoegiro Dongdaemun-gu, Seoul 02455, Republic of Korea}
\email{jongbaek@kias.re.kr}

\author{Dong Youp Suh}
\address{Department of Mathematical Sciences
	\\ KAIST \\ 291 Daehak-ro Yuseong-gu \\ Daejeon 34141 \\ Republic of Korea}
\email{dysuh@kaist.ac.kr}

\thanks{Kuroki was supported by JSPS KAKENHI Grant Number 17K14196.
	Lee was partially supported by IBS-R003-D1.
	Song was partially supported by the POSCO Science Fellowship of 
	POSCO TJ Park foundation and 
	Basic Science Research Program through the National Research Foundation of Korea (NRF) 
	funded by the Ministry of Education (NRF-2018R1D1A1B07048480).
	Lee, Song, and Suh were partially supported by Basic Science 
	Research Program through the National Research Foundation of Korea (NRF) 
	funded by the Ministry of  Science, ICT \& Future Planning 
	(No. 2016R1A2B4010823).}

\setcounter{tocdepth}{2} 

\begin{abstract}
To a direct sum of holomorphic line bundles, we can associate
two fibrations, whose fibers are, respectively,
the corresponding full flag manifold
and the corresponding projective space.
Iterating these procedures gives, respectively, a flag Bott tower and 
a generalized Bott tower.
It is known that a generalized Bott tower is a toric manifold.
However a flag Bott tower is not toric in general but we
show that it is a GKM manifold, and we also
show that for a given generalized Bott tower we can
find the associated flag Bott tower so that 
the closure of a generic torus orbit
in the latter is a blow-up of the former along certain invariant submanifolds.
We use GKM theory together with toric geometric arguments.
\end{abstract}

\maketitle

\section{Introduction}
A Bott tower $M_{\bullet}=\{M_j\mid 0 \leq j \leq m\}$ is a sequence of $\mathbb CP^1$-fibrations
$\mathbb CP^1 \hookrightarrow M_j {\longrightarrow} M_{j-1}$ 
such that $M_j$ is the projectivization of
the sum of two complex line bundles over $M_{j-1}$ where $M_0$ is a point
which is introduced in \cite{GrKa94}. 
Then each $M_j$ is a complex $j$\nobreakdash-dimensional nonsingular algebraic variety called the
\defi{$j$\nobreakdash-stage Bott manifold}.
Each Bott manifold $M_j$ has a $(\mathbb C^\ast)^{j}$\nobreakdash-action with 
which $M_j$ becomes 
a toric manifold, i.e., a nonsingular toric variety.

One of the important properties of Bott manifold is its relation with Bott--Samelson variety. 
A Bott--Samelson variety is a nonsingular algebraic variety that appeared in many areas of mathematics, for instance algebraic geometry and representation theory. For a given complex semisimple Lie group $G$ and a Borel subgroup $B \subset G$, the set of sections of a holomorphic line bundle over a Bott--Samelson variety has a structure of $B$-module, called a \defi{generalized Demazure module} (see~\cite{LLM02standard}). This gives a fruitful connection between representation theory and algebraic geometry. 
It is shown in \cite{GrKa94} and \cite{Pasq10} that every Bott--Samelson variety has a Bott manifold 
as its toric degeneration.\footnote{More precisely, \cite{GrKa94} provides a one-parameter family of complex structures on a Bott--Samelson variety which makes the Bott--Samelson variety into  a Bott manifold. Besides, \cite{Pasq10} constructs a toric degeneration of a Bott--Samelson variety, i.e., there is a flat family $\mathcal{X}$ over $\C$ such that $\mathcal{X}(t)$ is isomorphic to the Bott--Samelson variety for all $t \in \C \setminus \{0\}$ and $\mathcal{X}(0)$ is a Bott manifold. }
This relation between a Bott--Samelson variety and a Bott manifold gives interesting 
results on algebraic representations of $G$ in \cite{GrKa94}. 

Recently, a generalized notion of Bott--Samelson variety, called \emph{flag Bott--Samelson variety}, has been introduced in \cite{LeSu17} which extends the rich connection between representation theory and algebraic geometry given by Bott--Samelson variety.
Indeed, the set of sections of a holomorphic line bundle over a flag Bott--Samelson variety is also a generalized Demazure module. This result is applied to give polyhedral expressions for irreducible decompositions of tensor products of $G$-modules. 

In this article, we define a flag Bott tower $F_{\bullet}=\{F_j\mid 0 \leq j \leq m\}$ to be a sequence of 
the full flag fibrations
$\flag({n_j}+1) \hookrightarrow F_j \overset{p_j}{\longrightarrow} F_{j-1}$ where $F_j$ is the flagification of a sum of
$n_j+1$ many complex line bundles over $F_{j-1}$. We call each $F_j$ a \emph{flag Bott manifold}.
In \cite{LeSu17}, they construct a one-parameter family of complex structures on a flag Bott--Samelson variety which makes the flag Bott--Samelson variety into a flag Bott manifold, and this extends the known relation between Bott--Samelson varieties and Bott manifolds.

One of the goals of this article is to study torus actions on flag Bott manifolds.
In fact, the complex dimension of $F_m$ is $\sum_{j=1}^m n_j(n_j+1)/2$,
but there is an effective action of complex torus $\mathbf{H}$ of dimension $\sum_{j=1}^m n_j$  on $F_m$. Hence a flag Bott manifold is not a toric manifold in general.
With the restricted action of the compact torus $\mathbf{T}$ of  dimension $\sum_{j=1}^m n_j$ on a flag Bott manifold $F_m$, we get the following result:
\begin{theorem}[Theorem~\ref{thm_Fm_is_GKM}]
 Let $F_m$ be an $m$-stage flag Bott manifold with the effective action of $\bf T$. Then~$(F_m, \bf T)$ is a GKM manifold.
\end{theorem}
\noindent Moreover the concrete information of the GKM graph of $F_m$ is
computed in Theorem~\ref{thm_GKM_of_Fm}.

On the other hand, Bott manifolds are an important family of toric manifolds because of the cohomological rigidity problem which asks whether toric manifolds are topologically classified by their cohomology rings. This question has the affirmative answers
for some Bott manifolds (see~\cite{ChMa12}, \cite{Is12}, \cite{Choi15}, \cite{CMM15}). Moreover, it also has the affirmative answer for some generalized Bott manifolds (see~\cite{MaSu08}, \cite{CMS-Trnasaction}, \cite{CPS12}).
Here, a \textit{generalized Bott tower} $B_{\bullet}=\{B_j\mid 0 \leq j \leq m\}$ is defined similarly to a Bott tower 
but the difference is that $B_j$ is the projectivization of the sum of $n_j+1$ many 
complex line bundles 
instead of two line bundles.

Even though generalized Bott towers and flag Bott towers are two different generalizations of
Bott towers, there is an interesting relation between them. Namely,
let $B_{\bullet}$ be a generalized Bott tower with bundle maps $\pi_j \colon B_j \to B_{j-1}$. Then we define the \textit{associated flag Bott tower} $F_{\bullet}$ to $B_{\bullet}$ with bundle maps $p_j \colon F_j \to F_{j-1}$. Note that they satisfy $q_{j-1} \circ p_j = \pi_j \circ q_j$, where  $q_j \colon F_j \to B_j$ is induced by the projection map $\flag(n_j+1) \to \C P^{n_j}$ on each fiber  (see Section~\ref{sec_gBT_and_asso_fBT}). Moreover we prove that a generalized Bott manifold and its associated flag Bott manifold have the following relation: 
\begin{theorem}[Theorem~\ref{thm_main_thm_2}]
	Let $B_m$ be an $m$-stage generalized Bott manifold, and $F_m$ its associated flag Bott manifold. Then the closure of a generic orbit of $\bf H$-action in $F_m$ is the blow-up of $B_m$ along certain invariant submanifolds.
\end{theorem}
\noindent To obtain this result the GKM graph information of $F_m$ 
from Theorem~\ref{thm_GKM_of_Fm} is
essentially used together with some toric topological arguments.

We remark that every flag Bott tower is a $\C P$-tower, i.e., 
a sequence of an iterated complex projective space fibrations.
A $\C P$-tower is introduced in \cite{KuSu14} and \cite{KuSu15} 
as a more generalized notion than a generalized Bott tower. 

The paper is organized as follows.
In Section~\ref{sec_flag_Bott_manifolds}, we give an alternative description of a flag Bott manifold as the orbit space of
the product of general linear groups
under the action of 
the product of their Borel subgroups 
defined in 
\eqref{eq_Bm-action_defining_fBT}; see Proposition~\ref{prop_alter_desc_of_Fm}. 
In doing so, each complex line bundle
appearing in the construction of a flag Bott tower can be described in terms of characters of
maximal tori of general linear groups. Then we can associate
a sequence of integer matrices defined by the weights of the above mentioned characters to a flag Bott manifold as in
Theorem~\ref{prop_all_fBT_is_quotient}.
We also give an explicit description of the tangent bundle of a flag Bott manifold in 
Proposition~\ref{prop_tangent_bdle}, which will be used in the GKM description of a flag Bott manifold in Section~\ref{sec_GKM_desc_of_fBM}.

In Section~\ref{sec_GKM_desc_of_fBM}, we define the canonical torus action on a flag Bott manifold, and
find an explicit description of the tangential representation at a fixed point in 
Proposition~\ref{prop_fBT_is_GKM}. We then see easily that every flag Bott manifold is
a GKM manifold. Moreover an explicit description of the GKM graph of a flag Bott manifold  
is given in Theorem~\ref{thm_GKM_of_Fm}.

In Section~\ref{sec_gBT_and_asso_fBT}, we define the associated flag Bott tower to a given generalized Bott tower. Then
Proposition~\ref{prop:lambda_of_asso_gBT} gives the integer matrices corresponding to the associated
flag Bott tower. 

In Section~\ref{sec_generic_orbit_closure_in_asso_fBM}, we study the relation between a generalized Bott manifold $B_m$ and the closure $X$ of a generic orbit of the associated flag Bott manifold $F_m$. 
This can be accomplished by calculating the fan of $X$ 
in Theorem~\ref{thm_main_thm_1} using the axial functions of the GKM graph of $F_m$. Then we show that  the toric variety $X$ comes from a series of
blow-up of $B_m$ in Theorem~\ref{thm_main_thm_2}.

\section{Flag Bott manifolds}
\label{sec_flag_Bott_manifolds}

\subsection{Definition of flag Bott manifolds}
\label{sec_def_of_fBT}
Let $M$ be a complex manifold and $E$ an $n$-dimensional 
holomorphic vector bundle over $M$. 
Recall from \cite[p. 282]{BoTu82} that 
the \defi{associated flag bundle} 
$\flag(E) \to M$ is obtained from $E$ by replacing each fiber
$E_p$ by the full flag manifold $\flag(E_p)$.

\begin{definition}\label{def_fBT}
	A \defi{flag Bott tower} $F_{\bullet}= \{ F_j \mid 0 \leq j \leq m\}$ of
	height $m$, or an \defi{$m$-stage flag Bott tower}, is a sequence,
	\[
	\begin{tikzcd}
	F_m \arrow[r, "p_m"] & F_{m-1} \arrow[r, "p_{m-1}"] & \cdots \arrow[r, "p_2"]
	& F_1 \arrow[r, "p_1"] & F_0 = \{\text{a point}\},
	\end{tikzcd}
	\]
	of manifolds $F_j = \flag\left(\bigoplus_{k=1}^{n_j+1} \xii{j}{k} \right)$
	where $\xii{j}{k}$ is a holomorphic 
	line bundle over $F_{j-1}$ for each $1 \leq k \leq n_j+1$
	and $1 \leq j \leq m$.
	We call $F_j$ the \defi{$j$-stage flag Bott manifold} of the
	flag Bott tower $F_{\bullet}$. 
\end{definition}
Here are some examples of flag Bott manifolds.
\begin{example}\label{example_fBT}
	\begin{enumerate}
		\item The flag manifold $\flag(\C^{n+1})=\flag(n+1)$ is 
		a flag Bott tower of height 1.
		In particular, $\flag(2)=\C P^1$
		is a $1$\nobreakdash-stage flag Bott manifold.
		\item The product of flag manifolds 
		$\flag(n_1+1) \times \cdots \times \flag(n_m+1)$
		is a flag Bott manifold of height $m$. 
		\item Recall from \cite{GrKa94} that an \defi{$m$-stage Bott manifold} is 
		a sequence of $\C P^1$\nobreakdash-fibrations such that each stage is 
		the projective bundle of the sum of two line bundles.
		When $n_j = 1$ for $1 \leq j \leq m$, an $m$\nobreakdash-stage flag Bott manifold is an
		$m$\nobreakdash-stage Bott manifold. \hfill \qed
	\end{enumerate}
\end{example}
\begin{definition}
Two flag Bott towers $F_{\bullet}$ and $F_{\bullet}'$
are \defi{isomorphic} if there is a collection of (holomorphic) diffeomorphisms
$\varphi_j \colon F_j \to F_j'$ which commute with the 
maps $p_j \colon F_j \to F_{j-1}$ and $p_j' \colon F_j' \to F_{j-1}'$.
\end{definition}
\begin{remark}
\begin{enumerate}
	\item One can define $F_j$ to be 
	$\flag(E_j)$ for some holomorphic vector bundle $E_j$ over 
	$F_{j-1}$. However, since we want to consider torus actions on $F_m$,
	we assume $E_j$ to be a sum of holomorphic line bundles
	 in Definition~\ref{def_fBT}.
\item 	
Even though we are concentrating on full flag fibrations in this paper,
one can also study other kinds of induced fibrations such as 
partial flag fibrations, isotropic flag fibrations, etc., which
require further works. In~\cite{KKLS}, the authors study such iterated flag fibrations. \hfill \qed
\end{enumerate}
\end{remark}

\subsection{Orbit space construction of flag Bott manifolds}
\label{subsec_orb_sp_const_fbt}
In this subsection, we consider an orbit space
construction of a flag Bott tower
in~ Proposition~\ref{prop_alter_desc_of_Fm}
using the complex Lie groups $\GL(n) := \GL(n, \mathbb{C})$
in order to consider the canonical torus action on it (see Subsection~\ref{sec_torus_action_fBT}).

A flag Bott tower of height $1$ is
the flag manifold $\flag(n_1+1)$ which is the orbit space
$\GL(n_1+1)/B_{\GL(n_1+1)}$, where $B_{\GL(n_1+1)}$ is the
set of upper triangular matrices in $\GL(n_1+1)$.
To describe flag Bott manifolds of higher stages, we begin with a matrix $A$
of size $(n+1) \times (n'+1)$ whose row
vectors are $\mathbf{a}_1,\dots,\mathbf{a}_{n+1} \in \Z^{n'+1}$, i.e.,
\[
A = \begin{bmatrix}
\mathbf{a}_1 \\ \mathbf{a}_2 \\ \vdots \\ \mathbf{a}_{n+1} 
\end{bmatrix}
= \begin{bmatrix}
\mathbf{a}_1(1) & \mathbf{a}_1(2) & \cdots & \mathbf{a}_1(n'+1) \\
\mathbf{a}_2(1) & \mathbf{a}_2(2) & \cdots & \mathbf{a}_2(n'+1) \\
\vdots & \vdots & & \vdots \\
\mathbf{a}_{n+1}(1) & \mathbf{a}_{n+1}(2) & \cdots & \mathbf{a}_{n+1}(n'+1)
\end{bmatrix}
\in M_{n+1, n'+1}(\Z),
\]
which encodes a $B_{\GL(n'+1)}$-action on $\GL(n+1)$ as follows.
Let $H_{\GL(n+1)} \subset \GL(n+1)$, respectively $H_{\GL(n'+1)} \subset \GL(n'+1)$, 
be the set of diagonal matrices in $\GL(n+1)$, respectively $\GL(n'+1)$. 
Since the character group $\chi(H_{\GL(n'+1)})$ is isomorphic to $\Z^{n'+1}$, 
the matrix $A$ gives a homomorphism $H_{\GL(n'+1)} \to H_{\GL(n+1)}$ 
defined by
\begin{equation}\label{def_of_homo_defined_by_A}
h \mapsto \text{diag}\left( 
h^{\mathbf{a}_1}, h^{\mathbf{a}_2},\dots,h^{\mathbf{a}_{n+1}}
\right) \in H_{\GL(n+1)}.
\end{equation}
Here, for $h = \diag(h_1,\dots,h_{n'+1}) \in H_{\GL(n'+1)}$ and $\mathbf{a} = (\mathbf{a}(1),\dots,\mathbf{a}(n'+1)) \in \Z^{n'+1}$, $h^{\mathbf{a}}:=  h_1^{\mathbf{a}_1} \cdots h_{n'+1}^{\mathbf{a}(n'+1)}$.
By composing the canonical projection $\Upsilon \colon B_{\GL(n'+1)}
\to H_{\GL(n'+1)}$ with the homomorphism \eqref{def_of_homo_defined_by_A},
we define the homomorphism $\Lambda(A) \colon B_{\GL(n'+1)} \to H_{\GL(n+1)}$
associated to the matrix $A \in M_{n+1, n'+1}(\Z)$:
\begin{equation}\label{eq_def_of_b_to_a}
\Lambda(A)(b) := \text{diag}(\Upsilon(b)^{\mathbf{a}_1},
\Upsilon(b)^{\mathbf{a}_2},\dots,\Upsilon(b)^{\mathbf{a}_{n+1}}) \in H_{\GL(n+1)}
\quad \text{ for } b \in B_{\GL(n'+1)}.
\end{equation}

Now, let $n_1,\dots,n_m \in \mathbb{Z}_{>0}$. Then, for a given sequence of integer matrices
\begin{equation}\label{eq_cA_def}
\cA := (\A{j}{\ell})_{1 \leq \ell < j \leq m}
\in \prod_{1 \leq \ell < j \leq m} M_{n_j+1, n_{\ell}+1}(\Z),
\end{equation}
we define a right action $\PjA{j}$ of $\prod_{\ell=1}^jB_{\GL(n_{\ell}+1)}$ on 
$\prod_{\ell=1}^{j}\GL(n_{\ell}+1)$ by
\begin{equation}
\label{eq_Bm-action_defining_fBT}
\begin{split}
\PjA{j}&((g_1,g_2,\dots,g_j), (b_1,b_2,\dots,b_j)) \\
& := 	\Big(g_1b_1,
\big(\Lam{2}{1}(b_1)\big)^{-1}g_2b_2,
\big(\Lam{3}{1}(b_1)\big)^{-1}\big(\Lam{3}{2}(b_2)\big)^{-1}g_3b_3,\ldots,\\
& \quad \quad \quad
\big(\Lam{j}{1}(b_1)\big)^{-1}\big(\Lam{j}{2}(b_2)\big)^{-1} \cdots \big(\Lam{j}{j-1}(b_{j-1})\big)^{-1}
g_jb_j \Big)
\end{split}
\end{equation}
for $1 \leq j \leq m$, where $\Lam{j}{\ell} := \Lambda(\A{j}{\ell})$ is 
the homomorphism $B_{\GL(n_{\ell}+1)} \to H_{\GL(n_{j}+1)}$
associated to the matrix $\A{j}{\ell}$  as defined in \eqref{eq_def_of_b_to_a} 
for $1 \leq \ell < j \leq m$. 

\begin{example}
	For $n_1 = 2$, $n_2 = 1$, $n_3 =1$, consider the following matrices:
	\[
	\A{2}{1} = \begin{bmatrix}
	c_1 & c_2 & 0 \\ 0 & 0 & 0 
	\end{bmatrix}, \quad
	\A{3}{1} = \begin{bmatrix}
	d_1 & d_2 & 0 \\ 0 & 0 & 0
	\end{bmatrix}, \quad
	\A{3}{2} = \begin{bmatrix}
	f_1 & 0 \\ 0 & 0
	\end{bmatrix}.
	\]
	Then the right action of $B_{\GL(3)} \times B_{\GL(2)} \times B_{\GL(2)}$
	on $\GL(3) \times \GL(2) \times \GL(2)$
	defined in \eqref{eq_Bm-action_defining_fBT} is
	\[
	\reqnomode
	(g_1,g_2,g_3) \cdot (b_1,b_2,b_3) 
	= \left(g_1b_1, \diag\left(b_1^{(-c_1,-c_2,0)},1\right) g_2 b_2,
	\diag\left(b_1^{(-d_1,-d_2,0)} b_2^{(-f_1,0)},1 \right) g_3b_3 \right).
	\tag*{\qed}
	\]	
\end{example}

\begin{lemma}\label{lemma_right_action_of_B_is_free_and_proper}
	The right action $\PjA{j}$ in \eqref{eq_Bm-action_defining_fBT} is free and proper
	for $1\leq j \leq m$.
\end{lemma}
\begin{proof}
	For $g:=(g_1,\dots,g_j) \in \prod_{\ell=1}^j \GL(n_{\ell}+1)$ and
	$(b_1,\dots,b_j) \in \prod_{\ell=1}^jB_{\GL(n_{\ell}+1)}$,
	the equality $g_1 = g_1b_1$ implies that $b_1$ is the identity matrix
	since $g_1$ is invertible. Similarly, the equation 
	$g_2 = \big(\Lam{2}{1}(b_1)\big)^{-1} g_2b_2 = g_2b_2$ gives that $b_2$ is the identity.
	Continuing in this manner, we conclude
	that the isotropy subgroup at $g$ is trivial, this shows that the action
	$\PjA{j}$ is free. 
	
	To prove the properness of the action, it is enough to show that for
	every sequence $(g^r):= (g_1^r,\dots,g_j^r)$ in $\prod_{\ell=1}^j\GL(n_{\ell}+1)$ and $(b^r) := (b_1^r,\dots,b_j^r)$ in 
	$\prod_{\ell=1}^j B_{\GL(n_{\ell}+1)}$ such that
	both $(g^r)$ and $(\PjA{j}(g^r, b^r))$ converge,  a 
	subsequence of $(b^r)$ converges (see~\cite[Proposition 21.5]{Le13}).
	Note that for convergent sequences $(A^r) \to A$ and $(B^r)\to B$ in 
	$\GL(n+1)$, the sequence $(A^rB^r)$ also converges to $AB$ since the multiplication 
	map is continuous. 
	Also for a convergent sequence $(A^r) \to A$ in $\GL(n+1)$, we have that $A_{ij} 
	= \lim_{r\to\infty} (A^r)_{ij}$.
	Since both sequences $(g_1^r)$ and $(g_1^r b_1^r)$ converge, the sequence
	$(b_1^r)$ also converges in $B_{\GL(n_1+1)}$. Similarly, 
	sequences $\Big(\big(\Lam{2}{1}(b_1^r)\big)^{-1}g_2^rb_2^r \Big)$, $(g_2^r)$ and $(b_1^r)$
	converge so that the sequence $(b_2^r)$ also converges. 
	By continuing this process, we show
	that the action $\PjA{j}$ is proper. 
\end{proof}

For a complex manifold $M$ with a free and proper action of a group $G$,
the orbit space $M/G$ is a complex manifold (see~\cite[Proposition 2.1.13]{Huyb05}).
Hence by Lemma~\ref{lemma_right_action_of_B_is_free_and_proper}, the orbit space
\begin{equation}\label{eq_def_of_Fj}
F_j^{\quo}(\cA) := \prod_{\ell=1}^j \GL(n_{\ell}+1)/\PjA{j}
\end{equation}
is a complex manifold, where $\PjA{j}$ is the action defined in \eqref{eq_Bm-action_defining_fBT}.

For the remaining part of this subsection, we will prove that the orbit spaces $F_j^{\quo}(\cA)$ are flag Bott manifolds.
Since $\chi(\prod_{\ell=1}^j H_{\GL(n_{\ell}+1)}) \cong \bigoplus_{\ell=1}^j 
\Z^{n_{\ell}+1}$,
for each integer vector $(\mathbf{a}_1,\dots,\mathbf{a}_j) \in 
\bigoplus_{\ell=1}^j \Z^{n_{\ell}+1}$ we can
define a holomorphic line bundle over $F_j^{\quo}$ 
as follows:
\begin{equation}\label{eq_def_of_xi_a}
\xi(\mathbf{a}_1,\dots,\mathbf{a}_j) := 
\left(\prod_{\ell=1}^j \GL(n_{\ell}+1) \times \C \right)\bigg/
\prod_{\ell=1}^j B_{\GL(n_{\ell}+1)}
\end{equation}
where the right action is 
\[
(g_1,\dots,g_j, v) \cdot (b_1,\dots,b_j)
:= \left(\PjA{j}((g_1,\dots,g_j), (b_1,\dots,b_j)),
b_1^{-\mathbf{a}_1} \cdots b_j^{-\mathbf{a}_j}v \right).
\]

\begin{proposition}\label{prop_alter_desc_of_Fm}	
$F_{\bullet}^{\quo}(\cA):= \{F_{j}^{\quo}(\cA) \mid 0 \leq j \leq m \}$ is a flag Bott tower
	of height $m$.
\end{proposition}

\begin{definition}\label{def_fBT_determined_by_A}
		We say that \defi{a flag Bott tower $F_{\bullet}$ is determined by a sequence of matrices $\cA = (\A{j}{\ell})_{1 \leq \ell < j \leq m} \in \prod_{1 \leq \ell < j \leq m } M_{n_j+1, n_{\ell}+1} (\Z)$}  if $F_{\bullet}$ is isomorphic to $F_{\bullet}^{\quo}(\cA) = \{F^{\quo}_j(\cA) \mid 0 \leq j \leq m\}$ as flag Bott towers.
\end{definition}

Note that, in the next section, we will show that every flag Bott towers can be described as an orbit space, that is, every flag Bott tower is determined by a certain $\cA$ (see Theorem~\ref{prop_all_fBT_is_quotient}). 
\begin{proof}[Proof of Proposition~\ref{prop_alter_desc_of_Fm}]
	By the definition of the action $\PjA{j}$, we have
	the following fibration structure:
	\[
	\GL(n_{j}+1)/B_{\GL(n_j+1)} \hookrightarrow F_j^{\quo} \to F_{j-1}^{\quo}.
	\]
	Since $\GL(n_j+1)/B_{\GL(n_j+1)} \cong \flag(n_j+1)$,
	the manifold $F_j^{\quo}$ is a $\flag(n_j+1)$-fibration over $F_{j-1}^{\quo}$.
	For simplicity, let
	$\xi^{(j)} := \bigoplus_{k=1}^{n_j+1} \xi(\bolda{j}{k}{1},\dots,\bolda{j}{k}{j-1})$	, where $\bolda{j}{k}{\ell}$ is the $k$th row vector of
			the matrix $\A{j}{\ell}$ for $1 \leq \ell \leq j-1$. 
	Consider the map $\varphi_j \colon F_j^{\quo} \to 
	\flag(\xi^{(j)})$ defined by 
	\[
	[g_1,\dots,g_{j-1},g_j] \mapsto ([g_1,\dots,g_{j-1}], V_{\bullet}).
	\]
	Here $V_{\bullet}=
	(V_1 \subsetneq V_2 \subsetneq \cdots \subsetneq V_{n_j}
	\subsetneq (\xi^{(j)})_{[g_1,\dots,g_{j-1}]}) $ is the full flag
	such that the vector space $V_k$ is spanned by the first $k$ columns of $g_j$. 
	We claim that $\varphi_j$ is a holomorphic diffeomorphism. First, we  check that the map $\varphi_j$ is well-defined. 
	We observe that
	\begin{align*}
	[\PjA{j}((g_1,\dots,g_{j-1},g_j), (b_1,\dots,b_{j-1},b_j))] 
	&\mapsto ([\PjA{j-1}((g_1,\dots,g_{j-1}),(b_1,\dots,b_{j-1}))],V_{\bullet}')\\
	&= ([g_1,\dots,g_{j-1}],V_{\bullet}')
	\end{align*}
	for $(b_1,\dots,b_{j-1},b_j) \in \prod_{\ell=1}^{j-1} B_{\GL(n_{\ell}+1)} \times B_{\GL(n_j+1)}$. Here 
	$V_{\bullet}' = (V_1' \subsetneq V_2' \subsetneq 
	\cdots \subsetneq V_{n_j}' \subsetneq (\xi^{(j)})_{[g_1,\dots,g_{j-1}]} )$ 
	is the full flag whose vector space $V_k'$ is spanned by the first $k$ 
	columns of the matrix 
	$\big(\Lam{j}{1}(b_1) \big)^{-1} \cdots \big(\Lam{j}{j-1}(b_{j-1}) \big)^{-1}g_j$. 
	Since we have
	\begin{equation}\label{eq_prod_of_Lj1_to_Ljj-1}
	\big(\Lam{j}{1}(b_1)\big)^{-1} \cdots \big(\Lam{j}{j-1}(b_{j-1})\big)^{-1}\mathbf{v} \sim \mathbf{v}
	\quad \text{ for } \mathbf{v} \in (\xi^{(j)})_{[g_1,\dots,g_{j-1}]},
	\end{equation}
	the map $\varphi_j$ is well-defined. Here the equivalence relation $\sim$ comes from the definition of the bundle $\xi^{(j)}$.

The inverse $\flag(\xi^{(j)}) \to F_j^{\quo}$ of $\varphi_j$ is given by
	\[
	([g_1,\dots,g_j], V_{\bullet}) \mapsto [g_1,\dots,g_{j-1},g_j],
	\]
where $g_j$ is the matrix such that the first $k$ columns span the vector space $V_k$ for $1 \leq k \leq n_j+1$. Note that this map is again well-defined since $[g_1,\dots,g_{j-1},g_j] = [g_1,\dots,g_{j-1},g_jb_j]$ for $b_j \in B_{\GL(n_j+1)}$.  Hence the map $\varphi_j$ is a diffeomorphism, and the result follows since $\varphi_j$ commutes with bundle projection maps.
\end{proof}
\begin{example}
	For $n_1 = 2, n_2 = 1$, and $n_3 = 1$, let
	\[
	\A{2}{1} = \begin{bmatrix}
	c_1 & c_2 & 0 \\ 0 & 0 & 0 
	\end{bmatrix},
	\A{3}{1} = \begin{bmatrix}
	d_1 & d_2 & 0 \\ 0 & 0 & 0
	\end{bmatrix},
	\A{3}{2} = \begin{bmatrix}
	f_1 & 0 \\ 0 & 0
	\end{bmatrix}.
	\]
	Let $\PjA{j}$ be the right action of $\prod_{\ell=1}^j B_{\GL(n_{\ell}+1)}$ on $\prod_{\ell=1}^{j} \GL(n_{\ell}+1)$ defined
	in \eqref{eq_Bm-action_defining_fBT} for $j =1,2,3$. 
	Then, by Proposition~\ref{prop_alter_desc_of_Fm}, the following
	flag Bott tower $F_{\bullet}$ 
	is isomorphic to $F^{\quo}_{\bullet}(\cA)$ as flag Bott towers.
	\[
	\begin{tikzcd}
	& \xi((d_1,d_2,0),(f_1,0)) \oplus \underline{\C}  \dar 
	& \xi(c_1,c_2,0) \oplus \underline{\C} \dar \\
	\flag\left(\xi\left((d_1,d_2,0),(f_1,0)\right) \oplus \underline{\C}\right) \rar
	\arrow[d, equal]
	& \flag\left(\xi\left(c_1,c_2,0\right) \oplus \underline{\C}\right) \rar  
	\arrow[d, equal]
	& \flag(3) \rar \arrow[d, equal] & \{\text{a point}\}\arrow[d, equal] \\ [-3ex]
	F_3 & F_2 & F_1 & F_0
	\end{tikzcd}
	\]
	The line bundle $\xi((d_1,d_2,0),(f_1,0))$ over $F_2$ is 
	$(\GL(3) \times \GL(2) \times \C)/(B_{\GL(3)} \times B_{\GL(2)})$,
	where the right action of $B_{\GL(3)} \times B_{\GL(2)}$ is 
	\[
	\reqnomode
	(g_1,g_2,v) \cdot (b_1,b_2) := \left(\PjA{2}((g_1,g_2), (b_1,b_2)), 
	b_1^{-(d_1,d_2,0)} b_2^{-(f_1,0)} v\right). \tag*{\qed}
	\]
\end{example}

\subsection{Tautological filtration over a flag Bott manifold}
\label{subsec_taut_filt_over_FBM}
In this subsection, we prove the following theorem that every flag Bott tower $F_{\bullet}$ can be obtained by the orbit space construction as in Subsection~\ref{subsec_orb_sp_const_fbt}.
\begin{theorem}\label{prop_all_fBT_is_quotient}
	Let $F_{\bullet}$ be a flag Bott tower of height $m$. 
	Then there is a sequence of integer matrices
	$\cA = (\A{j}{\ell})_{1 \leq \ell < j \leq m}
	\in \prod_{1 \leq \ell < j \leq m} M_{n_j+1,n_{\ell}+1}(\Z)$
	such that $F_{\bullet}$ is isomorphic to $F^{\quo}_{\bullet}(\cA):= \{F^{\quo}_j(\cA) \mid 0 \leq j \leq m\}$
	as flag Bott towers.
\end{theorem}
By the above theorem, for any flag Bott tower $F_{\bullet}$ there exists a set $\cA$ satisfying that \textit{$F_{\bullet}$ is determined by~$\cA$} (see~Definition~\ref{def_fBT_determined_by_A}).
To give a proof of Theorem~\ref{prop_all_fBT_is_quotient}, we begin with studying
 holomorphic line bundles over a flag Bott manifold.
For $1 \leq j \leq m$, there is a 
\defi{universal} or \defi{tautological} filtration of subbundles
\begin{equation}\label{eq_tautological_filt_Fj}
0 = U_{j,0} \subset U_{j,1} \subset U_{j,2} \subset
\cdots \subset U_{j,n_j} \subset U_{j,n_j+1} = p_j^{\ast} \xi^{(j)}
\end{equation}
on $F_j = \flag(\xi^{(j)})$,
where we put $\xi^{(j)} := \bigoplus_{k=1}^{n_j+1} \xii{j}{k}$ for simplicity. 
Over a point 
\[
(p,V_{\bullet}) = \left(p,(V_0 \subset V_1 \subset \cdots \subset V_{n_j} \subset (\xi^{(j)})_p ) \right)
\]
of $F_j$ for $p \in F_{j-1}$, 
the fiber of the subbundle $U_{j,k}$ 
is the vector space $V_k$ of the flag $V_{\bullet}$ for 
$1 \leq k \leq n_j+1$. 
Hence we have the quotient line bundle $U_{j,k}/U_{j,k-1}$ over $F_j$ for 
$1 \leq k \leq n_j+1$. The following lemma states that
using these line bundles, we can express any holomorphic line bundle
over a flag Bott manifold.
\begin{lemma}\label{lemma_holo_line_bdle_over_Fj}
	Let $F_{\bullet}$ be a flag Bott tower.
	Then the set of line bundles 
		\[
	\left\{U_{j,k}/U_{j,k-1} \mid 1 \leq k \leq n_j+1 \right\} \cup 
	\bigcup_{\ell=1}^{j-1}
	\left\{p_{j}^* \circ \cdots \circ p_{\ell+1}^* (U_{\ell,k}/U_{\ell,k-1}) \mid 1 \leq k \leq n_{\ell}+1 \right\}
	\]
	generates the Picard group $\Pic(F_j)$ for $1 \leq j \leq m$.
\end{lemma}
\begin{proof}
	Using the result \cite[Remark~21.18]{BoTu82} on the cohomology ring of the induced flag bundle and an induction on the stage of $F_{\bullet}$, one can see that
	the degree 2 cohomology group $H^2(F_j; \Z)$ is generated 
	by the first Chern classes of line bundles
	\[
	\left\{U_{j,k}/U_{j,k-1} \mid 1 \leq k \leq n_j+1 \right\} \cup 
	\bigcup_{\ell=1}^{j-1}
	\left\{p_{j}^* \circ \cdots \circ p_{\ell+1}^* (U_{\ell,k}/U_{\ell,k-1}) \mid 1 \leq k \leq n_{\ell}+1 \right\}
	\]
	for $1 \leq j \leq m$. 
	Therefore, any cohomology class of degree $2$ can be obtained as the first Chern class of a tensor product of these line bundles.
	Hence it is enough to show that the cycle map 
	$c_1 \colon \Pic(F_j) \to H^2(F_j;\Z)$ is an isomorphism. 
	We recall that the cycle map 
	$\Pic(X) \to H^2(X;\Z)$ is an isomorphism
	for a full flag manifold $X$. 
	Also for the full flag bundle $X$ over a smooth variety $Y$,
	if the cycle map for $Y$ is an isomorphism, then 
	the cycle map for $X$ is also an isomorphism (see~\cite[Example 19.1.11]{Fult13}).
	This proves that the cycle map $c_1 \colon \Pic(F_j) \to H^2(F_j;\Z)$
	is an isomorphism for $1 \leq j \leq m$. 
\end{proof}

\begin{lemma}\label{lem_quot_bdle_and_a}
	For a sequence of integer matrices 
	$\cA=(\A{j}{\ell})_{1 \leq \ell < j \leq m}
	\in \prod_{1 \leq \ell < j \leq m} M_{n_j+1, n_{\ell}+1}(\Z)$,
	let $F_{\bullet}^{\quo} := F_{\bullet}^{\quo}(\cA)$ be the flag Bott tower defined 
	as in~\eqref{eq_def_of_Fj}. Then,
	the line bundle $U_{j,k}/U_{j,k-1} \to F_j^{\quo}$ is isomorphic to the bundle $\xi(\mathbf{0},\dots,\mathbf{0},\mathbf{e}_k) \to F_j^{\quo}$ defined in~\eqref{eq_def_of_xi_a}.
\end{lemma}
\begin{proof}
	From Proposition~\ref{prop_alter_desc_of_Fm}
	that the $j$-stage flag Bott manifold $F_j^{\quo}$ is
	the induced flag bundle $\flag(\xi^{(j)})$ over $F_{j-1}^{\quo}$, where
	$\xi^{(j)} = \bigoplus_{k=1}^{n_j+1} \xi(\bolda{j}{k}{1},\dots,\bolda{j}{k}{j-1})$
	and $\bolda{j}{k}{\ell}$ is the $k$th row vector of the matrix
	$\A{j}{\ell}$ for $1 \leq \ell \leq j-1$. 
	Consider a point $g = [g_1,\dots,g_{j}]$ in $F_j^{\quo}$.
	Because of the bundle structure $F_j^{\quo}=\flag(\xi^{(j)}) 
	\stackrel{p_j}{\longrightarrow} F_{j-1}^{\quo}$, 
	this point $g$ can be considered as a full flag
	$V_{\bullet} = \left(V_1 \subsetneq V_2 \subsetneq \cdots \subsetneq
	V_{n_j} \subsetneq (\xi^{(j)})_{p_j(g)} \right)$, where $(\xi^{(j)})_{p_j(g)}$ is the 
	fiber over $p_j(g)$. 
	The fiber of $U_{j,k}$ at $g$ is the vector space $V_{k} \subset (\xi^{(j)})_{p_j(g)}$ spanned by the first $k$ column vectors 
	$\mathbf{v}_1,\dots,\mathbf{v}_k \in (\xi^{(j)})_{p_j(g)}$ 
	of $g_{j} \in \GL(n_j+1)$. Hence
	the fiber of $U_{j,k}/U_{j,k-1}$ at $g$ is $V_k/V_{k-1}$,
	which is spanned by the vector $\mathbf{v}_k \in (\xi^{(j)})_{p_j(g)}$. For an element $b = (b_1,\dots,b_j) \in \prod_{\ell=1}^j B_{\GL(n_j)}$, the $k$th column vector $\mathbf{v}_k'$ of the last coordinate of $\PjA{j}(g,b)$ is given by
	\[
	\mathbf{v}_k' =
	\Big( \big(\Lam{j}{1}(b_1)\big)^{-1} \cdots \big(\Lam{j}{j-1}(b_{j-1})\big)^{-1} 
	\mathbf{v}_k\Big) b_j^{\mathbf{e}_k}
	= b_j^{\mathbf{e}_k}\big(\Lam{j}{1}{(b_1)}\big)^{-1} \cdots \big(\Lam{j}{j-1}(b_{j-1})\big)^{-1} 
	\mathbf{v}_k \sim b_j^{\mathbf{e}_k} \mathbf{v}_k.
	\]
	Here, the equivalence comes from~\eqref{eq_prod_of_Lj1_to_Ljj-1}.
	Hence the result follows.
\end{proof}

Using the above two lemmas, we can prove Theorem~\ref{prop_all_fBT_is_quotient}.
\begin{proof}[Proof of Theorem~\ref{prop_all_fBT_is_quotient}]
	We prove the proposition using the induction argument on the height of 
	a flag Bott tower. When the height is $1$, then it is obvious that
	any full flag manifold can be described as the orbit space
	$\GL(n_1+1)/B_{\GL(n_1+1)}$. 
	
	Assume that the theorem holds for flag Bott towers whose height is less that $m$.
	For a flag Bott tower $F_{\bullet}$ of height $m$, by the induction hypothesis, 
	we have a sequence of integer matrices
	$(\A{j}{\ell})_{1 \leq \ell < j \leq m-1}
	\in \prod_{1 \leq \ell < j \leq m-1} M_{n_j+1, n_{\ell}+1}(\Z)$
	such that $\{F_j \mid 0 \leq j \leq m-1 \}$ is isomorphic to the orbit spaces
	$\{F_j^{\quo} \mid 0 \leq j \leq m-1\}$
	as flag Bott towers.
	To prove the claim, it is enough to find suitable integer matrices 
	$\A{m}{1}, \dots, \A{m}{m-1}$ such that 
	$(\A{j}{\ell})_{1 \leq \ell < j \leq m}$ gives the flag Bott manifold $F_m$.

	Let $F_m= \flag\left(\bigoplus_{k=1}^{n_m+1} \xii{m}{k} \right)$, 
	where $\xii{m}{k}$ is a holomorphic
	line bundle over $F_{m-1}$. 
	Then, by the induction hypothesis, the $(m-1)$-stage flag Bott manifold 
	$F_{m-1}$ can be expressed as the orbit $\prod_{\ell=1}^{m-1} \GL(n_{\ell}+1) /\PjA{m-1}$. 
	Hence, by Lemmas~\ref{lemma_holo_line_bdle_over_Fj} and~\ref{lem_quot_bdle_and_a}, there exists 
	a suitable integer vector $\left(\bolda{m}{k}{1},\dots,\bolda{m}{k}{m-1} \right)
	\in \bigoplus_{\ell=1}^{m-1} \Z^{n_{\ell}+1}$ 
	such that 
	\[
	\xi\left(\bolda{m}{k}{1},\dots,\bolda{m}{k}{m-1} \right)
	=\xii{m}{k} \quad
	\text{ for } 1 \leq k \leq n_m+1. 
	\]
	Consider the integer matrix $\A{m}{\ell} 
	\in M_{n_m+1, n_{\ell}+1}(\Z)$ whose row vectors are 
	$\bolda{m}{1}{\ell},\dots,\bolda{m}{n_m+1}{\ell}$ for
	$1 \leq \ell \leq m-1$.
	Let $F_m^{\quo}$ be the flag Bott manifold determined by 
	integer matrices $(\A{j}{\ell})_{1 \leq \ell < j \leq m}$.
	Then by Proposition~\ref{prop_alter_desc_of_Fm}, we have 
	the following bundle map $\varphi$ which is a holomorphic diffeomorphism:
	\[
	\varphi \colon F_m^{\quo} \to 
	\flag\left(\bigoplus_{k=1}^{n_m+1}
	\xi(\bolda{m}{k}{1},\dots,\bolda{m}{k}{m-1})
	\right) = F_m. \qedhere
	\]
\end{proof}

\begin{remark}[Description of $F_m$ using compact Lie groups]
	\label{rmk_compact_description}
	Using the compact subgroups $\U(n_j+1) \subset \GL(n_j+1)$ and
	the compact maximal torus
	$T^{n_j+1} \subset H_{\GL(n_j+1)}$ for $1\leq j \leq m$, 
	consider the orbit space:
	\[
	\prod_{j=1}^m U(n_j+1) \bigg/ \prod_{j=1}^m T^{n_j+1},
	\]
	where the right action is defined by
	\begin{equation}\label{eq_quotient_torus}
	\begin{split}
	(g_1,\dots,g_m) \cdot (t_1,\dots,t_m)
	& = \Big(g_1t_1, \big(\Lam{2}{1}(t_1)\big)^{-1} g_2t_2, \big(\Lam{3}{1}(t_1)\big)^{-1}
	\big(\Lam{3}{2}(t_2)\big)^{-1} g_3t_3, \dots,\\
	& \quad \quad \quad 
	\big(\Lam{m}{1}(t_1)\big)^{-1} \big(\Lam{m}{2}(t_2)\big)^{-1} \cdots 
	\big(\Lam{m}{m-1}(t_{m-1})\big)^{-1}g_mt_m\Big).
	\end{split}
	\end{equation}
	Then the above manifold is a compact manifold which is
	diffeomorphic to $F_m$
	since $\U(n+1)/T^{n+1}$ is diffeomorphic to 
	$\GL(n+1)/B_{\GL(n+1)}$.
	We will also use this description for $F_m$. \hfill \qed
\end{remark}

\begin{remark}\label{rmk_eq_cohomology_ring_of_fBT}
	Let $F_m$ be the $m$-stage flag Bott manifold defined by 
	a sequence of integer matrices
	$\cA=(\A{j}{\ell})_{1 \leq \ell < j \leq m-1}
		\in \prod_{1 \leq \ell < j \leq m-1} M_{n_j+1, n_{\ell}+1}(\Z)$.
	Every flag Bott manifold is a $\C P$-tower. Hence using Borel--Hirzebruch formula,
	the cohomology ring and the equivariant cohomology ring with respect to the torus action defined in Section~\ref{sec_torus_action_fBT} of $F_m$ can be computed. The explicit formula is given in~\cite{KKLS} in terms of $\cA$.
\end{remark}

\subsection{Tangent bundle of $F_m$}
In this subsection, we study the tangent bundle of a flag Bott manifold
using a principal connection of a principal bundle.
For more details, see \cite[Chapter 8, Addendum 3]{Spiv79}. 
For a principal $H$-bundle
$\pi \colon P \to B$, the \defi{vertical subbundle} $\mathcal{V}$ is defined to be $\mathcal{V} := \{ v \in TP \mid \pi_{\ast} v = 0\} \subset TP$.
If 
we let $o_p \colon H \to H(p)$ be the orbit map
which maps $H$ onto its orbit through $p \in P$, then we have
\begin{equation}\label{eq_vertical_subbdle}
\mathcal{V}_{p}  = (o_p)_{\ast} \text{Lie}(H).
\end{equation}
A \defi{principal connection} $\mathcal H$ is a subbundle of
$TP$ such that for $p \in P$,
\begin{itemize}
	\item $T_p P = \mathcal{V}_p \oplus \mathcal H_p$,
	\item $(\Phi_h)_{\ast} \mathcal H_p 
	= \mathcal H_{\Phi_h(p)}$ where $\Phi_h$ is the right action
	by $h \in H$, and
	\item $\mathcal H_p$ varies smoothly with respect to $p \in P$.
\end{itemize}
Because of the first property of principal connection, we have 
that $\pi_{{\ast}} (\mathcal H_p) = T_{\pi(p)} B$. 

For convenience, let $\mathbb{T}$ denote the product of compact tori 
$\prod_{j=1}^m T^{n_j+1}$. 
By Remark~\ref{rmk_compact_description}, 
an $m$-stage flag Bott manifold $F_m$ can be described as the orbit
of the right action in \eqref{eq_quotient_torus}, i.e., 
\(
F_m = \prod_{j=1}^m \U(n_j+1)/
\mathbb T.
\)
Since $\mathbb T$ acts freely on the space $\prod_{j=1}^m \U(n_j+1)$, we have the principal $\mathbb{T}$-bundle
\begin{equation}\label{eq_principal_bdle_Fm}
\prod_{j=1}^m \U(n_j+1) \stackrel{\pi}{\longrightarrow} F_m.
\end{equation}

We describe the vertical subbundle $\mathcal{V}$ of the above principal
bundle \eqref{eq_principal_bdle_Fm}. 
For $1\leq j \leq m$,
let $\mathfrak{u}(n_j+1)$, respectively  $\mathfrak{t}^{n_j+1}$,
denote the Lie algebra of $\U(n_j+1)$, respectively
$T^{n_j+1} \subset \U(n_j+1)$.
For a point $g = (g_1,\dots,g_m) \in \prod_{j=1}^m \U(n_j+1)$, 
define
\[
(L_{g^{-1}})_{\ast} :=
(L_{g_1^{-1}})_{\ast} \times \cdots 
\times (L_{g_m^{-1}})_{\ast} \colon
T_g\left(\prod_{j=1}^m \U(n_j+1) \right)
\to \bigoplus_{j=1}^m \mathfrak{u}(n_j+1),
\]
where $L_{g_j}$ is the left translation by $g_j$ for $1 \leq j \leq m$. 
Then $(L_{g^{-1}})_{\ast}$ is an isomorphism, so that we have the trivialization:
\[
\prod_{j=1}^m \U(n_j+1)
\times \bigoplus_{j=1}^m \mathfrak{u}(n_j+1) \cong T\left(\prod_{j=1}^m \U(n_j+1) \right).
\]
For the principal bundle \eqref{eq_principal_bdle_Fm}, 
it follows from \eqref{eq_vertical_subbdle}
that 
\(
\mathcal V_g = (o_g)_{\ast} \left( \bigoplus_{j=1}^m \mathfrak{t}^{n_j+1} \right),
\)
where $o_g \colon \mathbb{T} \to \mathbb{T}(g)$ is the orbit map. 
For a given $\underline{t} = (\underline{t}_1,
\dots,\underline{t}_m) \in \bigoplus_{j=1}^m \mathfrak{t}^{n_j+1}$, take a path 
\[
\gamma \colon (-\ep, \ep) \to \prod_{j=1}^m T^{n_j+1}, \quad s \mapsto (t_1(s),\dots,t_m(s))
\]
such that $\gamma(0)=\mathbf{1}$, $t_j(s) \in T^{n_j+1}$ and $\frac{d}{ds} \gamma(s)|_{s = 0}
= \underline{t}$. For a point $g \in \prod_{j=1}^m \U(n_j+1)$ 
and $\mathbf{t} \in \mathbb{T}$, 
let $g \cdot \mathbf{t}$ denote the right action of $\mathbb{T}$
in \eqref{eq_quotient_torus}. Then we have the following:
\begin{align*}
(L_{g^{-1}})_{\ast} (o_g)_{\ast} \underline{t}
&= \frac{d}{ds} L_{g^{-1}}(g \cdot \gamma(s))|_{s=0} \\
&= \frac{d}{ds}
\bigg( t_1(s), g_2^{-1} \big(\Lam{2}{1}(t_1(s)) \big)^{-1} g_2 t_2(s),
\dots,
g_m^{-1} \big(\Lam{m}{1}(t_1(s))\big)^{-1} \cdots \big(\Lam{m}{m-1}(t_{m-1}(s))\big)^{-1}
g_mt_m(s)\bigg) \bigg|_{s=0} \\
&= \bigg( 
\underline{t}_1, \underline{t}_2 - \text{Ad}_{g_2^{-1}}
\left(\A{2}{1}(\underline{t}_1) \right), \dots, \underline{t}_m
-\text{Ad}_{g_m^{-1}} \left(\A{m}{1}(\underline{t}_1) + \cdots +
\A{m}{m-1}(\underline{t}_{m-1} )\right)
\bigg).
\end{align*}
Here $\text{Ad}_{g}(X) = gXg^{-1}$, 
i.e., the usual adjoint representation of $\U(n_j+1)$ on 
$\mathfrak{u}(n_j+1)$.
Therefore we see that the vertical subbundle $\mathcal V$ is the image of the
injective map:
\[
\prod_{j=1}^m \U(n_j+1)  \times
\bigoplus_{j=1}^m \mathfrak{t}^{n_j+1}
 \to \prod_{j=1}^m  \U(n_j+1)\times
\bigoplus_{j=1}^m \mathfrak{u}(n_j+1),
\]
where $\left( (g_1,\dots,g_m), (\underline{t}_1,\dots,
\underline{t}_m) \right)$ maps to 
\[
\left( 
(g_1,\dots,g_m), 
\left(\underline{t}_1, \underline{t}_2 - \text{Ad}_{g_2^{-1}}
\left(\A{2}{1}(\underline{t}_1) \right), \dots,
\underline{t}_m - \text{Ad}(g_m^{-1}) 
\left(\A{m}{1}(\underline{t}_1) +\cdots+ \A{m}{m-1}(\underline{t}_{m-1})\right)
\right)
\right).
\]

Now we describe a principal connection. Let
$\mathfrak{m}_j \subset \mathfrak{u}(n_j+1)$ be the subspace of matrices with
the zeros along the diagonal. Then 
$\mathfrak{m}_j$ is invariant under the adjoint action of 
$T^{n_j+1}$, and $\mathfrak{m}_j \cap \mathfrak{t}^{n_j+1} = \{0\}$.
\begin{proposition}\label{prop_connection_Fm}
At the point $e := (e,\dots,e) \in \prod_{j=1}^m \U(n_j+1)$, choose the horizontal space
$\mathcal{H}_e := \bigoplus_{j=1}^m \mathfrak{m}_j
\subset \bigoplus_{j=1}^m \mathfrak{u}(n_j+1)$.
For a point $g = (g_1,\dots,g_m) \in \prod_{j=1}^m \U(n_j+1)$, define $\mathcal{H}_g \subset T_g(\prod_{j=1}^m \U(n_j+1))$ by 
\[
\mathcal{H}_g := \bigoplus_{j=1}^m (L_{g_j})_{\ast}\mathfrak{m}_j.
\]
Then $\mathcal{H}$ is a connection.
\end{proposition}
\begin{proof}
First we need to show that for each point $g \in \prod_{j=1}^m \U(n_j+1)$, we have that
$\mathcal{H}_g \oplus \mathcal{V}_g = T_g(\prod_{j=1}^m \U(n_j+1))$. We claim that $\mathcal{V}_g \cap 
\mathcal{H}_g = \{{0}\}$. Suppose that 
$(o_g)_{\ast} (\underline{t})$ is contained in 
$\mathcal{H}_g$ for some $\underline{t} 
= (\underline{t}_1,\dots,\underline{t}_m) \in 
\bigoplus_{j=1}^m \mathfrak{t}^{n_j+1}$. This implies
that
\[
\left(\underline{t}_1, \underline{t}_2
- \text{Ad}_{g_2^{-1}}\left(\A{2}{1}(\underline{t}_1)\right), 
\dots, \underline{t}_m - \text{Ad}_{g_m^{-1}}
\left(\A{m}{1}(\underline{t}_1) + \cdots + \A{m}{m-1}(\underline{t}_{m-1})\right)
\right) \in \bigoplus_{j=1}^m \mathfrak{m}_j.
\]
In particular, $\underline{t}_1 \in \mathfrak{m}_1$, but it is also
contained in $\mathfrak{t}^{n_1+1}$. Since $\mathfrak{m}_1 \cap \mathfrak{t}^{n_1+1} 
= \{0\}$, we have that $\underline{t}_1 = 0$. Continuing in this manner 
we conclude that $\mathcal{V}_g \cap \mathcal{H}_g = \{0\}$,
and hence by the dimension reason, we have 
$\mathcal{H}_g \oplus \mathcal{V}_g = T_g(\prod_{j=1}^m \U(n_j+1))$.

Secondly, define the map $\Phi_{\mathbf{t}} \colon 
\prod_{j=1}^m \U(n_j+1)\to \prod_{j=1}^m \U(n_j+1)$ as the right translation by $\mathbf{t}$
as defined in \eqref{eq_quotient_torus}.
For an element $\mathbf{t} = (t_1,\dots,t_m) \in \prod_{j=1}^m T^{n_j+1}$, we claim that $(\Phi_{\mathbf{t}})_{\ast} \mathcal{H}_g
= \mathcal{H}_{\Phi_{\mathbf{t}}(g)}$. 
For any $(x_1,\dots,x_m) \in \prod_{j=1}^m \U(n_j+1)$, we have the following:
\begin{equation}\label{eq_connection}
\begin{split}
(\Phi_{\mathbf{t}} &\circ L_g)(x_1,\dots,x_m) \\
&= \Phi_{\mathbf{t}}(g_1x_1,\dots,g_mx_m) \\
&= \Big(g_1x_1t_1, \big(\Lam{2}{1}(t_1)\big)^{-1}g_2x_2t_2, \dots, 
\big(\Lam{m}{1}(t_1)\big)^{-1} \cdots \big(\Lam{m}{m-1}(t_{m-1})\big)^{-1} g_mx_m t_m\Big) \\
&= \Big(g_1t_1(t_1^{-1}x_1t_1),
\big(\Lam{2}{1}(t_1)\big)^{-1}g_2t_2(t_2^{-1}x_2t_2),\dots, 
\big(\Lam{m}{1}(t_1)\big)^{-1}\cdots \big(\Lam{m}{m-1}(t_{m-1})\big)^{-1} g_mt_m
(t_m^{-1}x_mt_m) \Big)\\
&= L_{\Phi_{\mathbf{t}}(g)}(t_1^{-1}x_1t_1,\dots,t_m^{-1}x_mt_m).
\end{split}
\end{equation}
This gives $(\Phi_{\mathbf{t}})_{\ast} \mathcal{H}_g
= \mathcal{H}_{\Phi_{\mathbf{t}}(g)}$ since 
$\mathfrak{m}_j$ is invariant under the adjoint action of 
$T^{n_j+1}$ for $1 \leq j \leq m$.

Finally since the left multiplication varies smoothly with
$(g_1,\dots,g_m) \in \prod_{j=1}^m \U(n_j+1)$,
this defines a connection. 
\end{proof}

As a corollary of Proposition~\ref{prop_connection_Fm} we have
the following description of the tangent bundle of $F_m$:
\begin{proposition}\label{prop_tangent_bdle}
The tangent bundle of $F_m$ is isomorphic to 
\[
\prod_{j=1}^m \U(n_j+1) \times_{\mathbb T}
\bigoplus_{j=1}^m \mathfrak{m}_j,
\]
where 
the following elements are identified:
\begin{equation}\label{eq_tan_bdle}
(g_1,\dots,g_m; X_1,\dots,X_m) \sim 
\left( (g_1,\dots,g_m) \cdot (t_1,\dots,t_m) ;
\text{Ad}_{t_1^{-1}}X_1,\dots,\text{Ad}_{t_m^{-1}}X_m \right)
\end{equation}
for $(t_1,\dots,t_m) \in \mathbb T$.
Here $(g_1,\dots,g_m) \cdot (t_1,\dots,t_m)$ is as defined in 
\eqref{eq_quotient_torus}.
\end{proposition}
\begin{proof}
Let $\varphi \colon \prod_{j=1}^m \U(n_j+1) \times_{\mathbb T}
\bigoplus_{j=1}^m \mathfrak{m}_j
\to TF_m$ be the map defined by 
$\varphi([g;X]) = ([g], (\pi_{{\ast}} \circ (L_g)_{\ast}(X)))$.
We claim that the map $\varphi$ is a bundle isomorphism. 
Because of the property of a principal connection and 
by the definition of $\mathcal H$, we have that
$(\pi_{{\ast}} \circ (L_g)_{\ast}(X)) \in T_{[g]} F_m$. 
It is enough to check that the map $\varphi$ is well-defined. 
For $\mathbf{t}=(t_1,\dots,t_m) \in \mathbb T$,
an element $\left[ \Phi_{\mathbf{t}}(g) ; \text{Ad}_{t_1^{-1}}X_1,
\dots, \text{Ad}_{t_m^{-1}}X_m \right]$ maps to 
$\left([\Phi_{\mathbf{t}}(g)], (\pi_{{\ast}} \circ (L_{\Phi_{\mathbf{t}}(g)})_{{\ast}})(
\text{Ad}_{t_1^{-1}}X_1,
\dots, \text{Ad}_{t_m^{-1}}X_m) \right)$.
From \eqref{eq_connection}, we can see that 
\[
(L_{\Phi_{\mathbf{t}}(g)})_{\ast} \left(
\text{Ad}_{t_1^{-1}}X_1,
\dots, \text{Ad}_{t_m^{-1}}X_m \right)
= (\Phi_{\mathbf{t}})_{\ast} \circ (L_g)_{\ast} (X_1,\dots,X_m).
\]
Because $\pi \circ \Phi_{\mathbf{t}} = \pi$, we have that
$\pi_{\ast} \circ (\Phi_{\mathbf{t}})_{\ast} = \pi_{\ast}$. This
implies that the map $\varphi$ is well-defined.
\end{proof}

\section{GKM descriptions of flag Bott manifolds}
\label{sec_GKM_desc_of_fBM}
Let $F_m$ be an $m$-stage flag Bott manifold.
In Subsection~\ref{sec_torus_action_fBT}, we define the canonical 
torus action on $F_m$
and by studying this action more carefully, 
we conclude that a flag Bott manifold $F_m$ is a GKM manifold
with the canonical action in Theorem~\ref{thm_Fm_is_GKM}. 

\subsection{Torus actions}
\label{sec_torus_action_fBT}
Let $F_m$ be an $m$-stage flag Bott manifold. 
For $1 \leq j \leq m$,
let $\mathbb H = \prod_{\ell=1}^m H_{\GL(n_{\ell}+1)}$
act on $F_j$ by 
\[
(h_1,\dots,h_m) \cdot [g_1,\dots,g_j]
:= [h_1g_1,\dots, h_jg_j] 
\]
for $(h_1,\dots,h_m) \in \mathbb H$ and $[g_1,\dots,g_j] \in F_j$.
Then $F_j \to F_{j-1}$ is $\mathbb H$-equivariant fiber bundle.
For notational convenience, we write
\begin{equation}\label{eq_n_nj_sum}
n := n_1 + \cdots + n_m.
\end{equation}
Therefore $\sum_{j=1}^m (n_j+1) = n+m$.
Let $\mathbb T \subset \mathbb H$ be the compact torus of real dimension
$n+m$. 
Note that the torus $\mathbb H$ acts holomorphically but does not 
act effectively on $F_m$. 
If we write  
$h_j = \text{diag}(h_{j,1},\dots,h_{j,n_j+1}) \in \GL(n_j+1)$,
the subtorus
\[
\mathbf H:= \{(h_1,\dots,h_m) \in \mathbb H \mid
h_{1,n_1+1} = \cdots = h_{m,n_m+1} = 1 \} \cong (\Cstar)^{n}
\]
acts effectively on $F_m$. 
Let $\bf T \subset \mathbf H$ denote the compact torus 
of real dimension $n$.
In this paper, we call the action of these tori the 
\defi{canonical }$\mathbb H$ ($\mathbb T$, $\mathbf H$ or $\mathbf T$)-\defi{action} 
on $F_m$. For a space $X$ with a $G$-action, we write $(X,G)$
for this $G$-space $X$ when we need to emphasize the acting group.
\begin{remark}\label{rmk_dim_of_torus_and_fBT}
	The complex dimension of an
	$m$-stage flag Bott manifold $F_m$ is 
	$	\frac{n_1(n_1+1)}{2} + \cdots + \frac{n_m(n_m+1)}{2}$
	while the complex dimension of the torus $\mathbf H$, which acts
	effectively on the manifold $F_m$, is $n=n_1 + \cdots + n_m$.
	They are equal if and only if $n_1 = \cdots = n_m = 1$, 
which is the case when a flag Bott manifold is a Bott manifold (see Example~\ref{example_fBT}(3)).
The highest dimension of a torus which can act on $F_m$ effectively is studied in~\cite{Kuro17}.
\hfill \qed
\end{remark}

\begin{example}\label{exampel_torus_action_flag}
	A 1-stage flag Bott manifold is the flag manifold 
	$\flag(n+1)=\GL(n+1)/B_{\GL(n+1)}$. Then the canonical torus action of $\mathbb H 
	= H_{\GL(n+1)}$ on the flag manifold $\flag(n+1)$
	is the left multiplication. 
 \hfill \qed
\end{example}
It is well known that the fixed point set $\flag(n+1)^{\mathbb H}$
can be identified with the symmetric group $\mathfrak{S}_{n+1}$ (see~\cite[Subsection 10.1]{Fult97}).	
	For a given permutation	$w \in \mathfrak{S}_{n+1}$, 
let $\dw$ denote the column permutation matrix, 
i.e., $\dw$ is an element in $\GL(n+1)$ whose
$(w(k),k)$-entries are 1 for $1 \leq k \leq n+1$, 
and all others are zero. For instance, the permutation $w = (231) \in \mathfrak{S}_3$ corresponds to the matrix
\begin{equation}\label{eq_column_permutation}
\dot{w} = \begin{bmatrix}
0 & 0 & 1 \\ 1 & 0 & 0 \\
0 & 1 & 0
\end{bmatrix} \in \GL(3).
\end{equation}
Here we use the one-line notation, i.e., $w(1)=2,w(2)=3$, and $w(3)=1$.	
Then the fixed point set is 
$\{ [\dw] \in \GL(n+1)/B_{\GL(n+1)} \mid w 
\in \mathfrak{S}_{n+1} \}$.
This property can be extended to the canonical action of $\mathbb H$ on 
$F_m$. 
\begin{proposition}\label{prop_fixed_points}
	Let $F_m$ be an $m$-stage flag Bott manifold with the action of
	$\mathbb H$. Then the fixed point set is identified with 
	the product of symmetric groups $\prod_{j=1}^m \mathfrak{S}_{n_j+1}$. More precisely, for 
	an element $(w_1,\dots,w_m) \in \prod_{j=1}^m \mathfrak{S}_{n_j+1}$, the corresponding fixed point
	in $F_m$ is $[\dw_1,\dots,\dw_m]$, where $\dw_j
	\in \GL(n_j+1)$ is the column
	permutation matrix of $w_j$. \hfill \qed
\end{proposition}

\subsection{Tangential representations of flag Bott manifolds}
\label{sec_tang_rep_of_fBT}
In this subsection, we study the tangential representations
of a flag Bott manifold $F_m$ at the fixed points
corresponding to the (noneffective) canonical action of $\mathbb T$ 
in Proposition~\ref{prop_fBT_is_GKM}.
Recall the definition of GKM manifolds from \cite{GKM98} and \cite{GuZa01}.
\begin{definition}
Let $T$ be the compact torus of dimension $n$, 
$\mathfrak{t}$ its Lie algebra, and $M$ a compact 
manifold of real dimension $2d$ with an effective action of $T$. We say that a pair
$(M,T)$ is a \defi{GKM manifold} if
\begin{enumerate}
	\item the fixed point set $M^T$ is finite,
	\item $M$ possesses a $T$-invariant almost-complex structure, and
	\item for every $p \in M^T$, the weights
	$
	\{ \alpha_{i,p} \in \mathfrak{t}^{\ast} \mid   1 \leq i \leq d\}
	$
	of the isotropy representation $T_p M$ of $T$
	are pairwise linearly independent. 
\end{enumerate}
\end{definition}
By considering the effective canonical action of $\mathbf T$ on $F_m$,
we will see that $(F_m, \mathbf T)$ is a GKM manifold in Theorem~\ref{thm_Fm_is_GKM}.
For this, we first need to 
compute the tangential representations of a flag Bott manifold
$F_m$ at fixed points. From Proposition~\ref{prop_tangent_bdle},
the tangent bundle $T F_m$
of a flag Bott manifold $F_m$ 
is isomorphic to 
\[
\prod_{j=1}^m \U(n_j+1) \times_{\mathbb T}
\bigoplus_{j=1}^m \mathfrak{m}_j,
\]
where $\mathfrak{m}_{j} \subset \mathfrak{u}(n_j+1)$ 
is the subspace of matrices with the zero diagonals for
$1 \leq j \leq m$.
For an element $(w_1,\dots,w_m) \in \prod_{j=1}^m \mathfrak{S}_{n_j+1}$, 
the corresponding fixed point in the flag Bott manifold $F_m$ is 
$\dw:= [\dot{w}_1,
\dots,\dot{w}_m]$. 
To describe the tangential representation $T_{\dw} F_m$ of $\mathbb T$, 
it is enough to find homomorphisms
$f_j \colon \mathbb T \to T^{n_j+1}$ satisfying that for $1 \leq j \leq m$
\[
[t_1\dot{w}_1,\dots,t_m\dot{w}_m ; X_1,\dots,X_m] 
= [\dot{w}_1,\dots,\dot{w}_m;
\Ad_{f_1(t_1,\dots,t_m)} X_1, \Ad_{f_2(t_1,\dots,t_m)} X_2,
\dots, \Ad_{f_m(t_1,\dots,t_m)} X_m ].
\]

Before computing the homomorphisms $f_j$, 
let us recall the adjoint action of $\mathbb T$
on $\mathfrak{m}_j$. Let $E_{(r,s)}$ be an element of $\mathfrak{gl}(n_j+1)$ whose 
$(r,s)$-entry is 1 and all others are zero.
Now we have 
$\mathfrak{m}_j \cong \text{span}_{\mathbb{C}}\{ 
z E_{(r,s)} + (-\overline{z})E_{(s,r)} \mid
z \in \C, 1 \leq s < r \leq n_j+1\}$.
We denote the standard basis of 
$\text{Lie}(\mathbb{T})^{\ast} \cong \mathbb{R}^{\sum_{j=1}^m(n_j+1)}$ by
\begin{equation}\label{eq_Lie_basis}
\{\ep_{1,1}^{\ast},
\dots,\ep_{1,n_1+1}^{\ast},\dots,\ep_{m,1}^{\ast},
\dots,\ep_{m,n_m+1}^{\ast}\}.
\end{equation}
With respect to this basis, 
let $A$ be the integer matrix of size 
$(n_j+1) \times (n+m)$
whose row vectors $\mathbf{c}_{j,1},\dots,\mathbf{c}_{j,n_j+1}$
are weights of the homomorphism $f_j$, so that
for an element $\mathbf{t} \in \mathbb T$,
\begin{equation}\label{eq_def_of_fj}
f_j \colon \mathbf{t} \mapsto \diag\left( 
\mathbf{t}^{\mathbf{c}_{j,1}},\dots,\mathbf{t}^{\mathbf{c}_{j,n_j+1}}
\right).
\end{equation}
Since $\Ad_{f_j(\mathbf{t})}E_{(r,s)}
= \mathbf{t}^{\mathbf{c}_{j,r} - \mathbf{c}_{j,s}} E_{(r,s)}$,
using the weight vectors $\{ \mathbf c_{j,k} \}$, we can describe that
\[
\mathfrak{m}_j \cong \bigoplus_{1 \leq s < r \leq n_j+1}
V(\mathbf c_{j,r} - \mathbf c_{j,s}),
\]
where $V(\mathbf c_{j,r} - \mathbf c_{j,s})$ is the 1-dimensional
$\mathbb{T}$-representation with the weight $\mathbf c_{j,r} - \mathbf c_{j,s} \in \bigoplus_{j=1}^m \Z^{n_j+1}$.
For an integer matrix $A$, we define 
\[
V(A) := \bigoplus_{1 \leq s < r \leq n_j+1}
V(\mathbf c_{j,r} - \mathbf c_{j,s}).
\]
Using this notation, we have the following proposition
whose proof will be given at the end of this subsection.
\begin{proposition}\label{prop_fBT_is_GKM}
Let $F_m$ be the $m$-stage flag Bott manifold determined by 
a set of integer matrices
	$(\A{j}{\ell})_{1 \leq \ell < j \leq m-1}
\in \prod_{1 \leq \ell < j \leq m-1} M_{n_j+1, n_{\ell}+1}(\Z)$.
Consider the (noneffective) canonical $\mathbb T$-action on $F_m$.
	For a fixed point $\dw = [\dw_1,\dots,\dw_m] \in F_m$,
	the tangential $\mathbb T$-representation is
	$T_{\dw} F_m \cong \bigoplus_{j=1}^m \mathfrak{m}_j$, where
	\begin{equation}\label{eq_mj}
	\mathfrak{m}_j \cong V \left(\left[X_1^{(j)}\ \ 
	X^{(j)}_2 \ \ \cdots \ \ X^{(j)}_{j-1} \ \ B_j \ \ O \ \ \cdots \ \ O
	\right] \right).
	\end{equation}
	Here $X^{(j)}_{\ell}$ is the
	matrix of size $(n_j+1) \times (n_{\ell}+1)$ defined by
	\begin{equation}\label{eq_def_of_Xjl}
	X^{(j)}_{\ell} = \sum_{\ell < i_1 < \cdots < i_r < j }
	\left(B_j A^{(j)}_{i_r} \right) \left(B_{i_r} A^{(i_{r})}_{i_{r-1}} \right) \cdots 
	\left(B_{i_1} A^{(i_1)}_{\ell} \right)B_{\ell}
	+ B_j \A{j}{\ell} B_{\ell}
	\quad \text{ for }
	1 \leq \ell < j \leq m,
	\end{equation}
	and $B_j$ is the row permutation matrix corresponding to $w_j$, i.e.,
	$B_j = (\dw_j)^T$.
	Furthermore, the weights of the isotropy representation 
	of $\mathbb T$ on $T_{\dw} F_m$ are pairwise linearly independent.
\end{proposition}
By considering the effective canonical action of $\bf T$ on $F_m$, 
the fixed point set is finite because of Proposition~\ref{prop_fixed_points}.
Also the canonical action of $\mathbf T$ on $F_m$ is holomorphic  
(see~Subsection~\ref{sec_torus_action_fBT}).
As a corollary of Proposition~\ref{prop_fBT_is_GKM}, 
we have the following theorem. 
\begin{theorem}\label{thm_Fm_is_GKM}
	Let $F_m$ be an $m$-stage flag Bott manifold 
	with the effective canonical action of $\bf T$. 
	Then $(F_m, \bf T)$ is a GKM manifold. 
\end{theorem}
\begin{example}\label{example_GKM_F1}
	Suppose that the flag Bott manifold $F_1$ is $\flag(3)$. 
	With the canonical action of the torus $\mathbb T = (S^1)^{3}$, 
	there are six fixed points $\{ [\dw] \mid w \in \mathfrak{S}_3\}$. 
	Let $\{\ep_1^{\ast},\ep_2^{\ast},\ep_3^{\ast}\}$ be the standard basis of $\Lie((S^1)^3)^{\ast} \cong \mathbb{R}^3$. 
	Consider an element $\dw$ in $\GL(3)$ corresponding to the permutation
	$w = (2 3 1) \in \mathfrak{S}_3$. 
	Then the row permutation matrix $B$ is 
	\[
	\begin{bmatrix}
	0 & 1 & 0 \\
	0 & 0 & 1 \\
	1 & 0 & 0
	\end{bmatrix},
	\]
	which is the transpose of the column permutation matrix $\dot{w}$ in~\eqref{eq_column_permutation}.
	Then we have the following tangential representation:
	\[
	\reqnomode
	T_{[\dw]}F_1 =	\mathfrak{m}_1 \cong	V(B)
	= V(\ep_3^{\ast} - \ep_2^{\ast}) \oplus V(\ep_1^{\ast} - \ep_3^{\ast}) \oplus
	V(\ep_1^{\ast} - \ep_2^{\ast}).	\tag*{\qed}
	\]
\end{example}
\begin{example}\label{example_GKM_F2_2}
	Consider a flag Bott tower $F_2$ of height 2 defined by 
	the integer matrix $\A{2}{1} = \begin{bmatrix}
	c_1 & c_2 & 0 \\ 0 & 0 & 0
	\end{bmatrix}$. Then $F_2$ is a 
	$\C P^1$-bundle over $\flag(3)$. 
	The manifold $F_2$ has the action of $(S^1)^3 \times (S^1)^2$,
	and there are 12 fixed points $\left\{[\dw_1,\dw_2] \mid w_1 \in \mathfrak{S}_3, w_2 \in \mathfrak{S}_2\right\}$. Let $\{\ep_{1,1}^{\ast},\ep_{1,2}^{\ast},\ep_{1,3}^{\ast},\ep_{2,1}^{\ast},
		\ep_{2,2}^{\ast}\}$ be the standard basis of 
	$\Lie ((S^1)^3 \times (S^1)^2)^{\ast} \cong \mathbb{R}^3 \oplus
	\mathbb{R}^2$.
	Consider the point $\dw = [\dw_1,\dw_2]$ where
	$w_1 = e$ and $w_2 = (2 1)$. Then the corresponding row
	permutation matrices are
	\[
	B_1 = I_3 = \begin{bmatrix}
	1 & 0 & 0 \\ 0 & 1 & 0 \\ 0 & 0 & 1
	\end{bmatrix}, \quad
	B_2 = \begin{bmatrix}
	0 & 1 \\ 1 & 0
	\end{bmatrix}.
	\]
	Hence the matrix $\X{2}{1}$ is 
	\[
	\X{2}{1} = B_2 \A{2}{1} B_1 = 
	\begin{bmatrix}
	0 & 1 \\ 1 & 0
	\end{bmatrix}
	\begin{bmatrix}
	c_1 & c_2 & 0 \\
	0 & 0 & 0
	\end{bmatrix}
	=\begin{bmatrix}
	0 & 0 & 0 \\ c_1 & c_2 & 0
	\end{bmatrix}.
	\]
	The tangential representation at the point $\dw$ can be computed
	as follows:
{{	\reqnomode
	\begin{align*}
	T_{\dw} F_2 = \mathfrak{m}_1 \oplus \mathfrak{m}_2
	& \cong V([I_3 \ O ]) \oplus V\left([\X{2}{1} \ B_2 ] \right) \\
	& = V \left( 
	\begin{bmatrix}
	1 & 0 & 0 & 0 & 0 \\
	0 & 1 & 0 & 0 & 0 \\
	0 & 0 & 1 & 0 & 0
	\end{bmatrix}
	\right)
	\oplus
	V \left( 
	\begin{bmatrix}
	0 & 0 & 0 & 0 & 1 \\
	c_1 & c_2 & 0 & 1 & 0
	\end{bmatrix}
	\right) \\
	&
	= V(\ep_{1,2}^{\ast} - \ep_{1,1}^{\ast}) \oplus V(\ep_{1,3}^{\ast} - \ep_{1,2}^{\ast})
	\oplus V(\ep_{1,3}^{\ast} - \ep_{1,1}^{\ast}) 
\oplus 
	V(c_1 \ep_{1,1}^{\ast} + c_2 \ep_{1,2}^{\ast} + \ep_{2,1}^{\ast} - \ep_{2,2}^{\ast}).  \tag*{\qed}
	\end{align*}}}
\end{example}

\begin{example}\label{example_GKM_F3}
Consider a flag Bott tower of height 3 with $n_1 = 2$,
$n_2 = 1$, and $n_3 = 1$ which is defined by
\[
\A{2}{1} = \begin{bmatrix}
1 & 2 & 0 \\ 0 & 0 & 0
\end{bmatrix}, \quad
\A{3}{1} = \begin{bmatrix}
3 & 4 & 0 \\ 0 & 0 & 0
\end{bmatrix}, \quad
\A{3}{2} = \begin{bmatrix}
5 & 0 \\ 0 & 0
\end{bmatrix}.
\]
Then the flag Bott manifold $F_3$ has the action of 
$(S^1)^3 \times (S^1)^2 \times (S^1)^2$, and the set of 
fixed points is $\{ [\dw_1,\dw_2,\dw_3] \mid
w_1 \in \mathfrak{S}_3, w_2,w_3 \in \mathfrak{S}_2\}$. 
Denote the standard basis of $\Lie((S^1)^3 \times (S^1)^2 
\times (S^1)^2) \cong \mathbb{R}^3 \oplus \mathbb{R}^2 \oplus
\mathbb{R}^2$ by $\{\ep_{1,1}^{\ast}, \ep_{1,2}^{\ast}, \ep_{1,3}^{\ast}, 
\ep_{2,1}^{\ast}, \ep_{2,2}^{\ast}, \ep_{3,1}^{\ast}, \ep_{3,2}^{\ast}\}$. 
Consider the fixed point $\dw = [\dw_1, \dw_2, \dw_3 ]$ where
$w_1 = (3 1 2)$, $w_2 = e$, and $w_3 = (2 1)$. The corresponding
row permutation matrices are
\[
B_1 = \begin{bmatrix}
0 & 0 & 1 \\
1 & 0 & 0 \\
0 & 1 & 0 
\end{bmatrix}, \quad
B_2 = \begin{bmatrix}
1 & 0 \\ 0 & 1
\end{bmatrix}, \quad
B_3 = \begin{bmatrix}
0 & 1  \\ 1 & 0
\end{bmatrix}.
\]
We have the following computations of 
$\X{2}{1}, \X{3}{1}, \X{3}{2}$:
\begin{align*}
\X{2}{1} &= B_2 \A{2}{1} B_1
= \begin{bmatrix}
2 & 0 & 1\\ 0 & 0 & 0
\end{bmatrix}, \\
\X{3}{1} &= B_3\A{3}{2}B_2\A{2}{1}B_1 + B_3\A{3}{1}B_1 
= \begin{bmatrix}
0 & 0 & 0 \\ 14 & 0 & 8
\end{bmatrix}, \\
\X{3}{2} &= B_3 \A{3}{2} B_2 
= \begin{bmatrix}
0 & 0 \\ 5 & 0
\end{bmatrix}.
\end{align*}
The tangential representation at the point $\dw$
can be computed as follows:
{{\reqnomode
\begin{align*}
T_{\dw} F_3 &= \mathfrak{m}_1 \oplus \mathfrak{m}_2  \oplus \mathfrak{m}_3\\
&\cong V\left([B_1 \ O \ O] \right) 
\oplus V \left([\X{2}{1} \ B_2 \ O] \right) \oplus
V \left([\X{3}{1} \ \X{3}{2} \ B_3] \right) \\
& = 
V \left( 
\begin{bmatrix}
0 & 0 & 1 & 0 & 0 & 0 & 0 \\ 1 & 0 & 0 & 0 & 0 & 0 & 0 \\
0 & 1 & 0 & 0 & 0 & 0 & 0
\end{bmatrix}
\right)
\oplus V \left( 
\begin{bmatrix}
2 & 0 & 1 & 1 & 0 & 0 & 0 \\
0 & 0 & 0 & 0 & 1 & 0 & 0
\end{bmatrix}
\right) \oplus 
V \left(
\begin{bmatrix}
0 & 0 & 0 & 0 & 0 & 0 & 1 \\
14 & 0 & 8 & 5 & 0 & 1 & 0
\end{bmatrix} 
\right) \\
&= V(\ep_{1,1}^{\ast} - \ep_{1,3}^{\ast}) \oplus V(\ep_{1,2}^{\ast}-\ep_{1,3}^{\ast})
\oplus V(\ep_{1,2}^{\ast} - \ep_{1,1}^{\ast}) \\
& \quad \quad \oplus
V(-2\ep_{1,1}^{\ast}-\ep_{1,3}^{\ast} -\ep_{2,1}^{\ast} + \ep_{2,2}^{\ast})
\oplus V(14\ep_{1,1}^{\ast} + 8\ep_{1,3}^{\ast} + 5 \ep_{2,1}^{\ast} + \ep_{3,1}^{\ast} 
- \ep_{3,2}^{\ast}). \tag*{\qed}
\end{align*}}}
\end{example}
Before presenting the proof of Proposition~\ref{prop_fBT_is_GKM}, we give a lemma which is directly induced by the definition of $\X{j}{\ell}$ in~\eqref{eq_def_of_Xjl}.
\begin{lemma}\label{lemma_Xjl_summand}
	The matrix $\X{j}{\ell}$ satisfies the following equality.
		\[
	X^{(j)}_{\ell} 
	= B_j A^{(j)}_{j-1} X^{(j-1)}_{\ell} 
	+ B_{j} A^{(j)}_{j-2} X^{(j-2)}_{\ell} 
	+ \cdots 
	+ B_{j} A^{(j)}_{\ell+1} X^{(\ell+1)}_{\ell} 
	+ B_{j} A^{(j)}_{\ell} B_{\ell}.
	\]
\end{lemma}

\begin{proof}[Proof of Proposition~\ref{prop_fBT_is_GKM}]
We first note that for any $t_j =
\text{diag}(t_{j,1},\dots,t_{j,n_j+1}) \in T^{n_j+1}
\subset \U(n_j+1)$, we have that
$\dot{w}_j^{-1} t_j \dot{w}_j 
= \text{diag}(t_{j,w_j(1)}, t_{j,w_j(2)},\dots,t_{j,w_j(n_j+1)})
\in T^{n_j+1}$.
Let $\w_j$ denote a homomorphism $T^{n_j+1} \to T^{n_j+1}$ 
define by $\w_j(t_j):=\dw_j^{-1} t_j\dw_j$.
Then we have that 
\begin{equation}\label{eq_tilde_w}
t_j \dw_j= \dw_j \dw_j^{-1} t_j \dw_j
= \dw_j \w_j(t_j).
\end{equation}
For the row permutation matrix $B_j = (\dw)^T$, we have that
$B_j (t_{j,1},\dots,t_{j,n_j+1})^T
= (t_{j,w_j(1)},\dots,t_{j,w_j(n_j+1)})^T$.
Hence $B_j$ is the matrix for the homomorphism $\w_j\colon T^{n_j+1} \to T^{n_j+1}$.

	Consider the case when $j = 1$. 	
	Then we can get 
	\begin{equation}\label{eq_find_fj_stage_1}
	\begin{split}
	&[t_1\dw_1,\dots,t_m \dw_m ;X_1,\dots,X_m]
	= [\dw_1 \w_1(t_1),t_2\dw_2,\dots,t_m \dw_m;
	X_1,\dots,X_m] \quad (\text{by } \eqref{eq_tilde_w})\\
	&\quad 
	= [(\dw_1, \Lam{2}{1}(\w_1(t_1)) t_2 \dw_2,
	\dots, \Lam{m}{1}(\w_1(t_1))t_m \dw_m)
	\cdot(\w_1(t_1),1,\dots,1); 
	 X_1,\dots,X_m] 
	\quad (\text{by } \eqref{eq_quotient_torus})\\
	&\quad = [\dw_1, (\Lambda^{(2)}_1 \circ \w_1)(t_1)
	t_2\dw_2,\dots,
	(\Lambda^{(m)}_1 \circ \w_1)(t_1)
	t_m \dw_m; 
	\Ad_{\w_1 (t_1)}X_1, X_2, \dots, X_m] 
	\quad (\text{by } \eqref{eq_tan_bdle}).
	\end{split}
	\end{equation}
	Therefore the homomorphism
$	f_1\colon \mathbb T \to T^{n_1+1}$ in \eqref{eq_def_of_fj} is given by 
$	(t_1,\dots,t_m) \mapsto \w_1(t_1)$,
	and
\[
	\mathfrak{m}_1 \cong
	V\left([B_1 \ \ O \ \ \cdots \ \ O] \right).
\]
	Hence the proposition holds for $j = 1$.
	
We continue the similar computation to \eqref{eq_find_fj_stage_1} for
the second coordinate as follows.
 For $\mathbf{t} = (t_1,\dots,t_m) \in \mathbb T$,
\begin{align*}
&[t_1\dw_1,t_2\dw_2, t_3\dw_3, \dots,t_m \dw_m ;X_1,\dots,X_m] \\
	&\quad = [\dw_1, (\Lambda^{(2)}_1 \circ \w_1)(t_1)
t_2\dw_2,(\Lam{3}{1} \circ \w_1)(t_1)
t_3 \dw_3,\dots,
(\Lambda^{(m)}_1 \circ \w_1)(t_1)
t_m \dw_m; 
\Ad_{\w_1 (t_1)}X_1, X_2, \dots, X_m] \quad
(\text{by } \eqref{eq_find_fj_stage_1})\\
&\quad = 
[\dw_1, \Lam{2}{1}(f_1(\mathbf{t})) t_2 \dw_2,
\Lam{3}{1}(f_1(\mathbf{t})) t_3 \dw_3,
\dots, \Lam{m}{1}(f_1(\mathbf{t})) t_m \dw_m;  
\text{Ad}_{f_1(\mathbf{t})} X_1,X_2,\dots,X_m] \\
&\qquad\quad (\text{by substituting } \w_1(t_1)=f_1(\mathbf{t}))\\
& \quad =
[\dw_1, \dw_2 \w_2(\Lam{2}{1}(f_1(\mathbf{t})) t_2), 
\Lam{3}{1}(f_1(\mathbf{t})) t_3 \dw_3,\dots, \Lam{m}{1}(f_1(\mathbf{t})) t_m \dw_m; 
\text{Ad}_{f_1(\mathbf{t})} X_1,X_2,\dots,X_m]
\quad (\text{by } \eqref{eq_tilde_w}) \\
& \quad = 
[\dw_1, \dw_2 f_2(\mathbf{t}),\Lam{3}{1}(f_1(\mathbf{t})) t_3 \dw_3,\dots, \Lam{m}{1}(f_1(\mathbf{t})) t_m \dw_m; 
\text{Ad}_{f_1(\mathbf{t})} X_1,X_2,\dots,X_m] \\
& \qquad\quad (\text{by letting } f_2(\mathbf{t}) = 
\w_2(\Lam{2}{1}(f_1(\mathbf{t}))t_2)) \\
&\quad = 
[\dw_1, \dw_2, \Lam{3}{2}(f_2(\mathbf{t})) \Lam{3}{1}(f_1(\mathbf{t})) t_3 \dw_3,
\dots, \Lam{m}{2}(f_2(\mathbf{t})) \Lam{m}{1}(f_1(\mathbf{t})) t_m \dw_m; 
\text{Ad}_{f_1(\mathbf{t})}X_1, \text{Ad}_{f_2(\mathbf{t})} X_2, X_3, \dots,X_m ] \\
& \qquad \quad (\text{by } \eqref{eq_tan_bdle}).
\end{align*}
Continuing this process, 
we may assume that $f_1,\dots,f_{j-1}$ can be defined so that
the following is satisfies for $j>1$:
	\begin{align*}
	&[t_1\dw_1,\dots,t_j\dw_j,\ldots; X_1,\dots,X_j, \ldots] \\	
	&\quad = [\dw_1,\dots,\dw_{j-1},
	\Lam{j}{j-1}(f_{j-1}(\mathbf{t})) \Lam{j}{j-2}(f_{j-2}(\mathbf{t})) \cdots
	\Lam{j}{1}(f_{1}(\mathbf{t})) t_j \dw_j, \ldots;
	\Ad_{f_1(\mathbf{t})}X_1,\dots,\Ad_{f_{j-1}(\mathbf{t})}X_{j-1},X_j,\ldots].
	\end{align*}
	We now define $f_j$. 
	By considering $\Lam{j}{j-1}(f_{j-1}(\mathbf{t})) \Lam{j}{j-2}(f_{j-2}(\mathbf{t})) \cdots
	\Lam{j}{1}(f_{1}(\mathbf{t})) t_j \dw_j$, we get the following:
	\begin{align*}
	&\Lam{j}{j-1}(f_{j-1}(\mathbf{t})) \Lam{j}{j-2}(f_{j-2}(\mathbf{t})) \cdots
	\Lam{j}{1}(f_{1}(\mathbf{t})) t_j \dw_j \\
	& \quad \quad = \dw_j 
	\w_j\left(\Lam{j}{j-1}(f_{j-1}(\mathbf{t})) \Lam{j}{j-2}(f_{j-2}(\mathbf{t})) \cdots
	\Lam{j}{1}(f_{1}(\mathbf{t})) t_j\right) \quad (\text{by } \eqref{eq_tilde_w})\\
	& \quad \quad = \dw_j 
	\left( \w_j \circ \Lam{j}{j-1} \circ f_{j-1} (\mathbf{t}) \right) 
	\left( \w_j \circ \Lam{j}{j-2} \circ f_{j-2} (\mathbf{t}) \right) 
	\cdots 
	\left( \w_j \circ \Lam{j}{1} \circ f_{1} (\mathbf{t}) \right) 
	(\w_j(t_j)).
	\end{align*}
	Therefore one can deduce that the map $f_j \colon \mathbb T 
	\to T^{n_j+1}$ is given by
	\[
	\mathbf{t} = (t_1,\dots,t_m) \mapsto 
	\left( \w_j \circ \Lam{j}{j-1} \circ f_{j-1} (\mathbf{t}) \right) 
	\left( \w_j \circ \Lam{j}{j-2} \circ f_{j-2} (\mathbf{t}) \right) 
	\cdots 
	\left( \w_j \circ \Lam{j}{1} \circ f_{1} (\mathbf{t}) \right) 
	(\w_j(t_j)).
	\]
	By considering the exponents of
	the map 
	$\w_j \circ \Lam{j}{\ell} \circ f_{\ell} \colon \mathbb T 
	\to T^{n_j+1}$ for $\ell = 1,\dots,j-1$,
	we get the
	following matrix of size $(n_j+1) \times ((n_1+1) + \cdots + (n_m+1))$:
	\[
	\begin{split}
	&\underbrace{B_j}_{(n_j+1) \times (n_j+1)} \cdot 
	\underbrace{A^{(j)}_{\ell}}_{(n_j+1) \times (n_{\ell}+1)} \cdot
	\underbrace{\left[X^{(\ell)}_1 \ \ X^{(\ell)}_2 \ \ \cdots \ \ X^{(\ell)}_{\ell-1} \ \
		B_{\ell} \ \ O \ \ \cdots \ \ O \right]}_{(n_{\ell}+1) \times ((n_1+1) + 
		\cdots + (n_m+1))} \\
	& \quad \quad = 
	\left[
	B_j A^{(j)}_{\ell} X^{(\ell)}_1 \ \
	B_j A^{(j)}_{\ell} X^{(\ell)}_2 \ \
	\cdots \ \
	B_j A^{(j)}_{\ell} X^{(\ell)}_{\ell-1} \ \
	B_j A^{(j)}_{\ell} B_{\ell} \ \ O \ \ \cdots \ \ O
	\right].
	\end{split}
	\]
	Therefore it is enough to show that
	\[
	X^{(j)}_{\ell} 
	= B_j A^{(j)}_{j-1} X^{(j-1)}_{\ell} 
	+ B_{j} A^{(j)}_{j-2} X^{(j-2)}_{\ell} 
	+ \cdots 
	+ B_{j} A^{(j)}_{\ell+1} X^{(\ell+1)}_{\ell} 
	+ B_{j} A^{(j)}_{\ell} B_{\ell},
	\]
	which comes from Lemma~\ref{lemma_Xjl_summand}. Hence we have 
	the tangential $\mathbb T$-representation as in the proposition.

Finally, we claim that 
the weights of the isotropy representation 
of $\mathbb T$ on $T_w F_m$ are pairwise linearly independent.
For a fixed point $w$, let $\mathbf {c}_1, \mathbf{c}_2 \in \Z^n$ be weights
of the tangential $\mathbb T$-representation $T_w F_m 
\cong \bigoplus_{j=1}^m \mathfrak{m}_j$. 
Assume that the weight $\mathbf{c}_1$ comes from $\mathfrak{m}_{j_1}$ and 
$\mathbf{c}_2$ comes from $\mathfrak{m}_{j_2}$ for $j_1 < j_2$. 
Then by the description in \eqref{eq_mj}, 
$\mathbf{c}_1$ is a linear combination of 
$\{\ep_{j,k}^{\ast} \mid 1\leq j \leq j_1, 1\leq k \leq n_j+1\}$.
Since $\mathbf{c}_2$ has nonzero coefficients in $\{\ep_{j_2,k}^{\ast}
\mid 1 \leq k \leq n_{j_2}+1\}$ and $j_1 < j_2$, 
two weights $\mathbf{c}_1$ and $\mathbf{c}_2$ are linearly independent. 
Suppose that both of two weights $\mathbf{c}_1$ and $\mathbf{c}_2$ come from
$\mathfrak{m}_j$. Then they have nonzero
coefficients in $\{\ep_{j,k}^{\ast} \mid 1 \leq k \leq n_j+1\}$ which
are determined by the permutation matrix $B_j$ by \eqref{eq_mj}. 
Hence they are linearly independent, so the result follows. 
\end{proof}

\subsection{GKM graphs}
In the previous subsection, we showed that 
a flag Bott manifold $(F_m, \bf T)$ is a GKM manifold.
For a given GKM manifold $(M,T)$, one can define the following
labeled graph $(\Gamma, \alpha)$; see \cite{GuZa01} for more details.
\begin{definition}
	Let $(M,T)$ be a GKM manifold. The \defi{GKM graph} 
	$(\Gamma, \alpha)$ consists of 
\begin{itemize}
	\item \textbf{vertices:} $V(\Gamma)= M^T$, 
	\item \textbf{edges:} $e \colon v \to w \in E(\Gamma)$ 
	if and only if there exists a $T$-invariant embedded
	2-sphere $X_e$ containing $v, w \in M^T$, and 
	\item \textbf{axial function:} 
	for an edge $e \colon v \to w$, the \defi{axial function} $\alpha$
	maps an edge $e$ to the weight of the isotropy representation $T_v X_e$
	of $T$. 
\end{itemize}
\end{definition}
For an oriented edge $e$ we write $i(e)$, respectively $t(e)$, 
the initial, respectively terminal, vertex of $e$. 
Moreover we write $\overline{e}$ for the oriented edge $e$ with the 
reversed orientation.
For $v \in V(\Gamma)$ we set
\[
E(\Gamma)_v = \{ e \in E(\Gamma) \mid i(e) = v\}.
\]
For the GKM graph  $(\Gamma, \alpha)$ associated to a GKM
manifold $(M,T)$, a collection $\theta = \{\theta_e\}$ of bijections
\[
\theta_e \colon E(\Gamma)_{i(e)} \to E(\Gamma)_{t(e)},
\quad e \in E(\Gamma)
\]
satisfying the following conditions  can be determined naturally:
\begin{enumerate}
	\item $(\theta_e)^{-1} = \theta_{\overline{e}}$
	for $e \in E(\Gamma)$,
	\item $\theta_e$ maps $e$ to $\overline{e}$ for $e \in E(\Gamma)$, and
	\item for $e \in E(\Gamma)$ and $e' \in E(\Gamma)_{i(e)}$,
	there exists $c \in \Z$ such that 
	$\alpha(\theta_e(e')) = 
	\alpha(e') + c \alpha(e)$.
\end{enumerate}
The collection $\theta = \{\theta_e\}$ is called the \defi{connection}.

In Subsection~\ref{sec_tang_rep_of_fBT}, 
we considered $F_m$ with the noneffective canonical $\mathbb T$-action,
and expressed the tangential representation 
$T_{\dw} F_m$ in terms of the weights using the standard basis 
$\{ \ep_{1,1}^{\ast},\dots,\ep_{1,n_1+1}^{\ast},\dots,\ep_{m,1}^{\ast},
\dots,\ep_{m,n_m+1}^{\ast}\}$ in \eqref{eq_Lie_basis}. 
But in the GKM description, we need to consider the effective canonical
$\bf T$-action on $F_m$. Therefore 
to consider the axial function with respect to $\bf T$-action,
we should put 
\begin{equation}\label{eq_effective_basis}
\ep_{1,n_1+1}^{\ast} = \cdots = \ep_{m,n_m+1}^{\ast} = 0
\end{equation} 
in the formula of Proposition~\ref{prop_fBT_is_GKM}.
\begin{theorem}\label{thm_GKM_of_Fm}
	Let $F_m$ be a flag Bott manifold with the effective canonical
	$\bf T$-action. Then
	the GKM graph $(\Gamma, \alpha)$ of $(F_{m}, \bf T)$ consists of 
	\begin{description}
		\item[vertices] $V(\Gamma) =  \prod_{j=1}^m \mathfrak{S}_{n_j+1}$,
		\item[edges] 
		$E(\Gamma)$ is the set of elements
		$w=(w_1,\dots,w_m)$ and $w'=(w_1',\dots,w_m')$ in $V(\Gamma)$ such that 
		$w'=(w_{1},\ldots,w_{j}(r,s),\ldots,w_{m})$ for some transposition $(r,s)\in \mathfrak{S}_{n_{j}+1}$, and
		\item[axial function] 
		for $w$ and $w'$ as above such that 
		$r,s\in [n_{j}+1]$, $r>s$, then
		\[
		\alpha(ww')=\ro{j}{r} - \ro{j}{s},
		\]
		where $\ro{j}{k}$ is the $k$th row of the matrix 
		$\left[ \X{j}{1} \ \X{j}{2} \ \cdots \ \X{j}{j-1} \ B_j
		\ O \cdots \ O \right]$
		for $k \in [n_j+1]$,
		the matrices $\X{j}{\ell}$ are as in \eqref{eq_def_of_Xjl}
		with the modification according to \eqref{eq_effective_basis}. 
	\end{description}
\end{theorem}
\begin{proof}
To find the GKM graph $\Gamma$, we recall that
the product $\Gamma_1 \times \Gamma_2$ of graphs $\Gamma_1$, $\Gamma_2$ 
consists of vertices
$V(\Gamma_1 \times \Gamma_2) := V(\Gamma_1) \times V(\Gamma_2)$ and
edges $E(\Gamma_1 \times \Gamma_2)$ such that
$ e \colon (w_1,w_2) \to (w_1',w_2') \in E(\Gamma_1 \times \Gamma_2)$ if and only if
either $w_1 =w_1'$ and $w_2 \to w_2' \in E(\Gamma_2)$, or 
$w_2 =w_2'$ and $w_1 \to w_1' \in E(\Gamma_1)$. 
We claim that the GKM graph $\Gamma$ of $F_m$ is the product of graphs
$\prod_{j=1}^m \Gamma_j $, where $\Gamma_j$ is 
the GKM graph of $\flag(n_j+1)$. 

By Proposition~\ref{prop_fixed_points}, we know that 
$V(\Gamma) = V(\prod_{j=1}^m \Gamma_j )$. 
To find edges on the graph $\Gamma$, we use an induction argument on the stage.
When the stage is $1$, then our claim obviously holds.
Assume that the GKM graph of $F_{j}$ is the product $\prod_{\ell=1}^j \Gamma_{\ell}$ for $1 \leq j \leq m-1$. 
For $w \in \mathfrak{S}_{n_m+1}$,
let $s_w \colon F_{m-1} \to F_m$ be a section of the fibration 
$F_m \to F_{m-1}$ defined by 
$[g_1,\dots,g_{m-1}] \mapsto [g_1,\dots,g_{m-1},\dw]$. 
Since the section $s_w$ is $\mathbf T$-equivariant, it produces
the GKM graph of $F_{m-1}$ in $\Gamma$.
Hence the section $s_w$ gives edges $(w_1,\dots,w_{m-1},w) 
\to (w_1',\dots,w_{m-1}',w)$ 
in $\Gamma$ such that $(w_1,\dots,w_{m-1}) 
\to (w_1',\dots,w_{m-1}') \in E(\prod_{j=1}^{m-1} \Gamma_{j})$. 

On the other hand, a fiber over each fixed point in $F_{j-1}$
produces the GKM graph of $\flag(n_j+1)$.
Therefore for $(w_1,\dots,w_{m-1}) \in V(\prod_{j=1}^{m-1} \Gamma_{j})$, we have edges $(w_1,\dots,w_{m-1},w_m)
\to (w_1,\dots,w_{m-1},w_m')$ such that $w_m \to w_m' \in E(\Gamma_m)$. 
Let $2N$ be the real dimension of $F_m$. Then
we have that $|E(\Gamma)_v| = N$ 
for every vertex $v \in V(\Gamma)$ by the definition of GKM graph.
The above constructions give exactly $N$ many edges starting from a vertex $v$, 
so we have that $\Gamma = (\prod_{j=1}^{m-1} \Gamma_{j})
\times \Gamma_m$.
By Proposition~\ref{prop_fBT_is_GKM} we have the axial function 
as stated in the theorem.
\end{proof}
As a direct consequence of Theorem~\ref{thm_GKM_of_Fm}, we get the following.
	\begin{corollary}
		The GKM graph $\Gamma$ of $F_m$ is combinatorially equivalent to the product $\prod_{j=1}^m \Gamma_j$, where $\Gamma_j$ is the GKM graph of $\flag(n_j+1)$.
	\end{corollary}

\begin{example}\label{example_GKM_F1_graph}
	Consider $F_1 = \flag(3)$ as in Example~\ref{example_GKM_F1}. 
 	At the point $[\dw]$ determined by $ w=(231) \in \mathfrak{S}_3$, 
 	we have that
 	$T_{[\dw]} F_1 \cong V(B)$, where 
 	$B = \begin{bmatrix}
 	0 & 1 & 0 \\ 0 & 0 & 1 \\ 1 & 0 & 0
 	\end{bmatrix}$. 
	With the effective canonical torus action,
	the tangential representation is
	\[
	T_{[\dw]} F_1 \cong V(-\ep_2^{\ast}) \oplus V(\ep_1^{\ast}) \oplus 
	V(\ep_1^{\ast}-\ep_2^{\ast}).
	\]
	We have an edge $(231) \to (132)$ in the GKM graph 
	since $(132) = (231)(3,1)$
	for the transposition $(3,1) \in \mathfrak{S}_3$.
	Hence the axial function for the edge $(231) \to (132)$ is 
	$\ep_1^{\ast} - \ep_2^{\ast}$. 
	One can do the similar computations for the other fixed points,
	and we have the GKM graph as in Figure~\ref{fig_GKM_graph_Fl3}.
	In the figure, parallel edges have the same axial functions. \hfill \qed
\end{example}

\begin{example}\label{example_GKM_graph_F2}
	Let $F_2$ be the 2-stage flag Bott manifold defined by
	$\A{2}{1} = \begin{bmatrix}
	c_1 & c_2 & 0 \\ 0 & 0 & 0 
	\end{bmatrix}$ as in Example~\ref{example_GKM_F2_2}. 
	The $3$-dimensional compact torus acts effectively on $F_2$.
	Let $\{\ep_{1,1}^{\ast},\ep_{1,2}^{\ast},\ep_{2,1}^{\ast}\}$ be the standard
	basis of $\Lie ((S^1)^2 \times (S^1))^{\ast}$. 
	Near the fixed point given by $(e,s_1) \in \mathfrak{S}_3 \times
	\mathfrak{S}_2$, we have
	the tangential representation as follows:
	\[
	V(\ep_{1,2}^{\ast} - \ep_{1,1}^{\ast})
	\oplus V(-\ep_{1,2}^{\ast}) \oplus V(-\ep_{1,1}^{\ast}) \oplus
	V(c_1\ep_{1,1}^{\ast} + c_2 \ep_{1,2}^{\ast} + \ep_{2,1}^{\ast}).
	\]
	One can see that the
	induced subgraph $\Gamma$, respectively $\Gamma'$, whose vertex set is
	$\mathfrak{S}_3 \times \{e\}$, respectively $\mathfrak{S}_3 \times \{s_1\}$,
	is the same as the GKM graph of $\flag(3)$ with the 
	action of the torus $T^2$ in Example~\ref{example_GKM_F1_graph}. 
	Therefore it is enough to consider the axial functions of 
	edges of the form $e_{w} := (w,e) \to (w,s_1)$ for $w \in \mathfrak{S}_3$. 
	By a similar computation to Example~\ref{example_GKM_F1_graph}, 
	we get the GKM graph of $F_2$
as in Figure~\ref{fig_F2_2}, whose axial function for vertical edges is listed as follows.
\[
\begin{array}{ll}
\alpha(e_{(123)})= -c_1\ep_{1,1}^{\ast} - c_2 \ep_{1,2}^{\ast} - \ep_{2,1}^{\ast}, &
\alpha(e_{(213)}) = -c_2 \ep_{1,1}^{\ast} - c_1 \ep_{1,2}^{\ast} - \ep_{2,1}^{\ast},\\
\alpha(e_{(231)}) = -c_1 \ep_{1,2}^{\ast} - \ep_{2,1}^{\ast}, &
\alpha(e_{(321)}) = -c_2 \ep_{1,2}^{\ast} - \ep_{2,1}^{\ast}, \\
\alpha(e_{(312)}) = -c_2\ep_{1,1}^{\ast} - \ep_{2,1}^{\ast}, &
\alpha(e_{(132)}) = -c_1\ep_{1,1}^{\ast} - \ep_{2,1}^{\ast}.\\
\end{array}
\]
Note that nontrivial 
coefficients of $\ep_{1,1}^{\ast}$ and $ \ep_{1,2}^{\ast}$ shows that $F_2$ is a
nontrivial $\C P^1$-bundle over $\flag(3)$. 
\end{example}
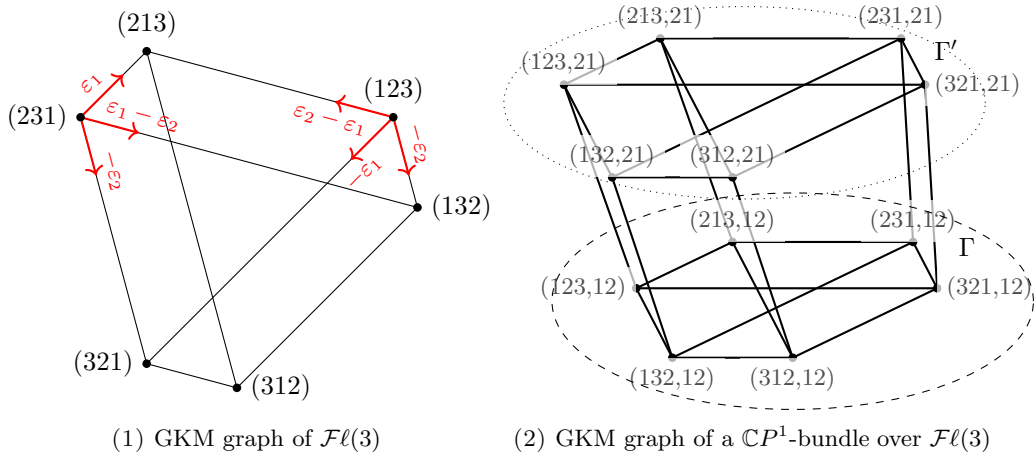
\begin{figure}[H]
	\begin{subfigure}[t]{0.4\textwidth}
		\centering
		\begin{tikzpicture}[scale=0.8]
		\newdimen\R
		\R = 3cm;
		\draw (0:\R) \foreach \x in {30, 120, 150, 240, 270, 360}	{ -- (\x: \R)} ;
		\foreach \x/\l/\p in
		{ 30/{(123)}/above,
			120/{(213)}/above,
			150/{(231)}/left,
			240/{(321)}/left,
			270/{(312)}/right,
			360/{(132)}/right
		}
		\node (r\x) [circle, draw, fill, inner sep = 1pt, label = {\p:\l}] at (\x:\R) {};
		\draw (30:\R) -- (240:\R);
		\draw (120:\R) -- (270:\R);
		\draw (150:\R) -- (360:\R);
		
		\draw[red, thick, ->, above] (r30) -- ($(r30)!1cm!(r120)$) 
		node[at end, below, sloped]{\small{$\ep_2-\ep_1$}};
		\draw[red, thick, ->, above] (r30) -- ($(r30)!1cm!(r240)$) 
		node[at end, below, sloped]{\small{$-\ep_1$}};
		\draw[red, thick, ->, above] (r30) -- ($(r30)!1cm!(r360)$) 
		node[midway, above, sloped]{\small{$-\ep_2$}};
		
		\draw[red, thick, ->, above] (r150) -- ($(r150)!1cm!(r120)$) 
		node[midway, above, sloped]{\small{$\ep_1$}};
		\draw[red, thick, ->, above] (r150) -- ($(r150)!1cm!(r360)$) 
		node[at end, above, sloped]{\small{$\ep_1-\ep_2$}};
		\draw[red, thick, ->, above] (r150) -- ($(r150)!1cm!(r240)$) 
		node[at end, above, sloped]{\small{$-\ep_2$}};
		\end{tikzpicture}
		\caption{GKM graph of $\flag(3)$.}
		\label{fig_GKM_graph_Fl3}
	\end{subfigure}%
	~
	\begin{subfigure}[t]{0.4\textwidth}
		\centering
		\begin{tikzpicture}[y={(2mm,-3.85mm)},z={(0cm,3cm)}, scale=0.8]
		
		\begin{scope}[color=gray!50, thin]
		\foreach \xi in {-1,0,1,2,3,4,5} { \draw (\xi,-3,1.5) -- (\xi, -3,0) -- (\xi, 5, 0); }%
		\foreach \yi in {-3,-1,1,3,5} {\draw (5,\yi,1.5) -- (5,\yi,0) -- (-1,\yi,0);}%
		\foreach \zi in {0,0.5,1,1.5} {\draw (-1,-3,\zi) -- (5,-3,\zi) -- (5,5,\zi);}%
		\end{scope}
		
		
		\draw[thin] (-1,-3,1.5)--(-1,-3,0)--(-1,5,0)--(5,5,0);
		
		\foreach \xi in {-1,0,1,2,3,4,5} {\draw (\xi,5,0) -- (\xi, 5.2, 0) node[anchor=north west] {\tiny{\xi}};}
		\foreach \yi in {-3,-1,1,3,5} {\draw (-1,\yi,0) -- (-1.1,\yi, 0) node[anchor=north east] {\tiny{\yi}};}
		\foreach \zi in {0,0.5,1,1.5} {\draw (-1,-3,\zi) -- (-1.1,-3,\zi) node[anchor=east] {\tiny{\zi}};}
		
		\node at (2.5,7,0) {$\varepsilon_{1,1}^{\ast}$};
		\node at (-2,1,0) {$\varepsilon_{1,2}^{\ast}$};
		\node at (-2,-3,0.7) {$\varepsilon_{2,1}^{\ast}$};

		\coordinate (B1) at (0,0,0) ;
		\coordinate (B2) at (2,-2,0) ;
		\coordinate (B3) at (5,-2,0) ;
		\coordinate (B4) at (5,0,0) ;
		\coordinate (B5) at (2,3,0) ;
		\coordinate (B6) at (0,3,0) ;
		
		\draw[thick] (B1) -- (B2) -- (B3) -- (B4) -- (B5) -- (B6) -- (B1);
		\draw[thick] (B1) -- (B4) ;
		\draw[thick] (B2) -- (B5) ;
		\draw[thick] (B3) -- (B6) ;
		
		\coordinate (T1) at (-1,-1,1) ;
		\coordinate (T2) at (1,-3,1) ;
		\coordinate (T3) at (5,-3,1) ;
		\coordinate (T4) at (5,-1,1) ;
		\coordinate (T5) at (1,3,1) ;
		\coordinate (T6) at (-1,3,1) ;
		
		\draw[ thick] (T1) -- (T2) -- (T3) -- (T4) -- (T5) -- (T6) -- (T1);
		\draw[thick] (T1) -- (T4) ;
		\draw[thick] (T2) -- (T5) ;
		\draw[thick] (T3) -- (T6) ;
		
		\draw[thick] (B1) -- (T1) ;
		\draw[thick] (B2) -- (T2) ;
		\draw[thick] (B3) -- (T3) ;
		\draw[thick] (B4) -- (T4) ;
		\draw[thick] (B5) -- (T5) ;
		\draw[thick] (B6) -- (T6) ;
		
		\definecolor{bottom}{rgb} {1,1,1}
		\definecolor{top}{rgb} {1,1,1}
		
		\foreach \x in {1,2,3,4,5,6}{
			\node[circle,draw,fill,inner sep=1pt] at (B\x) {};
		}
		\node[left, fill=bottom, 
		fill opacity = 0.7, text opacity = 1] at (B1) {\small(123,21)} ;
		\node[above, fill=bottom, 
		fill opacity = 0.7, text opacity = 1] at (B2) {\small(213,21)} ;
		\node[above, fill=bottom, 
		fill opacity = 0.7, text opacity = 1] at (B3) {\small(231,21)} ;
		\node[right, fill=bottom, 
		fill opacity = 0.7, text opacity = 1] at (B4) {\small(321,21)} ;
		\node[below, fill=bottom, 
		fill opacity = 0.7, text opacity = 1] at (B5) {\small(312,21)} ;
		\node[below, fill=bottom, 
		fill opacity = 0.7, text opacity = 1] at (B6) {\small(132,21)} ;

		\foreach \x in {1,2,3,4,5,6} {
			\node[circle, draw, fill, inner sep=1pt] at (T\x) {};
		}
		\node[above, fill=top, 
		fill opacity = 0.7, text opacity = 1] at (T1) {\small(123,12)} ;
		\node[above,  fill=top, 
		fill opacity = 0.7, text opacity = 1] at (T2) {\small(213,12)} ;
		\node[above,  fill=top, 
		fill opacity = 0.7, text opacity = 1] at (T3) {\small(231,12)} ;
		\node[right,  fill=top, 
		fill opacity = 0.7, text opacity = 1] at (T4) {\small(321,12)} ;
		\node[above,  fill=top, 
		fill opacity = 0.7, text opacity = 1] at (T5) {\small(312,12)} ;
		\node[above ,  fill=top, 
		fill opacity = 0.7, text opacity = 1] at (T6) {\small(132,12)} ;

		%
		\end{tikzpicture}
		\caption{GKM graph of a $\C P^1$-bundle over $\flag(3)$.}
		\label{fig_F2_2}
	\end{subfigure}
	\caption{GKM graphs.}\label{fig_GKM_graph}
\end{figure}

\begin{example}
Consider the 3-stage flag Bott manifold $F_3$ as in
Example~\ref{example_GKM_F3}.
Let $\dw = [\dw_1,\dw_2,\dw_3]$ be a fixed point where
$w_1 = (312) \in \mathfrak{S}_3, w_2 = e \in \mathfrak{S}_2$, and $w_3 = (21)
\in \mathfrak{S}_2$. For an edge $(w_1,w_2,w_3) 
\to (w_1, w_2, w_3(2,1))$, the axial function is 
$\rho_2^{(3)} - \rho_1^{(3)}$ where $\rho^{(3)}_k$ is the $k$th
row of the matrix
\[
\left[ \X{3}{1} \ \X{3}{2} \ B_3\right]
= \begin{bmatrix}
0 & 0 & 0 & 0& 0 & 0 & 1\\
14 & 0 & 8 & 5 & 0 & 1 & 0
\end{bmatrix}.
\]
Hence with the modification according to \eqref{eq_effective_basis},
the axial function is $14 \ep_{1,1}^{\ast} + 5 \ep_{2,1}^{\ast} + \ep_{3,1}^{\ast}$.
\hfill \qed
\end{example}

\begin{remark}
	Let $F_{\bullet}$ be a flag Bott tower, and
	$(\Gamma_j, \alpha_j)$ the GKM graph of $j$-stage flag Bott manifold $F_j$.
Then $(\Gamma_j, \alpha_j) \to (\Gamma_{j-1}, \alpha_{j-1})$ is a GKM fiber bundle, see~\cite[Definition 2.3.5]{sabatini2009topology}, induced from the fibration $F_j \to F_{j-1}$ for $1 \leq j \leq m$.
The module basis of GKM graph cohomology of GKM fiber bundle has been computed in~\cite{sabatini2009topology} and~\cite{guillemin2012cohomology}. In the paper~\cite{KKLS}, we compute the equivariant cohomology rings of flag Bott manifolds by using the Borel--Hirzebruch formula.
\end{remark}
\section{Generalized Bott manifolds and the associated flag Bott manifolds}
\label{sec_gBT_and_asso_fBT}
We begin this section by reviewing \emph{generalized Bott towers} studied 
in \cite{CMS10-quasitoric, CMS-Trnasaction} and studying their fans based on 
\cite[Section 7.3]{CLS11}. 

%

\begin{definition}{\cite[Defintion 6.1]{CMS10-quasitoric}}
	\label{def:gBT}
	A \defi{generalized Bott tower} $B_{\bullet}=\{B_j \mid 0 \leq j \leq m\}$
	of height $m$ (or an \defi{$m$-stage generalized Bott tower})
	is a sequence,
	\[
	\begin{tikzcd}
	B_m \arrow[r,"\pi_m"]
	& B_{m-1} \arrow[r,"\pi_{m-1}"]
	& \cdots \arrow[r,"\pi_3"] 
	& B_2 \arrow[r,"\pi_2"]
	& B_1 \arrow[r,"\pi_1"]
	& B_0 =\{ \text{a point} \},
	\end{tikzcd}
	\]
	of manifolds $B_j = \mathbb{P}(E^j_1 \oplus \cdots \oplus E^j_{n_j}
	\oplus \underline{\C})$ 
	where $E^j_k$ is a holomorphic line bundle over $B_{j-1}$ 
	for $1 \leq k \leq n_j$, 
	$\underline{\C}$ is the trivial line bundle over $B_{j-1}$, and
	$\mathbb{P}(\cdot)$ stands for the projectivization of each fiber.
	We call $B_j$ the \defi{$j$-stage generalized Bott manifold}
	of a generalized Bott tower. 
\end{definition}
\begin{example}
	\begin{enumerate}
		\item Every projective space $\C P^{n}$ is a generalized Bott
		tower of height $1$.
		\item The product of projective spaces 
		$\C P^{n_1} \times \cdots \times \C P^{n_m}$ is an $m$-stage
		generalized Bott manifold.
		\item When $n_j=1$ for $1 \leq j \leq m$,
		an $m$-stage generalized Bott manifold is an $m$-stage
		Bott manifold (see~Example~\ref{example_fBT}(3)). \hfill \qed
	\end{enumerate}
\end{example}
Recall from~\cite[Exercise II.7.9]{Hart77} that for each $1 \leq j \leq m$,
the set of isomorphic classes of holomorphic line bundles
on $B_{j-1}$ is isomorphic to 
$\Z^{j-1}$. More precisely, for $1 \leq j \leq m$, the homomorphism 
\[
\Z^{j-1} \to \Pic(B_{j-1}),~~
(a_{1},\dots,a_{j-1}) \mapsto 
(\eta^j_{1})^{\otimes a_{1}} \otimes (\eta^j_{2})^{ \otimes a_{2}} 
\otimes \cdots 
\otimes (\eta^j_{j-1})^{\otimes a_{j-1}}
\]
is an isomorphism
since $B_j$ is an iterated sequence of projective space bundles.
Here, 
$\eta^j_{j-1}$ is the tautological line bundle over $B_{j-1}$, and
$\eta^j_{\ell}
= \pi_{j}^{\ast} \circ \cdots \circ 
\pi_{\ell+1}^{\ast} (\eta^{\ell+1}_{\ell})$,
for each $1 \leq \ell \leq j-2$.
Therefore for each holomorphic line bundle $E^j_k$ over $B_{j-1}$,
there exist integers $a^j_{k,1},\dots,a^j_{k,j-1}$ such that
\[
E^j_{k} =  (\eta^j_{1})^{\otimes a^j_{k,1}} \otimes (\eta^j_{2})^{ \otimes a^j_{k,2}} 
\otimes \cdots 
\otimes (\eta^j_{j-1})^{\otimes a^j_{k,j-1}}.
\]
Hence, we conclude that given a generalized Bott manifold $B_{j-1}$, 
the collection of integers 
$$\{a^j_{k, \ell} \in \mathbb{Z} \mid 1 \leq k\leq n_j,~ 1\leq \ell \leq j-1\}$$
determines $B_j$.

In general, a projectivization of the sum of holomorphic line bundles over a 
toric variety is again a toric variety (see \cite[Section 7.3]{CLS11}).\footnote{Note that \cite{CLS11} uses a different convention to construct iterated projective bundles. They put the trivial line bundle on the first, but we put it on the last when we sum up line bundles in the definition of generalized Bott manifolds.}
Hence, 
so is a generalized Bott manifold $B_m$. To describe  
the fan of $B_m$, we prepare the following matrix $\Lambda$ of size 
$n \times m$;
\begin{equation}\label{eq_red_vector_char_ftn_of_gBT}
n:= n_1+ \cdots+n_m \quad \text{ and }\quad 
\Lambda:=
\left[
\begin{array}{cccccc}
-\mathbf{1}& \mathbf{0}& \cdots \\
\mathbf{a}^2_{1} & -\mathbf{1}& \mathbf{0}& \cdots \\
\vdots &\ddots & \ddots & \ddots \\
\mathbf{a}^{j}_{1} & \cdots & \mathbf{a}^{j}_{j-1}& -\mathbf{1}&\mathbf{0}& \cdots\\
\vdots &&& \ddots & \ddots \\
\mathbf{a}^m_{1} & \cdots &\cdots  &  & \mathbf{a}^m_{m-1}& -\mathbf{1}
\end{array}
\right] \!\!\!\!\!
\begin{array}{l}
\MyLBrace{1ex}{$n_1$}\\
\MyLBrace{1ex}{$n_2$}\\
\\[1.5ex]
\MyLBrace{1ex}{$n_j$}\\
\\[1.5ex]
\MyLBrace{1ex}{$n_m$}
\end{array}, 
\end{equation}
where we denote by $\mathbf{0}$, $\mathbf{1}$  and $\mathbf{a}^{j}_{\ell}$ 
the following vectors respectively: 
$$
\mathbf{0} = \begin{bmatrix}
0 \\ \vdots \\ 0
\end{bmatrix},~
\mathbf{1} = \begin{bmatrix}
1 \\ \vdots \\1
\end{bmatrix}, ~~\text{and}~~ 
\mathbf{a}^{j}_{\ell}=
\begin{bmatrix}
a^{j}_{1,\ell} \\ \vdots \\ a^{j}_{n_{j},\ell}
\end{bmatrix} \in \Z^{n_{j}} \quad \textrm{ for }1\leq \ell < j \leq m.
$$
Next, we define a set of vectors $\mathcal{U}:=\{u_{k_j}^j \mid 1\leq j\leq m,~ 1\leq k_j \leq n_j+1\}$ by 
$$u_{k_j}^j=
	\begin{cases} \ep_{j,k_j} & \text{ if } 1 \leq k_j \leq n_j, \\ 
j\text{-th column of }\Lambda&  \text{ if } k_j=n_j+1,\end{cases}$$
where $\ep_{1,1},\dots,\ep_{1,n_1},\dots,\ep_{m,1},\dots,\ep_{m,n_m}$ is the standard basis vector in $\mathbb{R}^{n} = \mathbb{R}^{n_1+ \cdots + n_m}$. 
Now, we consider the following cones
\begin{equation*}\label{eq_cones_in_gbt}
\sigma_{k_1, \dots, k_m}:=\Cone(\mathcal{U}\setminus \{u^1_{k_1}, \dots, u^m_{k_m}\}) 
\subset \mathbb{R}^{n}, 
\end{equation*}
and one can see that the vectors of $\mathcal{U}\setminus \{u^1_{k_1}, \dots, u^m_{k_m}\}$ form a $\mathbb{Z}$-basis of $\mathbb{Z}^n \subset \mathbb{R}^n$.
Hence $\sigma_{k_1, \dots, k_m}$ is a smooth cone of dimension $n$. 

\begin{proposition}
	A fan $\Sigma$ associated to $B_m$ consists of the cones 
	\begin{equation}\label{eq_cones_in_GBT}
	\Big\{ \sigma_{k_1, \dots, k_m} ~\Big|~  (k_1, \dots, k_m)\in \prod_{j=1}^m [n_j+1] \Big\}
	\end{equation}
	and their faces. 
\end{proposition}
\begin{proof}
	We show the claim by the induction on the stage of a generalized Bott manifold. When $m=1$, 
	we have $u^1_k=\mathbf{e}_k$ for $1 \leq k \leq n_1$ and $u^1_{n_1+1}=-\mathbf{1}$. 
	In this case, the fan 
	$\Sigma$ consists of the cones $\{\sigma_{k_1} \subset \mathbb{R}^{n_1} \mid 1\leq k_1 \leq n_1+1\}$
	and their faces, which yields $X_\Sigma\cong \C P^{n_1}$. 
	Next, assuming that the claim holds for $(m-1)$-stage generalized Bott manifold $B_{m-1}$, 
	a successively application of the result \cite[Section 7.3]{CLS11}, 
	in particular \cite[Proposition~7.3.3 and Example~7.3.5]{CLS11}, establishes that the claim 
	holds for the $m$-stage generalized Bott manifold $B_m$. 
\end{proof}

\begin{remark}
	The fan $\Sigma$ defined above is a simplicial fan whose underlying simplicial complex
	is the dual complex of the product $P:=\prod_{j=1}^m \Delta^{n_j}$ of simplices. As a quasitoric
	manifold \cite{DaJa91, BuPa15}, the polytope together with the set $\mathcal{U}$, where we assign a facet 
	$$\Delta^{n_1} \times \dots \times \Delta^{n_{j-1}} \times f_{k_j}^j \times \Delta^{n_{j+1}} \times \dots \times \Delta^{n_m}$$
	for some facet $f_{k_j}^j$ of $\Delta^{n_j}$ to the vector $u_{k_j}^j$ for $1 \leq k_j \leq n_j+1$, form a characteristic pair
	which determines the given generalized Bott manifold. We refer the readers to  \cite{CMS10-quasitoric} and 
	\cite{CMS-Trnasaction} for more details. 
\end{remark}


\begin{example}\label{example_gBT}
	Let $B_{\bullet}$ be a generalized Bott tower of height $3$ with $n_1 = 2$,
	$n_2 = 1$, and $n_3 = 2$. The $2$-stage generalized Bott manifold
	$B_2$
	is a $\C P^1$-fiber bundle over $\C P^2$, and 
	the $3$-stage $B_3$ is a $\C P^2$-fiber
	bundle over the manifold $B_2$. More precisely,
	\[
	\begin{tikzcd}[row sep=0.2cm]
	& E^3_1 \oplus E^3_2 \oplus \underline{\C} \dar
	& E^2_1 \oplus \underline{\C} \dar \\[2ex]
	\mathbb{P}(E^3_1 \oplus E^3_2 \oplus \underline{\C})  \arrow[d,equal] \rar &
	\mathbb{P}(E^2_1 \oplus \underline{\C})  \arrow[d,equal]  \rar
	& \C P^2  \arrow[d,equal]\\
	B_3 & B_2 & B_1 
	\end{tikzcd}
	\]
	where $\underline{\C}$ is the trivial line bundle, and 
	\[
	E^2_1 = (\eta^2_1)^{\otimes a^2_{1,1}},~ 
	E^3_1 = (\eta^3_1)^{\otimes a^3_{1,1}} \otimes (\eta^3_2)^{\otimes a^3_{1,2}},~
	E^3_2 = (\eta^3_1)^{\otimes a^3_{2,1}} \otimes (\eta^3_2)^{\otimes a^3_{2,2}}
	\]
	for some integers $a^{2}_{1,1}, a^3_{1,1}, a^3_{1,2}, a^3_{2,1},
	a^3_{2,2}$. Hence the matrix $\Lambda$ of $B_3$ is 
	\[
	\Lambda = 
	\begin{bmatrix}
	-1 & 0 & 0 \\
	-1 & 0 & 0 \\
	a^2_{1,1} & -1 &  0 \\
	a^3_{1,1} & a^3_{1,2} & -1 \\
	a^3_{2,1} & a^3_{2,2} & -1 
	\end{bmatrix}
	= \begin{bmatrix}
	-\mathbf{1} & \mathbf{0} & \mathbf{0} \\
	\mathbf{a}^2_1 & -\mathbf{1} & \mathbf{0} \\
	\mathbf{a}^3_1 &\mathbf{a}^3_2 & -\mathbf{1}
	\end{bmatrix} = [ u^1_3 \ u^2_2 \ u^3_3],
	\]
	where $\mathbf{a}^2_{1} = a^2_{1,1} \in \Z$,
	$\mathbf{a}^3_1 = (a^3_{1,1}, a^3_{2,1}) \in \Z^2$, and
	$\mathbf{a}^3_2 = (a^3_{2,1}, a^3_{2,2}) \in \Z^2$.  
	Moreover the fan $\Sigma$ associated to $B_3$ consists of cones
\begin{align*}
	&\Cone(\varepsilon_{1,1}, \varepsilon_{1,2}, \varepsilon_{2,1}, \varepsilon_{3,1}, \varepsilon_{3,2}), 
	&&\Cone(\varepsilon_{1,1}, u^1_3, \varepsilon_{2,1}, \varepsilon_{3,1}, \varepsilon_{3,2}), 
	&&\Cone(\varepsilon_{1,2}, u^1_3, \varepsilon_{2,1}, \varepsilon_{3,1}, \varepsilon_{3,2}), \\
	&\Cone(\varepsilon_{1,1}, \varepsilon_{1,2}, u^2_2, \varepsilon_{3,1}, \varepsilon_{3,2}), 
	&&\Cone(\varepsilon_{1,1}, u^1_3, u^2_2, \varepsilon_{3,1}, \varepsilon_{3,2}), 
	&&\Cone(\varepsilon_{1,2}, u^1_3,u^2_2, \varepsilon_{3,1}, \varepsilon_{3,2}), \\
	&\Cone(\varepsilon_{1,1}, \varepsilon_{1,2}, \varepsilon_{2,1}, \varepsilon_{3,1}, u^3_3),
&&\Cone(\varepsilon_{1,1}, u^1_3, \varepsilon_{2,1}, \varepsilon_{3,1}, u^3_3),
&&\Cone(\varepsilon_{1,2}, u^1_3, \varepsilon_{2,1}, \varepsilon_{3,1}, u^3_3), \\
&\Cone(\varepsilon_{1,1}, \varepsilon_{1,2}, u^2_2, \varepsilon_{3,1}, u^3_3),
&&\Cone(\varepsilon_{1,1}, u^1_3, u^2_2, \varepsilon_{3,1}, u^3_3),
&&\Cone(\varepsilon_{1,2}, u^1_3,u^2_2, \varepsilon_{3,1}, u^3_3), \\
	&\Cone(\varepsilon_{1,1}, \varepsilon_{1,2}, \varepsilon_{2,1}, \varepsilon_{3,2}, u^3_3),
&&\Cone(\varepsilon_{1,1}, u^1_3, \varepsilon_{2,1}, \varepsilon_{3,2}, u^3_3),
&&\Cone(\varepsilon_{1,2}, u^1_3, \varepsilon_{2,1},  \varepsilon_{3,2}, u^3_3), \\
&\Cone(\varepsilon_{1,1}, \varepsilon_{1,2}, u^2_2, \varepsilon_{3,2}, u^3_3),
&&\Cone(\varepsilon_{1,1}, u^1_3, u^2_2,\varepsilon_{3,2}, u^3_3),
&&\Cone(\varepsilon_{1,2}, u^1_3,u^2_2, \varepsilon_{3,2}, u^3_3)
\end{align*}
	and their faces.
	\hfill \qed
\end{example}

\begin{definition}\label{prop:lambda_of_asso_gBT}
	Let $B_{\bullet}$ be a generalized Bott tower determined by 
	the block matrix $\Lambda$ with entries $\mathbf{a}^j_{\ell}$ as in~\eqref{eq_red_vector_char_ftn_of_gBT}.
	We call a flag Bott tower $F_{\bullet}$ \defi{is associated to $B_{\bullet}$} if it is 
	determined by the set of integer matrices 
	$\{\A{j}{\ell} \in M_{n_j+1, n_{\ell}+1}(\Z) \mid 
	{1 \leq \ell < j \leq m}\}$ where
\[
	\A{j}{\ell} = \bigg[ 
	\begin{array}{cccc}
	\mathbf{a}^j_{\ell} & 
	\mathbf{0}  & \cdots & \mathbf{0} \\
	\tikzmark{mark11} 0 \tikzmark{mark12} &  \tikzmark{mark21} 0 & \cdots & 0 \tikzmark{mark22}
	\end{array}	\bigg] \!\!\!\!\!
	\begin{array}{l}
	{} \tikzmark{mark1}  \\ {} \tikzmark{mark2} 
	\end{array}\qquad.
	\]
\tikz[overlay, remember picture, decoration = {brace, amplitude = 3pt}]{
\draw[decorate, thick]  (mark12.south east)  -- (mark11.south west) node [midway, below = 3pt] {$1$};
\draw[decorate, thick] (mark22.south) -- (mark21.south)  node [midway, below = 3pt] {$n_{\ell}$};
\draw[decorate, thick] (mark1.55) -- (mark1.-55) node [midway, right = 3pt] {$n_j$};
\draw[decorate, thick] (mark2.55) -- (mark2.-55) node [midway, right = 3pt] {$1$};
}
\end{definition}
\begin{example}\label{example_gBT_afBT}
Let $B_3$ be the generalized Bott tower of height $3$ 
in Example~\ref{example_gBT}.
The associated flag Bott manifold $F_3$ to $B_3$ is determined by
the following integer matrices:
\begin{align*}
\A{2}{1} &= \begin{bmatrix}
\mathbf{a}^2_{1} & \mathbf{0} & \mathbf{0} \\
0 & 0 & 0
\end{bmatrix}
= \begin{bmatrix}
a^2_{1,1} & 0 & 0 \\ 0 & 0 & 0
\end{bmatrix} \in M_{2,3}(\Z), \\
\A{3}{1} &= \begin{bmatrix}
\mathbf{a}^3_1 & \mathbf{0} & \mathbf{0} \\
0 & 0 & 0 
\end{bmatrix}
= \begin{bmatrix}
a^3_{1,1} & 0 & 0 \\
a^3_{2,1} & 0 & 0 \\
0 & 0 & 0 
\end{bmatrix} \in M_{3,3}(\Z), \quad 
\A{3}{2} = \begin{bmatrix}
\mathbf{a}^3_2 & \mathbf{0} \\
0 & 0
\end{bmatrix}
= \begin{bmatrix}
a^3_{1,2} & 0 \\
a^3_{2,2} & 0 \\ 
0 &0
\end{bmatrix} \in M_{3,2}(\Z).
\end{align*}
\end{example}

For a generalized Bott tower $B_{\bullet}$ and its associated flag Bott tower $F_{\bullet}$, we have the following commutative diagram.
\begin{equation}\label{eq_assoc_flag_bundle_diagram}
\begin{tikzcd}[row sep=0.3cm, column sep=0.3cm]
 & {} 
& & q_{m-1}^{{\ast}}E_{m} \arrow[dd] \arrow[ld]  
&& && q_1^{\ast} E_2  \arrow[dd] \arrow[ld]
&& q_0^{\ast} E_1\arrow[dd] \arrow[ld]
\\
F_{m} \arrow[dd, "{q_{m}}"] \arrow[rr,"  p_{m}"] 
& & F_{m-1}\arrow[dd,"q_{m-1}"] \arrow[rr,"\quad \quad p_{m-1}",
crossing over]
&& \cdots \arrow[rr, "p_2"] 
&& F_1 \arrow[rr, "\quad \quad p_1", crossing over] \arrow[dd, "q_{1}"]
&& F_0 \arrow[dd, "q_0=\text{id}"]
\\
& {}
& & E_{m} \arrow[ld] 
&& && E_2 \arrow[ld]
&& E_1 \arrow[ld]	  \\
 B_{m}\arrow[rr,"  \pi_{m}"] 
& & B_{m-1} \arrow[rr,"  \pi_{m-1}"] 
&& \cdots \arrow[rr,"\pi_2"]
&& B_1 \arrow[rr, "\quad \quad \pi_1"]
&& B_0
\end{tikzcd}
\end{equation}
Indeed, the associated flag Bott tower $F_{\bullet}$ can be constructed inductively as follows. For each $1 \leq j \leq m$, consider the following pull-back diagram.
\[
\begin{tikzcd}
q_{j-1}^\ast E_{j} \arrow[r, "\widetilde{q}_{j-1}"]
	\arrow[d, ""{name=U}] & 
E_{j} \arrow[d, ""{name=D}]\\
F_{j-1} \arrow[r,"q_{j-1}"]& B_{j-1} 
\arrow[draw=none,"\text{\large{$\circlearrowright$}}" description,from=U,to=D]
\end{tikzcd}
\]
By flagifying each fiber of the above bundles,
 we obtain the associated pull back diagram of flag bundles.
\[
\begin{tikzcd}
F_j:= \flag(q_{j-1}^\ast E_{j}) 
\arrow[r, "\widetilde{q}_{j-1}"]
\arrow[d, "p_j"{name=U}]
\arrow[rr, bend left, "q_j"]
& \flag(E_{j})
\arrow[r, "s_j"] 
\arrow[d, ""{name = D}]
&\mathbb{P}(E_{j})=B_j 
\arrow[dl, "\pi_j"]\\
F_{j-1} 
\arrow[r, "q_{j-1}"]
& B_{j-1} 
\arrow[draw=none,"\text{\large{$\circlearrowright$}}" description,from=U,to=D]
\end{tikzcd}
\]
Then  $F_{j}$ is the total space of $\flag(q_{j-1}^\ast E_{j})$, and $q_{j} := s_j \circ 
\widetilde{q}_{j-1}$. Here, the
map $s_{j} \colon \flag(E_{j}) \to \mathbb{P}(E_{j})$
sends each fiberwise full flag $V_{\bullet} = 
( V_1 \subsetneq V_2 \subsetneq \cdots \subsetneq
V_{n_{j}} \subsetneq (E_j)_p)$ to the element $V_1$ in 
$\mathbb{P}((E_j)_p)$ for $p \in B_{j-1}$.

\section{Generic orbit closures in the associated flag Bott manifolds}
\label{sec_generic_orbit_closure_in_asso_fBM}
For an $m$-stage generalized Bott manifold $B_m$, let 
$F_m$ be its associated flag Bott manifold with the effective canonical
action of $\mathbf H$ defined in Subsection~\ref{sec_torus_action_fBT}.
In this section, we study the closure of a generic orbit of
the torus $\mathbf H$ in the associate flag Bott manifold $F_m$ and 
its relation with $B_m$ in Theorem~\ref{thm_main_thm_2}. 
For this, we first review combinatorics
of permutohedral varieties.

\subsection{Permutohedral varieties}
\label{subsec:permutohedron}
The closure $X_n$ of a generic orbit in the flag variety  
$\flag(n+1)$ with the effective action of $(\Cstar)^n$
as in Example~\ref{exampel_torus_action_flag}
is a toric variety called the \defi{permutohedral variety}; see for instance \cite{Klya85} and \cite{June14}. 
In this subsection, we recall the fan $\Sigma_{n} \subset \mathbb{R}^{n}$ of the permutohedral variety. 
Note that the fan $\Sigma_n$ is the normal fan of an $n$-dimensional permutohedron $P_n$ with particular outward normal vectors. To be more precise, there is a bijection between the set $\Sigma_{n}(1)$ of rays and nonempty proper subsets of $[n+1]$:
\[
\Sigma_{n}(1) \stackrel{1-1}{\longleftrightarrow} \{ A \mid \emptyset \subsetneq A \subsetneq [n+1] \}.
\]
For a nonempty proper subset $A$ of $[n+1]$, the corresponding ray $\rho_A$ is generated by
\begin{equation}\label{eq_u_A}
u_A := \begin{cases}
\displaystyle\sum_{x \in A} \varepsilon_x & \text{ if } n+1 \notin A, \\
\displaystyle-\sum_{x \in [n+1] \setminus A} \varepsilon_x & \text{ otherwise},
\end{cases}
\end{equation}
where $\{\varepsilon_1,\dots,\varepsilon_n\}$ is the standard basis vector of $\mathbb{R}^n$. 
Hence there are $2^{n+1}-2$ many rays in $\Sigma_{n}$.
The minimal generator in the intersection of a ray and the underlying lattice is called the \defi{ray generator}.
We note that $u_A$ defined in~\eqref{eq_u_A} is the ray generator of $\rho_A$. 

The maximal cones are indexed by proper chains of $n$ nonempty proper subsets of $[n+1]$. For a proper chain 
\begin{equation}\label{eq_chain_A}
A_{\bullet} : \emptyset \subsetneq A_1 \subsetneq A_2 \subsetneq \cdots \subsetneq A_n \subsetneq [n+1]
\end{equation}
of nonempty proper subsets, we have the corresponding maximal cone
\[
\Cone(u_{A_1},u_{A_2},\dots,u_{A_n}).
\]
Therefore the number of maximal cones is $(n+1)!$. 

Moreover we have a correspondence between the maximal cones in $\Sigma_n$ and the elements of the symmetric group $\mathfrak{S}_{n+1}$. For a permutation $w = (w(1) \ \cdots \ w(n+1))$ in $\mathfrak{S}_{n+1}$, we associate a maximal cone in $\Sigma_n$ determined by the chain $A_{\bullet}$ where
\begin{equation}\label{eq_chain_Ak}
A_k := \{w(n+2-k),\dots,w(n+1)\} \quad \text{ for } 1 \leq k \leq n.
\end{equation}
This description is sometimes much convenient 
to see the combinatorics of $\Sigma_n$. For 
instance, two maximal cones corresponding to permutations $v$ and $w$ in $\mathfrak{S}_{n+1}$ are adjacent if and only if there exists $i \in [n]$ such that $v = w \cdot s_i$, where $s_i$ is the transposition $(i,i+1) \in \mathfrak{S}_{n+1}$.

\begin{example}
	When $n = 2$, Figure~\ref{fig_ray} represents ray generators in $\Sigma_2$.
	Consider a permutation $(231) \in \mathfrak{S}_3$. 
	Then the corresponding chain $A_{\bullet}$ defined in~\eqref{eq_chain_Ak} is 
	\[
	A_{\bullet} \colon \emptyset \subsetneq \{1\} \subsetneq \{1,3\} \subsetneq \{1,2,3\}.
	\]
	Hence the permutation $(231)$ defines a maximal cone 
	$\Cone(u_{\{1\}},u_{\{1,3\}})$.
	As permutations $(231)$ and $(321)$ satisfy the relation
	$(231) = (321) \cdot s_1$, two maximal cones $\Cone(u_{\{1\}},u_{\{1,3\}})$ and $\Cone(u_{\{1\}},u_{\{1,2\}})$ are adjacent.
	Figure~\ref{fig_m_cones} describes the maximal cones in $\Sigma_2$. \hfill \qed
\end{example}
\begin{figure}[H]
	\begin{subfigure}[b]{0.5\textwidth}
		\centering
		\begin{tikzpicture}
		\draw[->] (0,0)--(0,1) node[above] {$u_{\{2\}}=\ep_2$};
		\draw[->] (0,0) -- (1,0) node[right] {$u_{\{1\}}=\ep_1$};
		\draw[->] (0,0) -- (1,1) node[right] {$u_{\{1,2\}}=\ep_1+\ep_2$};
		\draw[->] (0,0) -- (-1,0) node[left] {$u_{\{2,3\}}=-\ep_1$};
		\draw[->] (0,0) -- (0,-1) node[below right] {$u_{\{1,3\}}=-\ep_2$};
		\draw[->] (0,0) -- (-1,-1) node[below left] {$u_{\{3\}}=-\ep_1-\ep_2$};
		\end{tikzpicture}
		\caption{Ray generators in $\Sigma_{2}$.}
		\label{fig_ray}
	\end{subfigure}%
	\begin{subfigure}[b]{0.5\textwidth}
		\centering
		\begin{tikzpicture}
		\fill[pattern color = black!10!white, pattern = vertical lines] (0,0)--(2,0)--(2,2)--cycle;
		\fill[pattern color = black!10!white, pattern = north east lines] (0,0) -- (-2,0) -- (-2,2)--(0,2) -- cycle;
		\fill[pattern color = black!10!white, pattern = horizontal lines] (0,0)--(2,2)--(0,2)--cycle;
		\fill[pattern color = black!10!white, pattern = north east lines] (0,0)--(2,0) -- (2,-2) -- (0,-2) -- cycle;
		\fill[pattern color = black!10!white, pattern = horizontal lines] (0,0) -- (0,-2) -- (-2,-2) --cycle;
		\fill[pattern color = black!10!white, pattern = vertical lines] (0,0) -- (-2,-2)--(-2,0) --cycle;
		
		\draw[gray] (0,0) -- (0,2);
		\draw[gray] (0,0) -- (2,0);
		\draw[gray] (0,0) -- (2,2);
		\draw[gray] (0,0) -- (-2,0);
		\draw[gray] (0,0) -- (-2,-2);
		\draw[gray] (0,0) -- (0,-2);
		
		\node at (0.75,1.5) {\scriptsize $\{2\} \subsetneq \{1,2\}$};
		\node at (0.5,1) {\scriptsize \textcolor{red}{$(312)$}};
		\node at (1.25,0.3) {\scriptsize $\{1\} \subsetneq \{1,2\}$};
		\node at (1.5,0.8) {\scriptsize \textcolor{red}{$(321)$}};
		\node at (1.25,-1) {\scriptsize $\{1\} \subsetneq \{1,3\}$};
		\node at (1.25, -0.5) {\scriptsize \textcolor{red}{$(231)$}};
		\node at (-0.75,-1.5) {\scriptsize $\{3\} \subsetneq \{1,3\}$};
		\node at (-0.5,-1) {\scriptsize \textcolor{red}{$(213)$}};
		\node at (-1.25,-0.3) {\scriptsize $\{3\} \subsetneq \{2,3\}$};
		\node at (-1.5,-0.8) {\scriptsize \textcolor{red}{$(123)$}};
		\node at (-1.25,1) {\scriptsize $\{2\} \subsetneq \{2,3\}$};
		\node at (-1.25,0.5) {\scriptsize \textcolor{red}{$(132)$}};
		
		\draw[thick,->] (0,0)--(0,1) ;
		\draw[thick,->] (0,0) -- (1,0) ;
		\draw[thick,->] (0,0) -- (1,1) ;
		\draw[thick,->] (0,0) -- (-1,0) ;
		\draw[thick,->] (0,0) -- (0,-1) ;
		\draw[thick,->] (0,0) -- (-1,-1) ;
		
		\end{tikzpicture}
		\caption{Maximal cones in $\Sigma_{2}$.}
		\label{fig_m_cones}
	\end{subfigure}
	\caption{Fan $\Sigma_{2}$.}
	\label{fig_fan_sigma_2}
\end{figure}
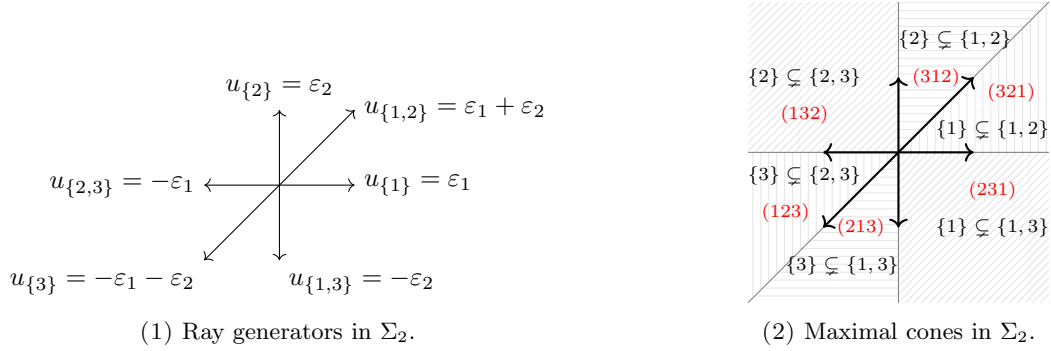

\begin{remark}\label{rmk:permuto_and_simplex}
	Let $\Sigma'_n \subset \mathbb{R}^n$ be the fan of complex projective space $\C P^n$ whose ray generators $u_1,\dots,u_{n+1}$ are given by
	\[
	u_k = 
	\begin{cases}
	\ep_k & \text{ if } 1 \leq k \leq n, \\
	-\ep_1 - \cdots - \ep_n & \text{ if } k = n+1
	\end{cases}
	\]
	Then the set of cones in $\Sigma'_n$ can be identified with the set of nonempty proper 
	subsets of $[n+1]$. To be more precise, for any  
	dimension $d$ cone $\tau$ in $\Sigma'_n$,
	we have a subset $\{i_1,\dots,i_d\} \subset [n+1]$ such that 
	\[
	\tau = \Cone(u_{i_1},\dots,u_{i_d}).
	\]
	It is well known that the fan $\Sigma_n \subset \mathbb{R}^n $ of the permutohedral variety can be obtained 
	from $\Sigma'_n$ by star subdivisions of all cones of dimension grater than 0
	in the decreasing order of the dimensions of the cones (see~\cite{Proc90}). 
	Hence, the set of rays in the fan $\Sigma_n$ corresponds bijectively to the set of all cones of dimension grater than 0 in $\Sigma'_n$. \hfill \qed
\end{remark}

\subsection{The main result on generic orbit closures in $F_m$}
\label{subsec_main_thm}
Consider the canonical effective $\mathbf H$-action on $F_m$ 
defined in Subsection~\ref{sec_torus_action_fBT}.
In order to consider the closure of a generic $\mathbf H$-orbit in $F_m$, we 
first define a generic element in $F_m$. 
Let $g=(g_{ij})$ be an element in $\GL(n+1)$.
For an ordered sequence $1 \leq i_1 < i_2 <
\cdots < i_k \leq n+1$, we consider the Pl\"{u}cker coordinate
\[
X_{i_1,\dots,i_k}(g) := \det ((g_{i_p,p})_{1 \leq p \leq k}).
\]
\begin{definition}
We call an element $g \in \GL(n+1)$ \defi{generic}
if $X_{i_1,\dots,i_k}(g)$ is nonzero for any $k 
\in [n+1]$ and  ordered sequence 
$1 \leq i_1 < i_2 < \cdots < i_k \leq n+1$. 
We call a point $[g_1,\dots,g_m]$ in $F_m$ is 
\defi{generic} if $g_j \in \GL(n_j+1)$ is generic
for $j = 1,\dots,m$. 
\end{definition}
For example, $g=\begin{bmatrix}
1 & 0 \\ 0 & 1
\end{bmatrix}$ is not a generic element since
$X_2(g) = 0$. But $g = \begin{bmatrix}
1 & 0 \\ 1 & 1
\end{bmatrix}$ is generic. 
The above definition of generic elements can be found in 
\cite{FlHa91}, \cite{Klya95}, and \cite{Dabr96}.
It is not difficult to show that the genericity of a point  $[g_1,\dots,g_m]$
in $F_m$ does not depend on the representative of a point.

A \defi{generic orbit} in $F_m$ is 
the $\mathbf H$-orbit of a generic point. 
In Theorem~\ref{thm_main_thm_2} we give a relation between 
a generalized Bott manifold $B_m$ and the closure of a generic orbit of 
$\mathbf H$ in its associated flag Bott manifold~$F_m$, which extends the relation between $\C P^{n}$, as an 
1-stage generalized Bott manifold, and the 
$n$-dimensional permutohedral variety (see~Remark~\ref{rmk:permuto_and_simplex}).

\begin{theorem}\label{thm_main_thm_1}
Let $B_m$ be an $m$-stage generalized Bott manifold 
determined by an integer matrix
$\Lambda$ as in~\eqref{eq_red_vector_char_ftn_of_gBT} and
let $F_m$ be the associated $m$-stage flag Bott manifold. 
Then  the closure of a generic orbit of 
$\mathbf H$ in the associated flag Bott manifold $F_m$ is a nonsingular projective toric variety whose fan $\Sigma$ is given as follows: 
\begin{enumerate}
	\item the rays are parametrized by the set 
	\[
	\{(\ell,A) \mid \emptyset \subsetneq A \subsetneq [n_{\ell}+1], 1 \leq \ell \leq m\}.
	\] 
	For $(\ell,A)$ the corresponding ray is generated by the vector
	\[
	u^{\ell}_A = 
	\begin{cases}
	\mathlarger\sum_{x \in A} \ep_{\ell,x} & \text{ if }	
	n_{\ell}+1 \notin A, \\
	\mathlarger-\sum_{x \in [n_{\ell}+1]\setminus A} \ep_{\ell,x}
	+\mathlarger\sum_{j=\ell+1}^{m} \mathlarger\sum_{k=1}^{n_j} a^j_{k,\ell} \ep_{j,k}
	& \text{ otherwise}
	\end{cases}
	\]
	where $\{\ep_{j,k}\}$ is the standard basis of 
	the Lie algebra of the compact torus $\mathbf T \subset \mathbf H$	whose dual is the standard basis $\{\ep_{j,k}^{\ast}\}$ of $\Lie(\mathbf{T})^{\ast}$.
	\item The maximal cones are indexed by the sequences of proper chains of subsets 
	\[
	\{(A^1_{\bullet},\dots,A^m_{\bullet}) \mid A^{\ell}_{\bullet} = (\emptyset \subsetneq A^{\ell}_1 \subsetneq A^{\ell}_2 \subsetneq \cdots \subsetneq A^{\ell}_{n_{\ell}} \subsetneq [n_{\ell}+1]), 1 \leq \ell \leq m\}.
	\]
	For $(A^1_{\bullet},\dots,A^m_{\bullet})$, the corresponding maximal cone is defined to be
	\[
	\textup{Cone}\left(\bigcup_{\ell=1}^m \{u^{\ell}_{A_1^{\ell}},\dots,u^{\ell}_{A_{n_{\ell}}^{\ell}}\}\right).
	\]
\end{enumerate}
\end{theorem}
The proof of Theorem~\ref{thm_main_thm_1} needs a series of lemmas,
and will be given in the next subsection.
The following corollary will play an important role in the proof of Theorem~\ref{thm_main_thm_2}. 
\begin{corollary}\label{cor_char_vectors_are_obtained_by_sum}
	For each $1 \leq \ell \leq m$ and a nonempty proper subset
	$\emptyset \subsetneq 
	A \subsetneq [n_{\ell}+1]$, we have the following relation:
\begin{equation}\label{eq_u_ell_A_sum}
u^{\ell}_A = \sum_{x \in A} u^{\ell}_{\{x\}}.
\end{equation}
Furthermore, for $x \in [n_{\ell}+1]$, the ray generator
$u^{\ell}_{\{x\}}$ coincides with the ray generator $u^{\ell}_x$ in the fan $\Sigma'$
of the generalized Bott manifold $B_m$.
\hfill \qed
\end{corollary}
\begin{proof}
First we notice that $u^{\ell}_{\{x\}} = \varepsilon_{\ell,x} = u^{\ell}_x$ if $x \neq n_{\ell}+1$.
Hence we get the equality~\eqref{eq_u_ell_A_sum} when $n_{\ell}+1 \notin A$. 
On the other hand, we have that
\[
u^{\ell}_{\{n_{\ell+1}\}} = -\sum_{x \in [n_{\ell}]} \varepsilon_{\ell,x} 
+ \sum_{j= \ell+1}^m \sum_{k=1}^{n_j} a^j_{k,\ell} \varepsilon_{j,k} = u^{\ell}_{n_{\ell}+1}.
\]
When $n_{\ell}+1 \in A$, we get that
\begin{align*}
\sum_{ x \in A} u^{\ell}_{\{x\}}
&= u^{\ell}_{\{n_{\ell}+1\}} + \sum_{x \in A \setminus \{n_{\ell}+1\}} u^{\ell}_{\{x\}} \\
&= -\sum_{x \in [n_{\ell}]} \varepsilon_{\ell,x} 
+ \sum_{j= \ell+1}^m \sum_{k=1}^{n_j} a^j_{k,\ell} \varepsilon_{j,k}
+ \sum_{x \in A \setminus \{n_{\ell}+1\}} \varepsilon_{\ell,x} \\
&= - \sum_{x \in [n_{\ell}+1] \setminus A } \varepsilon_{\ell,x} + \sum_{j= \ell+1}^m \sum_{k=1}^{n_j} a^j_{k,\ell} \varepsilon_{j,k} \\
&= u^{\ell}_A. \qedhere
\end{align*}
\end{proof}

\begin{example}\label{example_afBT_char_vector}
Let $B_3$ be a generalized Bott tower of height $3$ 
as in Example~\ref{example_gBT_afBT} whose matrix $\Lambda$ 
is given by 
\[
\Lambda
= \begin{bmatrix}
-1 & 0 & 0 \\
-1 & 0 & 0 \\
a^2_{1,1} & -1 & 0 \\
a^{3}_{1,1} & a^3_{1,2} & -1 \\
a^3_{2,1} & a^3_{2,2} & -1
\end{bmatrix}.
\]
Let $F_3$ be the associated flag Bott manifold, and let $X$ be 
the closure of a generic orbit of the torus $(\C^{\ast})^5$. 
Then the fan $\widetilde{\Sigma}$ of $X$ has 14 rays. Consider the ray generator $u^1_{\{3\}}$. Then by Theorem~\ref{thm_main_thm_1},
the vector $u^1_{\{3\}}$ is 
\[
\sum_{x \in [3]\setminus \{3\}} -\ep_{1,x} 
+ \sum_{j=2}^3 \sum_{k=1}^{n_j} 
a^j_{k,1} \ep_{j,k} 
= -\ep_{1,1} -\ep_{1,2}
+ a^2_{1,1}\ep_{2,1} + a^3_{1,1} \ep_{3,1}
+a^3_{2,1} \ep_{3,2},
\]
where $\{\ep_{1,1},\ep_{1,2},\ep_{2,1},\ep_{3,1},\ep_{3,2}\}$ is
the standard basis of the Lie algebra of the compact torus
contained in $(\Cstar)^5$. With this standard basis, we have the following ray generators.
\[
\begin{array}{lll}
u^1_{\{1\}} = (1,0,0,0,0),&
u^1_{\{2\}} = (0,1,0,0,0),&
u^1_{\{3\}}= (-1,-1,a^2_{1,1},a^3_{1,1},a^3_{2,1}),\\
u^1_{\{1,2\}} = (1,1,0,0,0), &
u^1_{\{1,3\}} = (0,-1,a^2_{1,1},a^3_{1,1},a^3_{2,1}),&
u^1_{\{2,3\}} =(-1,0,a^2_{1,1},a^3_{1,1},a^3_{2,1}),\\
u^2_{\{1\}}= (0,0,1,0,0),& 
u^2_{\{2\}} = (0,0,-1,a^3_{1,2},a^3_{2,2}), &\\
u^3_{\{1\}} = (0,0,0,1,0), &
u^3_{\{2\}}= (0,0,0,0,1),&
u^3_{\{3\}} = (0,0,0,-1,-1),\\
u^3_{\{1,2\}}= (0,0,0,1,1), &
u^3_{\{1,3\}}=(0,0,0,0,-1), &
u^3_{\{2,3\}}=(0,0,0,-1,0).
\end{array}
\]
For a subset $\{1,3\} \subset [3]$, the ray generator $u^1_{\{1,3\}}$ is
$(0,-1,a^2_{1,1},a^3_{1,1},a^3_{2,1})$. Also, we have the following:
\[
\reqnomode
u^1_{\{1,3\}}
= (1,0,0,0,0) + (-1,-1,a^2_{1,1},a^3_{1,1},a^3_{2,1})
= u^1_{\{1\}} + u^1_{\{3\}}. \tag*{\qed}
\]
\end{example}

For a fan $\Sigma$ and a cone $\tau \in \Sigma$, 
we recall from~\cite[Definition 3.3.17]{CLS11} the definition of star subdivision $\Sigma^{\ast}(\tau)$ of $\Sigma$ along $\tau$.
Let $u_{\tau} = \sum_{\rho \in \tau(1)} u_{\rho}$, where $u_{\rho}$ is the ray generator of a ray $\rho$.
For each cone $\sigma \in \Sigma$ containing $\tau$, set
\[
\Sigma^*_{\sigma}(\tau) = \{\Cone(A) \mid A \subseteq \{u_{\tau}\} \cup \sigma(1), \tau(1) \nsubseteq A\}. 
\]
Then the \defi{star subdivision} $\Sigma^{\ast}(\tau)$ is defined to be
\[
\Sigma^*(\tau) = \{\sigma \in \Sigma \mid \tau \nsubseteq \sigma \} \cup \bigcup_{\tau \subseteq \sigma} \Sigma^*_{\sigma}(\tau).
\]
Hence the fan $\Sigma^{\ast}(\tau)$ has one more ray generated by the vector $u_{\tau}$.

Corollary~\ref{cor_char_vectors_are_obtained_by_sum} 
says that the set of ray generators
\[
\bigcup_{\ell=1}^m\big\{ u^{\ell}_{\{x\}} \mid x\in [n_{\ell}+1]\big\}
\]
can produce all other ray generators of the fan $\Sigma$, which yields the following property.
\begin{theorem}\label{thm_main_thm_2}
Let $B_m$ be the $m$-stage generalized Bott manifold 
determined by the integer matrix
$\Lambda$ as in~\eqref{eq_red_vector_char_ftn_of_gBT}, and
let $\Sigma'$ be the fan of $B_m$.
Let $F_m$ be the associated $m$-stage flag Bott manifold
to $B_m$. 
Then the fan~$\Sigma$ of the closure $X$ of a generic orbit of the canonical 
$\mathbf H$-action in the associated flag Bott manifold $F_m$ is the star subdivisions of $\Sigma'$ along the following cones
\[
\left\{\textup{Cone}\left(\{u^{\ell}_x \mid x \in A \}\right) \mid \emptyset \subsetneq A \subsetneq [n_{\ell}+1], 1 \leq \ell \leq m\right\} \subset \Sigma
\]
in the increasing order of $1 \leq \ell \leq m$ and in
the decreasing order of $|A|$.
\end{theorem}
\begin{example}
	Let $B_3$ and $F_3$ be generalized Bott manifold and its associated flag Bott manifold given in Example~\ref{example_afBT_char_vector}. 
	To obtain the fan $\Sigma$ of the closure $X$ of a generic torus orbit in $F_3$ from the fan $\Sigma'$ of $B_3$, we consider the star subdivisions of $\Sigma'$ along the following cones in the listed order:
	\[
	\begin{split}
	&\{\Cone(\{u^{1}_x \mid x \in A\}) \mid \emptyset \subsetneq A \subsetneq [3], |A| = 2\} = \{\Cone(u^1_{1}, u^1_2), \Cone(u^1_1, u^1_3), \Cone(u^1_2,u^1_3)\},  \\
	&\{\Cone(\{u^{1}_x \mid x \in A\}) \mid \emptyset \subsetneq A \subsetneq [3], |A| = 1\} =  \{\Cone(u^1_1), \Cone(u^1_2), \Cone(u^1_3)\}, \\
	&\{\Cone(\{u^{2}_x \mid x \in A\}) \mid \emptyset \subsetneq A \subsetneq [2], |A| = 1\} = \{\Cone(u^2_1), \Cone(u^2_2)\}, \\
	&\{\Cone(\{u^{3}_x \mid x \in A\}) \mid \emptyset \subsetneq A \subsetneq [3], |A| = 2\} 
	 = \{ \Cone(u^3_1, u^3_2), \Cone(u^3_1,u^3_2), \Cone(u^3_2,u^3_3)\},  \\
	&\{\Cone(\{u^{3}_x \mid x \in A\}) \mid \emptyset \subsetneq A \subsetneq [3], |A| = 1\} 
	= \{\Cone(u^3_1), \Cone(u^3_2), \Cone(u^3_3)\}.
	\end{split}
	\]
\end{example}

To give a proof of Theorem~\ref{thm_main_thm_2}, we first review the following classical result about a toric variety fibration over 
a toric variety. We refer to \cite[Proposition~7.3]{Oda78Torus}, as well as 
\cite[Chapter~3.3]{CLS11}, \cite[Chapter~VI.6]{Ewa96}. 

\begin{proposition}\label{lem_toric_var_fibration}
	Let $\Sigma$ and $\Sigma'$ be complete fans in $N_\RR:=N\otimes_\Z \mathbb{R}$ and $N'_\RR:=N'\otimes_\Z \mathbb{R}$ for some lattices $N$ and $N'$ respectively, which are compatible with a surjective 
	$\Z$-linear map $\bar \phi\colon N \to N'$. 
	Let $\Sigma''$ be a subfan of $\Sigma$ consisting of the cones 
	$\{\sigma \in \Sigma \mid \sigma \subset \ker \bar \phi_\RR \}$ and $X_{\Sigma''}$
	the corresponding toric variety.
	Then, the toric morphism $\phi\colon X_\Sigma \to X_{\Sigma'}$ 
	induced from $\bar\phi$ is an equivariant fiber bundle with fiber $X_{\Sigma''}$
	if and only if 
	\begin{enumerate}
		\item there exists a lifting $\widetilde\Sigma \subseteq \Sigma$ of $\Sigma'$ 
		such that $\bar\phi_{\RR} \colon N_\mathbb{R} \to N'_\RR$ 
		maps $\tilde{\sigma}\in \widetilde{\Sigma}$ bijectively to a cone $\sigma'\in \Sigma'$,
		\item $\Sigma$ consists of cones 
		$\{\tilde\sigma + \sigma'' \mid \tilde{\sigma}\in \widetilde{\Sigma},~ \sigma'' \in \Sigma'' \}$. 
	\end{enumerate}
\end{proposition}

The fan $\Sigma$ determined by
the condition of Proposition~\ref{lem_toric_var_fibration} is called 
the \emph{join} of $\widetilde\Sigma$ and $\Sigma''$ and denoted by $\Sigma=\widetilde\Sigma \bigcdot \Sigma''$. 
We refer to \cite[Chapter III.1, Chapter VI.6]{Ewa96}. 
We need one more result to give a proof of Theorem~\ref{thm_main_thm_2}.
\begin{lemma}\label{lemma_join_star}
Let $\Sigma_1$ and $\Sigma_2$ be fans such that $\Sigma_1(1) \cap \Sigma_2(1) = \emptyset$. Suppose that $\tau \in \Sigma_1$. Then 
\[
\Sigma_1^{\ast}(\tau) \bigcdot \Sigma_2 = (\Sigma_1 \bigcdot \Sigma_2)^{\ast}(\tau).
\]
Here we denote the cone $\tau + \{0\}$ in $\Sigma_1 \bigcdot \Sigma_2$ by $\tau$.
\end{lemma}
\begin{proof}
For a cone $\tau \in \Sigma_1$, we have that
\[
\begin{split}
 \Sigma_1^{\ast}(\tau) \bigcdot \Sigma_2 
&= (\{\sigma_1 \in \Sigma_1 \mid \tau \nsubseteq \sigma_1 \} \bigcdot \Sigma_2) \cup \bigcup_{\substack{\tau \subseteq \sigma_1, \\ \sigma_1 \in \Sigma_1}} ((\Sigma_1)_{\sigma_1}^{\ast}(\tau) \bigcdot \Sigma_2), \\
(\Sigma_1 \bigcdot \Sigma_2)^{\ast}(\tau + \{0\})
&= \{\sigma_1 + \sigma_2 \in \Sigma_1 + \Sigma_2 \mid \tau+\{0\} \nsubseteq \sigma_1 + \sigma_2 \}
\cup \bigcup_{\tau + \{0\} \subseteq \sigma_1 + \sigma_2} (\Sigma_1 \bigcdot \Sigma_2)^{\ast}_{\sigma_1 + \sigma_2} (\tau + \{0\}). 
\end{split}
\]
We note that by the definition of join of fans, we get
\[
\{\sigma_1 \in \Sigma_1 \mid \tau \nsubseteq \sigma_1 \} \bigcdot \Sigma_2 = 
\{\sigma_1 + \sigma_2 \in \Sigma_1 + \Sigma_2 \mid \tau + \{0\} \nsubseteq \sigma_1 + \sigma_2 \}.
\]
Moreover, we have 
\[
\bigcup_{\tau + \{0\} \subseteq \sigma_1 + \sigma_2} (\Sigma_1 \bigcdot \Sigma_2)^{\ast}_{\sigma_1 + \sigma_2} (\tau + \{0\})
= \bigcup_{\substack{\tau \subseteq \sigma_1, \\ \sigma_1 \in \Sigma_1}} \bigcup_{\sigma_2 \in \Sigma_2} (\Sigma_1 \bigcdot \Sigma_2)^{\ast}_{\sigma_1 + \sigma_2} (\tau + \{0\}).
\]

Therefore to prove the lemma, it is enough to show that for any $\sigma_1 \in \Sigma_1$ satisfying $\tau \subseteq \sigma_1$, the following equality holds:
\begin{equation}\label{eq_join_star}
(\Sigma_1)^{\ast}_{\sigma_1}(\tau) \bigcdot \Sigma_2=
\bigcup_{\sigma_2 \in \Sigma_2} (\Sigma_1 \bigcdot \Sigma_2)^{\ast}_{\sigma_1 + \sigma_2} (\tau+\{0\}). 
\end{equation}
We note that for $\sigma_2 \in \Sigma_2$, 
\begin{equation}\label{eq_join_star_2}
(\Sigma_1 \bigcdot \Sigma_2)^{\ast}_{\sigma_1+ \sigma_2} (\tau + \{0\})
= \{ \Cone(B) \mid B \subseteq \{u_{\tau}\} \cup \sigma_1(1) \cup \sigma_2(1), \tau(1) \nsubseteq B\}. 
\end{equation}

Suppose that $A \subseteq \{u_{\tau}\} \cup \sigma_1(1)$ satisfying $\tau(1) \nsubseteq A$. Then for a cone $\sigma_2 \in \Sigma_2$, $\Cone(A) + \sigma_2$ is an element in $(\Sigma_1)^{\ast}_{\sigma_1}(\tau) \bigcdot \Sigma_2$.
Since $\Cone(A) + \sigma_2 = \Cone(A \cup \sigma_2(1))$ and 
$\tau(1) \nsubseteq A \cup \sigma_2(1)$, the cone $\Cone(A) + \sigma_2$ is an element in 
$(\Sigma_1 \bigcdot \Sigma_2)^*_{\sigma_1 + \sigma_2}(\tau + \{0\})$ by~\eqref{eq_join_star_2}. 

Now, we consider $\Cone(B)$ in $(\Sigma_1 \bigcdot \Sigma_2)^{\ast}_{\sigma_1 + \sigma_2}(\tau+\{0\})$ for some $\sigma_2 \in \Sigma_2$. We set $A := B \cap (\{u_{\tau}\} \cup \sigma_1(1))$ and $B' := B \cap \sigma_2(1)$. Since $B \subseteq \{u_{\tau}\} \cup \sigma_1(1) \cup \sigma_2(1)$, we have $\Cone(B) = \Cone(A) + \Cone(B')$. Moreover, $\Cone(A) \in (\Sigma_1)^{\ast}_{\sigma_1}(\tau)$, and $\Cone(B') \in \Sigma_2$ since $\Cone(B')$ is a face of the cone $\Cone(B)$. Hence the equality~\eqref{eq_join_star} holds, and we have proven the lemma. 
\end{proof}

\begin{proof}[Proof of Theorem~\ref{thm_main_thm_2}]
By Proposition~\ref{lem_toric_var_fibration}, there exist liftings
$\widetilde{\Sigma}'_{n_1},\dots,\widetilde{\Sigma}'_{n_{m-1}}$ of the fans $\Sigma'_{n_1},\dots,\Sigma'_{n_{m-1}}$ of complex projective spaces such that
\[
\Sigma' = \widetilde{\Sigma}'_{n_1} \bigcdot \cdots \bigcdot \widetilde{\Sigma}'_{n_{m-1}} \bigcdot \Sigma'_{n_m}.
\]
More precisely, the lifting $\widetilde{\Sigma}'_{n_{\ell}} \subset \mathbb{R}^n$ consists of the cones
\[
\Cone(u^{\ell}_{1},\dots,\widehat{u}^{\ell}_{k_{\ell}},\dots,u^{\ell}_{n_{\ell}+1})
 \]
and their faces. On the other hand, the fan $\Sigma$ of the  closure of a generic orbit in the associated flag Bott manifold also can be written by
\[
\Sigma = \widetilde{\Sigma}_{n_1} \bigcdot \cdots \bigcdot \widetilde{\Sigma}_{n_{m-1}} \bigcdot \Sigma_{n_m},
\]
where $\widetilde{\Sigma}_{n_{\ell}}$ is a lifting of the fan $\Sigma_{n_{\ell}}$ of the permutohedral variety whose maximal cones are given by 
\[
\Cone(u^{\ell}_{A^{\ell}_1},\dots,u^{\ell}_{A^{\ell}_{n_{\ell}}})
\]
for a proper chain $\emptyset \subsetneq A^{\ell}_1 \subsetneq \cdots \subsetneq A^{\ell}_{n_{\ell}} \subsetneq [n_{\ell}+1]$ of subsets. 

By Lemma~\ref{lemma_join_star}, the operations join and star subdivision commute each other. Hence it is enough to show that the star subdivisions of the fan $\widetilde{\Sigma}_{n_{\ell}}'$ along the cones $\{\Cone(\{u^{\ell}_x \mid x \in A \}) \mid \emptyset \subsetneq A \subsetneq [n_{\ell}+1] \}$ in the decreasing order of dimensions of cones agrees with the fan $\widetilde{\Sigma}_{n_{\ell}}$. We note that the fan $\Sigma_n$ of the permutohedral variety can be obtained by star subdivisions of all the cones of dimension grater than $0$ of the fan $\Sigma_n'$ of $\C P^n$ in the decreasing order of dimensions of cones (see~Remark~\ref{rmk:permuto_and_simplex}). Moreover, 
for $1 \leq \ell \leq m$ and any nonempty proper subset $\emptyset \subsetneq 
\{x_1,\dots,x_d \} \subsetneq [n_{\ell}+1]$, 
the following equalities hold by Corollary~\ref{cor_char_vectors_are_obtained_by_sum}:
\[
u^{\ell}_{\{x_1, \dots, x_d\}}=
\sum_{i=1}^d
u^{\ell}_{\{x_i\}}
=
\sum_{i=1}^d u^{\ell}_{x_i}.
\]
Therefore the fan $\widetilde{\Sigma}_{n_{\ell}}$ is obtained from $\widetilde{\Sigma}_{n_{\ell}}'$ by star subdividing along the cones $\{\Cone(\{u^{\ell}_x \mid x \in A \}) \mid \emptyset \subsetneq A \subsetneq [n_{\ell}+1] \}$ in the given order, so the result follows. 
\end{proof}

\begin{remark}
In this paper, we concentrate on the closure of a generic torus orbit in the associated flag Bott manifold.
Since the matrices for the associated flag Bott manifolds can have nonzero entries only on the first column, there are flag Bott manifolds which are not the associated flag Bott manifolds. The second and the fourth authors compute the fan of the closure of a generic torus orbit in any flag Bott manifold in~\cite{LeSu18}. 
\end{remark}
\begin{remark}
	There are several studies on the closures of non-generic torus orbits. For instance, \cite{GS87combinatorial} studied torus orbit closures in homoneneous manifolds $G/P$ in terms of matroids, and, recently, \cite{LeeMasuda} and~\cite{LeeMasudaPark} study torus orbit closures associated to Schubert varieties and Richardson varieties, respectively.
\end{remark}
\subsection{Proof of Theorem~\ref{thm_main_thm_1}}
\label{sec_proof_of_main_thm}

For an $m$-stage flag Bott manifold $F_m$, 
consider the effective canonical $\bf H$-action.
Each fiber of a bundle $F_j \to F_{j-1}$ has the restricted
$(\Cstar)^{n_j}$-action, and its orbit closure of a generic point
is the permutohedral variety $X_{n_j}$.
Therefore the closure of a generic orbit of the torus
$\bf H$ in $F_m$ has the structure of iterated permutohedral 
variety bundles. Hence, the following lemma is straightforward from 
the successive application of  Proposition~\ref{lem_toric_var_fibration}. 

\begin{lemma}\label{lemma_orbit_of_X}
	Let $F_m$ be the associated $m$-stage flag Bott manifold and $X$ 
	the closure of a generic orbit of the torus $\mathbf H$ in $F_m$. 
	Let $\Sigma_{n_1}, \dots, \Sigma_{n_m}$ be fans of permutohedral
	varieties $X_{n_1}, \dots, X_{n_m}$, respectively. Then, there are 
	liftings $\widetilde{\Sigma}_{n_1}, \dots, \widetilde{\Sigma}_{n_{m-1}}$ 
	such that 
	$$\Sigma=\widetilde{\Sigma}_{n_1} \bigcdot \cdots \bigcdot \widetilde{\Sigma}_{n_{m-1}} \bigcdot \Sigma_{n_m}.$$
\end{lemma}

It remains to compute the primitive generators of rays in $\Sigma$. 
In general, a toric variety can be regarded as a GKM manifold with respect to the action of 
compact torus in the algebraic torus. 

\begin{remark}\label{rmk_relation_bet_fan_and_GKM}
	Two combinatoric objects, a smooth complete fan $\Sigma$ and 
	a GKM graph $(\Gamma, \alpha)$, of a toric variety are related by associating 
	maximal cones in $\Sigma$ with vertices of $\Gamma$, and cones of 
	codimension 1 in $\Sigma$ with edges of $\Gamma$. 
	In particular, if $\Sigma$ is an $n$-dimensional smooth fan, then an $n$-dimensional cone 
	$\sigma$ has $n$ facets, say $\tau_1, \dots, \tau_n$, which  correspond to the outgoing 
	edges, say $e_1, \dots, e_n$, in $\Gamma$ from the vertex corresponding to~$\sigma$. Let $\rho$ be a 
	$1$-dimensional cone in $\Sigma$, then $(n-1)$ many facets of $\sigma$ contains $\rho$ 
	except one facet. 
\end{remark}

Regarding $\Sigma$ be a fan in $\Lie(\mathbf T)$, 
the next Lemma \ref{lemma:char_ftn_and_axial_ftn} shows the relation between 
the ray generators of rays in $\Sigma$ and the axial function 
$\alpha\colon E(\Gamma) \to \mathfrak{t}_\Z^\ast$. 

\begin{lemma}\cite[Proposition 7.3.18]{BuPa15}\label{lemma:char_ftn_and_axial_ftn}
	Let $e_1, \dots, e_n$ and $\rho$ be as in Remark \ref{rmk_relation_bet_fan_and_GKM}, 
	and $u_\rho$ the ray generator 
	of $\rho$. Then	the following system of equations holds:
	\begin{equation}\label{eq_rel_bet_char_fun_and_axial_ftn}
	\langle \alpha(e_i), u_\rho \rangle = 
	\begin{cases}
	1 & \text{ if } i = 1,\\
	0 & \text{ if } 2 \leq i \leq n.
	\end{cases}
	\end{equation}
	In particular, given $\alpha(e_1), \dots, \alpha(e_n)$, the vector 
	$u_\rho$ is uniquely determined. \hfill \qed
\end{lemma}

Lemma~\ref{lemma:char_ftn_and_axial_ftn} says that the tangential representation 
at a fixed point determines the ray generator $u_\rho$ of a $1$-dimensional cone $\rho$ 
contained in  the a maximal dimensional cone $\sigma$ corresponding to the given fixed point. 
The next lemma shows that $u_\rho$ obtained in \eqref{eq_rel_bet_char_fun_and_axial_ftn} 
is independent from the choice of a maximal dimensional cone containing $\rho$.

\begin{lemma}\label{lemma:independence of choice of vertex}
	The primitive generator $u_\rho$ of an $1$-dimensional cone $\rho$ 
	obtained from equations \eqref{eq_rel_bet_char_fun_and_axial_ftn} 
	is well-defined, i.e., it is independent of the choice of a maximal dimensional 
	cone $\sigma$ containing $\rho$. 
\end{lemma}

\begin{proof}
	Suppose that $\sigma$ and $\sigma'$ are two maximal cones containing 
	$\rho$, whose facets are $\{\tau_i\mid 1 \leq i \leq n\}$ and  $\{\tau'_i\mid 1 \leq i \leq n\}$, 
	respectively. 
	Here, we may assume that $\sigma$ and $\sigma'$ are adjacent, 
	i.e., $\sigma$ and $\sigma'$ meet at a common facet, say  $\tau_n=\tau'_n$, 
	otherwise we choose a path of maximal cones connecting $\sigma$ and $\sigma'$, 
	and apply the same argument. 
	
	By the correspondence between cones in a smooth complete fan and a GKM graph 
	mentioned in Remark~\ref{rmk_relation_bet_fan_and_GKM}, we set up the following notation:
	\begin{enumerate}
		\item $\tau_1$ and $\tau'_1$: facets of $\sigma$ and $\sigma'$ which does not contain $\rho$, respectively;
		\item $e_1$ and $e_1'$:  edges in $\Gamma$ corresponding to $\tau_1$ and $\tau_1'$, respectively.
	\end{enumerate}
	We refer to Figure~\ref{fig_GKM-graph_near_two_vertices} for a 3-dimensional example. 
     \begin{figure}
	\begin{tikzpicture}
	\draw[very thick, ->] (0,0)--(3,3); \node[above] at (3,3) {\footnotesize$\rho$};
	
	\begin{scope}[scale=1.2]
	\draw[dashed, ->] (0,0)--(2.6, 2.3);
	\draw[dashed, ->] (0,0)--(1.5, 2.7);
	\end{scope}
	
	\node at (0.7, 1.4) {\footnotesize$\tau_2$};
	\draw[dotted, thick] (2.3, 1.2) [out=60, in=180]  to (4,2);\node[right] at (4,2) {\footnotesize$\tau_2'$};
	
	\draw[fill=blue, opacity=0.4] (0,0)--(1.8, 1.8)--(3,1)--cycle;
	\draw[dotted, thick] (3.5/3, 2.6/3) [out=180, in=120] to (1,-0.5); \node[below] at (1,-0.5) {\footnotesize$\tau_3=\tau_3'$};
	\draw[fill=yellow, opacity=0.5] (0,0)--(0.3, 2)--(1.8, 1.8)--(1.7, 0.8)--cycle;

	\draw (0.3,2)--(1.7, 0.8)--(3, 1)--(2.6, 2.3)--(1.5, 2.7)--cycle;
	
	\begin{scope}[scale=1.2]
	\draw[->] (0,0)--(0.3,2);
	\draw[->] (0,0)--(1.7, 0.8);
	\draw[->] (0,0)--(3, 1);
	\end{scope}
	
	\draw (1.8, 1.8)--(0.3,2);
	\draw (1.8, 1.8)--(1.7, 0.8);
	\draw (1.8, 1.8)--(3, 1);
	\draw (1.8, 1.8)--(2.6, 2.3);
	\draw (1.8, 1.8)--(1.5, 2.7);
	
	\node at (3.8/3, 4.6/3) {\footnotesize$\sigma$};
	\node at (2.1, 1.3) {\footnotesize$\sigma'$};
	\draw[dotted, thick] (4.7/3, 1.8/3) [out=330, in=180] to (4,0); \node[right] at (4,0) {\footnotesize$\tau_1'$};
	\node at (2/3, 2.8/3) {\footnotesize$\tau_1$};
	
	\begin{scope}[scale=1.7, xshift=120, yshift=-25]
	
	\draw (3.8/3, 4.6/3)--(2.1, 1.3)--(7.4/3, 5.1/3)--(5.9/3, 6.8/3)--(3.6/3, 6.5/3)--cycle;
	
	\draw[thick, ->] (3.8/3, 4.6/3)--(3.8/3 + 2.1/3-3.8/9, 4.6/3+ 1.3/3-4.6/9);
	\draw[thick, ->] (3.8/3, 4.6/3)--(3.8/3+3.6/9-3.8/9, 4.6/3+ 6.5/9-4.6/9);
	\draw[thick, ->] (3.8/3, 4.6/3)--(3.8/3+0.5/2-3.8/6, 4.6/3+1/2-4.6/6);
	
	\draw[thick, ->] (2.1, 1.3)--(2.1+ 3.8/9-2.1/3, 1.3+4.6/9-1.3/3);
	\draw[thick, ->] (2.1, 1.3)--(2.1+7.4/9-2.1/3, 1.3+5.1/9-1.3/3);
	\draw[thick, ->] (2.1, 1.3)--(2.1+2.5/2-2.1/2, 1.3+0.5/2-1.3/2);
	
	\node[rotate=-15] at (1.5, 1.65) {\footnotesize$e_3$};
	\node[rotate=-15] at (1.9, 1.55) {\footnotesize$e_3'$};
	\node at (1.1,1.8) {\footnotesize$e_2$};
	\node at (2.4,1.4) {\footnotesize$e_2'$};
	\node at (0.7,1.3) {\footnotesize$e_1$};
	\node at (2.1,0.8) {\footnotesize$e_1'$};
	
	\fill(3.8/3, 4.6/3) circle (1pt); \node[below] at (3.8/3, 4.6/3) {$v$};
	\fill(2.1, 1.3) circle (1pt); \node[below] at (2, 1.3) {$v'$};
	\fill(7.4/3, 5.1/3) circle (1pt);
	\fill(5.9/3, 6.8/3) circle (1pt);
	\fill(3.6/3, 6.5/3) circle (1pt);
	
	\draw (3.8/3, 4.6/3)--(0.5,1);
	\draw (2.1, 1.3)--(2.5,0.5);
	\draw (7.4/3, 5.1/3)--(3,1.5);
	\draw (5.9/3, 6.8/3)--(2.5,2.8);
	\draw (3.6/3, 6.5/3)--(0.5,2.5);
	\end{scope}
	
	\end{tikzpicture}
	\caption{A $3$-dimensional fan and corresponding GKM graph. }
	\label{fig_GKM-graph_near_two_vertices}
\end{figure}
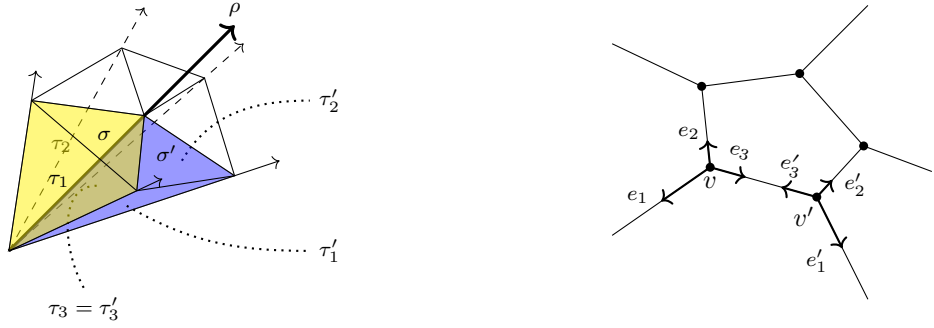
	
	Now, it is enough to show that $u_\rho$ satisfies the following relations:
	\begin{equation}\label{eq_linear_equations_arond_another_vert}
	\langle \alpha(e_i'), u_\rho \rangle =
	\begin{cases}
	1 & \text{ if } i = 1,\\
	0 & \text{ if } 2 \leq i \leq n
	\end{cases}
	\end{equation}
	For the given GKM graph $(\Gamma, \alpha)$ and 
	the connection $\theta=\{\theta_e \mid e\in E(\Gamma)\}$, consider 
	\[
	\theta_{e_n} \colon \{e_1, \dots, e_n\} \to \{e'_1, \dots, e'_n\}.
	\]
	Since the closure $\overline{O(\rho)}$ of the orbit $O(\rho)$ is a 
	toric subvariety of $X_\Sigma$, the subgraph by taking vertices corresponding to 
	maximal cones containing $\rho$ is indeed a GKM-subgraph,   
	whose connection is inherited from the original one $\theta$. Therefore
	$\theta_{e_n}$ 
	maps $\{e_2, \dots, e_n\}$ bijectively to $\{e_2', \dots, e_n'\}$.
	Hence we have that $\theta_{e_n} (e_1) =e_1'$. 
	For convenience, we assume that $\theta_{e_n}(e_i)=e_i'$ for $i=1, \dots, n$. 
	
	For $1 \leq i \leq n$, we have the relation
	\[
	\alpha(e'_i)=\alpha(e_i) + c_i \alpha(e_n),
	\]
	for some $c_i \in \Z$. 
	Hence the equations \eqref{eq_rel_bet_char_fun_and_axial_ftn} become
	\[
	\langle \alpha(e_i') - c_i \alpha(e_n), u_\rho \rangle = 
	\begin{cases}
	1 & \text{ if } i =1,\\
	0 & \textrm{ if } 2 \leq i \leq n
	\end{cases}
	\]
	which turn out to be 
	the relations \eqref{eq_linear_equations_arond_another_vert}, because 
	$\langle \alpha(e_n), u_\rho \rangle=0$. Hence the result follows. 
\end{proof}

Now we give a proof of Theorem~\ref{thm_main_thm_1}. 
By Lemma~\ref{lemma_orbit_of_X}, we know that the combinatorial structure of the fan $\Sigma$ is given as in~Theorem~\ref{thm_main_thm_1}(2). Now it is enough to show that the ray generators are given as in~Theorem~\ref{thm_main_thm_1}(1).

For a given $1 \leq \ell \leq m$ and a nonempty proper subset $A$ of
$[n_{\ell}+1]$, consider a ray $\rho^{\ell}(A)$ of $\Sigma$. 
To compute the ray generator of $\rho^{\ell}(A)$,
it is enough to consider only one maximal cone containing $\rho^{\ell}(A)$ because of 
Lemma~\ref{lemma:independence of choice of vertex}. 

We note that there is one-to-one correspondence between the set of maximal cones in $\widetilde{\Sigma}$ and $\prod_{j=1}^m \mathfrak{S}_{n_j+1}$ as in~\ref{eq_chain_Ak}.
More precisely, for $(v_1,\dots,v_m) \in \prod_{j=1}^m \mathfrak{S}_{n_j+1}$, we define
\begin{equation}\label{eq_chain_Akp}
A^{\ell}_p := \{v(n_{\ell}+2-p),\dots,v(n_{\ell}+1)\} \quad
\text{ for } 1 \leq p \leq n_{\ell}, 1 \leq \ell \leq m.
\end{equation}
Moreover, for a given maximal cone indexed by $(v_1,\dots,v_m)$, the adjacent maximal cones $\sigma^j_i$ are determined by permutations
\begin{equation}\label{eq_permutations_adjacent}
(v_1,\dots,v_{j-1}, v_j \cdot s_i, v_{j+1},\dots,v_m)
\end{equation}
for $1 \leq i \leq n_j$ and $1  \leq j \leq m$.

From now on, 
set $A = \{ x_1 < x_2 < \cdots < x_{n_{\ell}+1-d} \}$ and
$[n_{\ell}+1] \setminus A = \{y_1<y_2<\cdots<y_d\}$. 
Define a permutation $v_{\ell,A}$ to be 
\begin{equation}\label{eq_def_of_v_ell}
v_{\ell,A} = (y_1 \ y_2  \cdots y_d \ x_1 \ x_2 \cdots x_{n_{\ell}+1-d})
\in \mathfrak{S}_{n_{\ell}+1}.
\end{equation}
Also define $\mathbf v := (v_1,\dots,v_{\ell},\dots,v_m) \in \prod_{j=1}^m \mathfrak{S}_{n_j+1}$ 
by setting $v_{\ell} = v_{\ell,A}$ and 
$v_j = e \in \mathfrak{S}_{n_j+1}$ for~$j \neq \ell$.
Then using~\eqref{eq_chain_Akp}, the maximal cone $\sigma_{\bf v}$ indexed by $\mathbf v$ contains the ray $\rho^{\ell}_A$.
We note that among adjacent maximal cones indexed by permutations in~\eqref{eq_permutations_adjacent}, the maximal cone $\sigma^{\ell}_d$ is the unique maximal cone which does not contain the ray $\rho^{\ell}_A$, because 
\begin{equation*}
	v_{\ell} \cdot s_d = 
	v_{\ell,A}(d,d+1)=(y_1 \ \cdots \ y_{d-1} \  x_1 \ y_d \ x_2 \ \cdots \ x_{n_{\ell}+1-d}).
	\end{equation*}

Because of Lemmas~\ref{lemma:char_ftn_and_axial_ftn}
and~\ref{lemma:independence of choice of vertex},
it is enough to show that the vector 
\[
u^{\ell}_A =	\begin{cases}
\mathlarger\sum_{x \in A} \ep_{\ell,x} & \text{ if }	
n_{\ell}+1 \notin A, \\
\mathlarger\sum_{x \in [n_{\ell}+1]\setminus A} -\ep_{\ell,x}
+\mathlarger\sum_{j=\ell+1}^{m} \mathlarger\sum_{k=1}^{n_j} a^j_{k,\ell} \ep_{j,k}
& \text{ otherwise}
\end{cases}
\]
in Theorem~\ref{thm_main_thm_1}
satisfies the following equations:
\[
\langle \alpha(e^j_i), u^{\ell}_A \rangle =
\begin{cases}
1& \text{ if } j = \ell \text{ and }i = d, \\
0& \text{ otherwise},
\end{cases}
\]
where $e^j_i$ is an edge of the GKM graph $\Gamma$ of $X$ corresponding to the facet $\sigma_{\bf v} \cap \sigma^j_i$ of the maximal cone $\sigma_{\bf v}$,  and $\alpha$ is the axial function $\alpha \colon E(\Gamma) \to \mathfrak{t}_{\mathbb{Z}}^{\ast}$.

To prove the claim, we separate cases as $j < \ell$,
$j = \ell$, and $j > \ell$. 
\begin{enumerate}
\item[\textbf{Case 1}] {$j < \ell$.}
By Theorem~\ref{thm_GKM_of_Fm},
the axial functions
of the edge $\alpha(e^j_i)$ is a linear combination of 
{
	\def\OldComma{,}
	\catcode`\,=13
	\def,{%
		\ifmmode%
		\OldComma\discretionary{}{}{}%
		\else%
		\OldComma%
		\fi%
	}%
$\ep_{1,1}^{\ast},\dots,\ep_{1,n_1}^{\ast},\dots,
\ep_{j,1}^{\ast}, \dots, \ep_{j,n_j}^{\ast}$.
}
On the other hand,
since $u^{\ell}_A$ is a linear combination of 
{
	\def\OldComma{,}
	\catcode`\,=13
	\def,{%
		\ifmmode%
		\OldComma\discretionary{}{}{}%
		\else%
		\OldComma%
		\fi%
	}%
$\ep_{\ell,1},\dots,
\ep_{\ell,n_{\ell}},\dots,\ep_{m,1},\dots,\ep_{m, n_m}$} and $j < \ell$,
their pairings always vanish.

\item[\textbf{Case 2}] $j = \ell$.
By Theorem~\ref{thm_GKM_of_Fm},
the axial functions of the edge $\alpha(e^{\ell}_i)$ is 
a linear combination of 
{
	\def\OldComma{,}
	\catcode`\,=13
	\def,{%
		\ifmmode%
		\OldComma\discretionary{}{}{}%
		\else%
		\OldComma%
		\fi%
	}%
$\ep_{1,1}^{\ast},
\dots,\ep_{1,n_1}^{\ast},\dots,\ep_{\ell,1}^{\ast}, \dots, 
\ep_{\ell,n_{\ell}}^{\ast}$}. More precisely, we have that 
\[
\alpha(e^{\ell}_i) = (\ep_{\ell, v_{\ell,A}(i+1)})^{\ast}-
(\ep_{\ell,v_{\ell,A}(i)})^{\ast} + \text{ other terms},
\]
where `other terms' are the terms of 
$\ep_{p,k}^{\ast}$ for $p < \ell$ and
$v_{\ell,A}$ is a permutation defined in \eqref{eq_def_of_v_ell}.
Since the vector $u^{\ell}_A$ is a linear combination of 
$\ep_{\ell,1},\dots,\ep_{\ell,n_{\ell}},\dots,
\ep_{m,1},\dots,\ep_{m,n_m}$, we have
\begin{equation}\label{eq_alpha_and_lambda}
\langle
\alpha(e^{\ell}_i), u^{\ell}_A
\rangle
= \langle
(\ep_{\ell, v_{\ell,A}(i+1)})^{\ast} - (\ep_{\ell,
	v_{\ell,A}(i)})^{\ast}, u^{\ell}_A
\rangle.
\end{equation}
Because of the definition of the permutation $v_{\ell,A}$, we have that 
$v_{\ell,A}(i) \in A $ if and only if $i \geq d+1$. 
Therefore for the case when $n_{\ell}+1 \notin A$, we have that
the value $\langle (\ep_{\ell, v_{\ell,A}(i)})^{\ast}, u^{\ell}_A \rangle$ 
equals to 
$0$ if $i \leq d$, and $1$ otherwise.  
Also for the case when $n_{\ell}+1 \in A$, we get that
the pairing 
$\langle (\ep_{\ell, v_{\ell,A}(i)})^{\ast}, u^{\ell}_A \rangle$
is $-1$ if $i \leq d$ and $0$ otherwise. 

By applying \eqref{eq_alpha_and_lambda} for 
$n_{\ell} +1 \notin A$, we have the following:
\[
\langle \alpha(e^{\ell}_i), u^{\ell}_A
\rangle = \begin{cases}
0-0 =0 & \text{ for } 1 \leq i < d, \\
1-0 = 1 & \text{ for } i = d,\\
1-1=0& \text{ for }  d < i \leq n_{\ell}.
\end{cases}
\]
Similarly, when $n_{\ell}+1 \in A$, we get 
the following:
\[
\langle \alpha(e^{\ell}_i),u^{\ell}_A\rangle
= \begin{cases}
-1-(-1) = 0 &\text{ for } 1 \leq i < d, \\
0 - (-1) = 1 & \text{ for } i = d, \\
0-0=0 & \text{ for } d < i \leq n_{\ell}.
\end{cases}
\]
\item[\textbf{Case 3}] $j > \ell$. 
The matrix $\X{\ell}{j}$ in Proposition~\ref{prop_fBT_is_GKM} is 
\[
\X{j}{\ell}=
\sum_{\ell < i_1 < \cdots < i_r < j}
\left( B_j \A{j}{i_r} \right)
\left( B_{i_r} \A{i_r}{i_{r-1}} \right)
\cdots
\left( B_{i_1} \A{i_1}{\ell} \right)
B_{\ell} + B_j \A{j}{\ell} B_{\ell}.
\]
Since $v_j = e$ for $j \neq \ell$, the matrix $\X{j}{\ell}$ can
be written by
\[
\X{j}{\ell} =
\left(\sum_{\ell < i_1 < \cdots < i_r < j}
\A{j}{i_r} 
\A{i_r}{i_{r-1}}
\cdots
\A{i_1}{\ell} + \A{j}{\ell}\right)
B_{\ell}.
\]
By Proposition~\ref{prop:lambda_of_asso_gBT},
the matrix $\A{j}{i}$ has nonzero entries only
on the first column. The matrix $B_{\ell}$ is the row permutation
matrix corresponding to $v_{\ell,A}$, so that
$B_{\ell}$ is the column permutation matrix
corresponding to $v_{\ell,A}^{-1}$. 
Hence by multiplying the matrix $B_{\ell}$ on the right, 
the matrix $\X{j}{\ell}$ 
has nonzero entries only on the $y_1$th column. 
\begin{enumerate}
\item[\textbf{Subcase 1}] $n_{\ell}+1 \notin A$.
Since the matrix $\X{j}{\ell}$ 
has nonzero entries only on the $y_1$th column, 
we have that $\langle \alpha(e^j_i), u^{\ell}_A \rangle = 0$
for all $j > \ell$.  

\item[\textbf{Subcase 2}] $n_{\ell}+1 \in A$.
For a pair $(p, j)$ such that $\ell < p < j \leq m$,
the matrix $\X{j}{p}$ has nonzero entries only on the
first column. 
For simplicity, for $\ell < p < j$,
denote the $(i,1)$-entry of $\X{j}{p}$ by $x^{(j)}_{p,i}$.
Similarly, denote the $(i, y_1)$-entry of $\X{j}{\ell}$ by
$x^{(j)}_{\ell,i}$.
Then we have the following:
\begin{align*}
\langle
\alpha(e^j_i), u^{\ell}_A
\rangle
&=  \left \langle 
(x^{(j)}_{\ell,i+1} - x^{(j)}_{\ell,i})(\ep_{\ell, y_1})^{\ast}
+ \sum_{p = \ell+1}^{j-1}(x^{(j)}_{p,i+1} - x^{(j)}_{p,i})(\ep_{p,1})^{\ast}
+ (\ep_{j, i+1})^{\ast} - (\ep_{j, i})^{\ast}, u^{\ell}_A
\right \rangle \\
&= \left \langle 
(x^{(j)}_{\ell,i+1} - x^{(j)}_{\ell,i})(\ep_{\ell, y_1})^{\ast},
-(\ep_{\ell, y_1} + \cdots + \ep_{\ell, y_d}) 
\right \rangle \\
& \quad \quad \quad + \left \langle
\sum_{p = \ell+1}^{j-1}(x^{(j)}_{p,i+1} - x^{(j)}_{p,i})(\ep_{p,1})^{\ast}
+ (\ep_{j, i+1})^{\ast} - (\ep_{j,i})^{\ast}, 
\sum_{p=\ell+1}^m \sum_{k=1}^{n_p} a^p_{k,\ell} \ep_{p,k}
\right \rangle \\
&= (-1)(x^{(j)}_{\ell,i+1} - x^{(j)}_{\ell,i})
+ \sum_{p=\ell+1}^{j-1}
(x^{(j)}_{p,i+1} - x^{(j)}_{p,i})(a^p_{1,\ell})
+ (a^j_{i+1,\ell} - a^{j}_{i,\ell}).
\end{align*}
To show the above pairing vanishes, it is enough to show that
\[
x^{(j)}_{\ell,i} = 
\sum_{p=\ell+1}^{j-1}
x^{(j)}_{p,i}a^p_{1,\ell} + a^j_{i,\ell} \quad
\text{ for all } i,
\]
which comes from the definition of $\X{j}{\ell}$:
\begin{align*}
\X{j}{\ell}B_{\ell}^{-1} &=
\sum_{\ell < i_1 < \cdots < i_r < j}
\A{j}{i_r} 
\A{i_r}{i_{r-1}}
\cdots
\A{i_1}{\ell} + \A{j}{\ell}\\
&= 
\X{j}{j-1}\A{j-1}{\ell}
+ \cdots 
+ \X{j}{\ell+2} \A{\ell+2}{\ell} 
+ \X{j}{\ell+1} \A{\ell+1}{\ell}
+ \A{j}{\ell} \\
&= \sum_{p=\ell+1}^{j-1} \X{j}{p} \A{p}{\ell}
+ \A{j}{\ell}.
\end{align*}
Hence we have  $\langle \alpha(e^j_i), u^{\ell}_A \rangle = 0$
for all $j > \ell$. 
\end{enumerate}
\end{enumerate}

	Now we prove the smoothness. 
	Since the permutohedral variety $X_n$ is nonsingular (see~\cite[Corollary of Theorem 3.3]{Dabr96}), for a proper chain $\emptyset \subsetneq A_1 \subsetneq \cdots \subsetneq A_n \subsetneq [n+1]$ of nonempty proper subsets of $[n+1]$, we have that
	\begin{equation}\label{eq_perm_is_smooth}
	\det[ u_{A_1} \ u_{A_2} \ \cdots  \ u_{A_n}] = \pm 1. 
	\end{equation}
	To show that a generic torus orbit closure is smooth, it is enough to show that every maximal cone in $\Sigma$ is smooth. For a maximal cone indexed by $(A^1_{\bullet},\dots,A^m_{\bullet})$, consider the matrix whose column vectors are the corresponding ray generators:
	\begin{equation}\label{eq_smoothness}
	\left[ u^1_{A^1_1} \ \cdots \ u^1_{A^1_{n_1}} \ \cdots \ u^m_{A^m_1} \ \cdots \ u^m_{A^m_{n_m}} \right].
	\end{equation}
	Then the matrix~\eqref{eq_smoothness} is a block lower triangular matrix whose sizes of blocks are $n_1,\dots,n_m$.
	Moreover, the determinant of the matrix in~\eqref{eq_smoothness} is 
	\[
	\det \left( \left[ u_{A^1_1} \ \cdots \ u_{A^1_{n_1}} \right]\right) \cdot
	\det \left( \left[ u_{A^2_1} \ \cdots \ u_{A^2_{n_2}} \right]\right) \cdots
	\det \left( \left[ u_{A^m_1} \ \cdots \ u_{A^m_{n_m}} \right]\right) = \pm 1
	\]
	by~\eqref{eq_perm_is_smooth}. Here $\{u_{A^{\ell}_1},\dots,u_{A^{\ell}_{n_{\ell}}}\}$ is the set of ray generators of the maximal cone in the fan of $X_{n_{\ell}}$ indexed by the proper chain $\emptyset \subsetneq A^{\ell}_1 \subsetneq \cdots \subsetneq A^{\ell}_{n_{\ell}} \subsetneq [n_{\ell}+1]$ for $1 \leq \ell \leq m$. This proves that the closure of a generic torus orbit in the associated flag Bott manifold is smooth. 

%
%

\providecommand{\bysame}{\leavevmode\hbox to3em{\hrulefill}\thinspace}
\providecommand{\MR}{\relax\ifhmode\unskip\space\fi MR }
\providecommand{\MRhref}[2]{%
	\href{http://www.ams.org/mathscinet-getitem?mr=#1}{#2}
}
\providecommand{\href}[2]{#2}

\end{document}